\documentclass[11pt,reqno]{amsart}

\usepackage{amsmath,amssymb,amsfonts,mathtools}
\usepackage[shortlabels]{enumitem}
\usepackage{microtype}
\usepackage{hyperref}

\hypersetup{
  colorlinks=true,
  linkcolor=blue,
  citecolor=blue,
  urlcolor=blue
}

\numberwithin{equation}{section}

\theoremstyle{plain}
\newtheorem{theorem}{Theorem}[section]
\newtheorem{proposition}[theorem]{Proposition}
\newtheorem{lemma}[theorem]{Lemma}
\newtheorem{corollary}[theorem]{Corollary}

\theoremstyle{definition}
\newtheorem{definition}[theorem]{Definition}
\newtheorem{example}[theorem]{Example}
\theoremstyle{remark}
\newtheorem{remark}[theorem]{Remark}
\newcommand{\Z}{\mathbb Z}
\newcommand{\Q}{\mathbb Q}
\newcommand{\Rr}{\mathbb R}
\newcommand{\C}{\mathbb C}
\newcommand{\F}{\mathbb F}

\DeclareMathOperator{\Disc}{Disc}
\DeclareMathOperator{\Ext}{Ext}
\DeclareMathOperator{\Th}{Th}
\newcommand{\QCoh}{\ensuremath{\mathrm{QCoh}}}
\newcommand{\uHom}{\ensuremath{\underline{\Hom}}}
\DeclareMathOperator{\Hom}{Hom}
\DeclareMathOperator{\lk}{lk}
\newcommand{\Mod}{\ensuremath{\mathrm{Mod}}}

\newcommand{\Ho}{\ensuremath{\mathrm{Ho}}}

\DeclareMathOperator{\coker}{coker}
\DeclareMathOperator{\Tor}{Tor}
\DeclareMathOperator{\rk}{rk}

\DeclareMathOperator{\diag}{diag}
\DeclareMathOperator{\GL}{GL}
\DeclareMathOperator{\Mat}{Mat}
\DeclareMathOperator{\ad}{ad}
\DeclareMathOperator{\ch}{char}

\newcommand{\cofib}{\mathrm{cofib}}

\newcommand{\doi}[1]{\href{https://doi.org/#1}{doi:#1}}

\newcommand{\TMF}{\ensuremath{\mathrm{TMF}}}
\newcommand{\CP}{\ensuremath{\mathbb{CP}}}
\newcommand{\Sq}{\ensuremath{\mathrm{Sq}}}
\newcommand{\Canon}{\ensuremath{\mathrm{Canon}}}
\newcommand{\Tok}{\ensuremath{\mathrm{Tok}}}
\newcommand{\R}{\ensuremath{\mathcal{R}}}
\newcommand{\cZ}{\ensuremath{\mathcal{Z}}}

\newcommand{\arxiv}[1]{\href{https://arxiv.org/abs/#1}{arXiv:#1}}

\newcommand{\GS}{\mathrm{GS}}

\title[Canonical torsion linking pairings and \TMF\ state spaces]{Canonical torsion linking pairings and explicit \TMF\ state spaces of closed $3$--manifolds}

\author{Ruiliang Li}
\address{Tsinghua University, Beijing, China}
\email{lirl23@mails.tsinghua.edu.cn}


\date{January 9, 2026}

\begin{document}

\begin{abstract}
We study the \TMF-valued \((3+1)\)--dimensional TQFT of Gukov--Krushkal--Meier--Pei and give an explicit description of the \TMF-module state space \(\cZ_{\TMF}(Y)\) associated to a closed oriented \(3\)--manifold \(Y\) from an integral surgery presentation.

Given a symmetric integral surgery matrix \(A\) for \(Y=\partial W(A)\), we construct a canonical invariant \(\Canon(A)\) encoding \(b_1(Y)\) together with the isometry class of the torsion linking pairing on \(\Tor H_1(Y;\Z)\).
This refines the classical classification of torsion linking pairings by selecting canonical representatives compatible with Kirby stabilization, and we prove that \(\Canon(A)\) is invariant under Kirby moves and complete.
Passing from \(\Canon(A)\) to the associated canonical token package \(T=\Tok(A)\), we assemble a canonical symmetric integral matrix \(B(T)\) and define a realization functor \(\mathcal R:\Tok\to \Ho(\Mod_{\TMF})\) by
\[
\mathcal R(T)\;:=\;L_{B(T)}\bigl[\,3b_+(B(T))-2b_-(B(T))\,\bigr].
\]
Our main theorem identifies the GKMP state space with this explicit realization,
\[
\cZ_{\TMF}(Y)\ \simeq\ \mathcal R\bigl(\Tok(A)\bigr),
\]
yielding a prime-by-prime description in terms of finitely many local building blocks.

We further compute the closed \(4\)--manifold class \(Z_{\TMF}(\CP^2)=\pm\,\nu\in\pi_3\TMF\) and, assuming the naturality statement for reversing simply connected cobordisms formulated by Gukov--Krushkal--Meier--Pei, we identify \(Z_{\TMF}(S^2\times S^2)=\eta\in\pi_1\TMF\).
Finally, we establish a rank-one time-reversal duality \(L_{(-n)}\simeq L_{(n)}^\vee[1]\) for every integer \(n\).
\end{abstract}

\maketitle

\section{Introduction}\label{sec:intro}

Topological quantum field theories (TQFTs), in the axiomatic sense of Atiyah~\cite{Atiyah88},
assign algebraic objects to manifolds and functorial maps to cobordisms.
In practice, the most basic question is computational: given a closed oriented $3$--manifold $Y$,
how does one determine the state space assigned to $Y$ by a fixed $(3+1)$--dimensional TQFT?

This paper addresses that question for the $\TMF$--valued TQFT constructed by
Gukov--Krushkal--Meier--Pei~\cite{GKMP25}.
Their theory assigns to a closed oriented $3$--manifold $Y$ a (graded) object
\[
\cZ_{\TMF}(Y)\ \in\ \Ho(\Mod_{\TMF}),
\]
and to a compact oriented cobordism $W:Y_0\to Y_1$ a morphism
$\cZ_{\TMF}(W):\cZ_{\TMF}(Y_0)\to \cZ_{\TMF}(Y_1)$.
A central feature of the construction is that when $Y=\partial X$ bounds a simply connected
$4$--manifold $X$, the state space is expressed in terms of the integral intersection form of $X$
through the GKMP $\TMF$--modules $L_b$ and a normalization shift.
For a general $Y$ presented by surgery, however, the resulting description depends on auxiliary choices;
one must understand how the $\TMF$--module changes under Kirby moves~\cite{Kirby78}.
The aim of this paper is to replace that dependence on choices by a canonical set of algebraic invariants,
and to give an explicit formula for $\cZ_{\TMF}(Y)$ in terms of those invariants.

\subsection{Surgery matrices, linking forms, and the canonicalization problem}
Fix an ordered framed link $L\subset S^3$ with integer framings, and let
$A\in M_n(\Z)$ be its symmetric linking matrix.
Let $W(A)$ be the associated $2$--handlebody and set
\[
Y_A := \partial W(A).
\]
The group $H_1(Y_A;\Z)$ is naturally identified with the cokernel
$G(A):=\coker(A:\Z^n\to\Z^n)$, and the torsion subgroup $\Tor(G(A))$
carries a nonsingular linking pairing
\[
\lambda_A:\Tor(G(A))\times\Tor(G(A))\longrightarrow \Q/\Z,
\]
canonically determined up to isometry.
(We recall the precise conventions, including the singular case and orientation sign, in \S2.)
Kirby moves on the surgery link induce two algebraic operations on the matrix $A$:
unimodular congruence $A\mapsto P^{\mathsf T}AP$ and $\pm 1$ stabilization
$A\mapsto A\oplus(\pm 1)$~\cite{Kirby78}.
Both operations preserve the isometry class of $(\Tor(G(A)),\lambda_A)$, and hence preserve the
pairing $\lambda_{Y_A}$ on $\Tor H_1(Y_A;\Z)$.

The classical classification of linking pairings (see Wall~\cite{Wall63} and
Kawauchi--Kojima~\cite{KawauchiKojima80}) gives a complete set of isometry invariants for $(\Tor(G(A)),\lambda_A)$, prime by prime.
What is \emph{not} provided by the classical literature is a \emph{canonical choice} of normal form compatible with the stabilization operations that appear in the GKMP theory.
For the purposes of an explicit $\TMF$--module formula, one needs a deterministic procedure which,
given $A$, outputs a canonical representative of the isometry class of $(\Tor(G(A)),\lambda_A)$
(and keeps track of the free rank $b_1(A)$ when $A$ is singular),
and which is stable under Kirby moves.

\subsection{Main results}
Our first result is a canonicalization of torsion linking pairings tailored to the GKMP setting.
Concretely, we associate to each symmetric integral matrix $A$ a canonical invariant
\[
\Canon(A)\ \in\ \Tok,
\]
where $\Tok$ is a commutative monoid of finite prime--indexed packages encoding, in a canonical way,
the isometry class of $(\Tor(G(A)),\lambda_A)$ together with the integer $b_1(A)$.
The point is not the encoding itself (which can be serialized for computation), but rather that
$\Canon(A)$ is a \emph{mathematical invariant} and behaves functorially under orthogonal sums and
Kirby stabilization.

\medskip\noindent
\textbf{Theorem A (Canonical invariants for linking pairings).}
\emph{The assignment $A\mapsto \Canon(A)$ depends only on the isometry class of the pair
$\bigl(\Tor(G(A)),\lambda_A\bigr)$ and on $b_1(A)$.
Moreover, $\Canon(A)=\Canon(A')$ if and only if there exists an isometry
$\Tor(G(A))\cong\Tor(G(A'))$ identifying $\lambda_A$ with $\lambda_{A'}$ and $b_1(A)=b_1(A')$.
In particular, $\Canon(A)$ is invariant under unimodular congruence and under $\pm1$ stabilization,
hence is an invariant of the Kirby stable congruence class of the surgery presentation, and therefore
an invariant of the resulting $3$--manifold $Y_A$.}

\medskip
The construction is carried out in \S3, with local classification and canonical choices for odd primes
and for the $2$--primary subtleties developed in \S\S4--5.
Throughout, we emphasize that the role of $\Canon(A)$ is to provide a \emph{canonical normal form}
rather than to reprove the known classification.

Our second result is an explicit dictionary between these canonical invariants and the GKMP state space.
To state it, recall that GKMP attach to an integral symmetric bilinear form $b$
a $\TMF$--module $L_b$, and the state space of a simply connected bounding $4$--manifold $X$
is given by the normalized shift
$L_{b(X)}[\,3b_+(X)-2b_-(X)\,]$.
We construct, for each token package $T\in\Tok$, a canonical symmetric integral matrix $B(T)$ realizing
the corresponding torsion linking form, and define the realization functor
\[
\mathcal R(T)\ :=\ L_{B(T)}\bigl[\,3b_+(B(T))-2b_-(B(T))\,\bigr]\ \in\ \Ho(\Mod_{\TMF}),
\]
as in \S\ref{sec:tqft-dictionary}.
The principal theorem identifies the GKMP state space with this explicit realization.

\medskip\noindent
\textbf{Theorem B (Explicit state space formula; token realization).}
\emph{For every symmetric integral matrix $A$, there exists an equivalence}
\[
\cZ_{\TMF}(Y_A)\ \simeq\ \mathcal R\bigl(\Canon(A)\bigr)
\qquad\text{in }\Ho(\Mod_{\TMF}).
\]
\emph{Equivalently, $\cZ_{\TMF}(Y)$ depends only on $(b_1(Y),\Tor H_1(Y;\Z),\lambda_Y)$ and is computed
by the explicit $\TMF$--module associated to the canonical invariants $\Canon(A)$.}
\emph{This is proved as Theorem~\ref{thm:token-realization-agrees}.}

\begin{remark}[Not a classification of $3$--manifolds]\label{rem:not-3mfld-classification}
A surgery matrix does not determine the diffeomorphism type of the underlying $3$--manifold,
and the torsion linking form (even together with $b_1$) is far from classifying $3$--manifolds.
Theorem~B asserts that the GKMP state space factors through this coarse algebraic datum;
consequently many non-diffeomorphic $3$--manifolds receive equivalent \TMF-state spaces.
\end{remark}

\begin{remark}[Equivalences versus canonical equivalences]\label{rem:equiv-vs-canonical}
Gukov--Krushkal--Meier--Pei show that the state space \(\cZ_{\TMF}(Y)\) is well-defined up to equivalence in \(\Ho(\Mod_{\TMF})\), but in general not up to a canonical equivalence; see \cite[\S7.2]{GKMP25}.
Our token package \(\Canon(A)\) and the assembled module \(\mathcal R(\Canon(A))\) are constructed canonically from a surgery matrix, and hence give a canonical representative of this equivalence class.
We do not attempt to upgrade the comparison equivalence in Theorem~B to a canonical one.
In \S\ref{sec:calibration-etanu}, when we compute the closed \(4\)--manifold value \(Z_{\TMF}(S^2\times S^2)\), we explicitly assume the additional naturality statement for reversing simply connected cobordisms formulated as \cite[Question~7.14]{GKMP25}.
\end{remark}

\medskip
A key algebraic input is an explicit stable classification statement:
two symmetric integral matrices with the same \(b_1\) and presenting isometric torsion linking pairings become congruent after
adding blocks $(\pm1)$; see Proposition~\ref{prop:stable-classification}.
This is the step that allows one to pass from invariants of $\lambda_A$ to invariants of the
normalized GKMP module, and it is stated in a form directly usable in the proof of Theorem~B.

The dictionary is compatible with connected sum and primewise decomposition.
In particular, $\mathcal R$ is symmetric monoidal with respect to disjoint union of token packages:
\[
\mathcal R(T\oplus T')\ \simeq\ \mathcal R(T)\otimes_{\TMF}\mathcal R(T')
\qquad\text{(Theorem~\ref{thm:assembly}).}
\]
Thus the computation of $\cZ_{\TMF}(Y)$ reduces to a finite list of local generators and their tensor
products, as made explicit in \S\ref{sec:tqft-dictionary}.

Beyond the general formula, we extract two pieces of structure that are intrinsic to $\TMF$
and therefore provide genuinely \emph{geometric information about the GKMP theory}.
First, we identify the values of the closed $4$--manifolds $S^2\times S^2$ and $\CP^2$ with the
classical Hopf elements in $\pi_*\TMF$.

\medskip\noindent
\textbf{Theorem C (Hopf elements from closed $4$--manifolds).}
\emph{In the GKMP theory, one has}
\[
Z_{\TMF}(\CP^2)\ =\ \pm\,\nu\ \in\ \pi_3(\TMF),
\]
\emph{and, assuming \cite[Question~7.14]{GKMP25},}
\[
Z_{\TMF}(S^2\times S^2)\ =\ \eta\ \in\ \pi_1(\TMF),
\]
\emph{where $\eta,\nu$ denote the images in $\pi_*\TMF$ of the classical stable Hopf elements
(cf.\ Adams~\cite{Adams60}).}
\emph{This is proved as Theorem~\ref{thm:etanu-values}.}

\medskip
Second, we prove a time--reversal duality for the rank--one building blocks, formulated as a precise
Grothendieck duality statement in $\TMF$--module theory.

\medskip\noindent
\textbf{Theorem D (Rank--one time--reversal duality).}
\emph{For each integer $n$, the rank--one GKMP module satisfies a canonical equivalence}
\[
L_{(-n)}\ \simeq\ L_{(n)}^\vee[1],
\]
\emph{compatible with the geometric time--reversal cobordism pairing.}
\emph{This is proved as Theorem~\ref{thm:rank-one-time-reversal} in \S\ref{sec:rank-one-duality}.}

\medskip
The equivalence in Theorem~D is canonical up to multiplication by the unit $-1\in\pi_0(\TMF)^\times$;
we explain in \S\ref{sec:rank-one-duality} how this sign ambiguity arises from the normalization of
rank--one generators and how it interacts with duality pairings.
(In particular, it does not affect the validity of Theorem~D, but it must be tracked carefully when
one seeks strict functoriality on the nose.)

\subsection{Proof strategy and organization}
The paper is organized to isolate the genuinely new inputs from the classical classification results.

In \S2 we fix conventions for surgery matrices, linking pairings, and Kirby moves, and we recall the
normalization shift built into the GKMP state space.
In \S3 we define the canonical invariant $\Canon(A)$ and prove its invariance properties.
Sections~4 and~5 develop the local classification needed for a canonical choice:
odd primes are treated via square classes and cyclic summands, while the $2$--primary refinements
require a separate analysis to ensure nondegeneracy of the relevant graded pairings and compatibility
with stabilization.
In \S\ref{sec:tqft-dictionary} we construct the realization functor $\mathcal R$ by choosing a canonical
representative matrix $B(T)$ for each token package $T$, and we prove the state space formula
(Theorem~\ref{thm:token-realization-agrees}) using the stable classification
(Proposition~\ref{prop:stable-classification}) together with the GKMP invariance theorem.
Section~\ref{sec:calibration-etanu} proves Theorem~\ref{thm:etanu-values} and fixes the resulting
normalization of the GKMP theory on two fundamental unimodular forms.
Finally, \S\ref{sec:rank-one-duality} establishes the rank--one time--reversal duality
(Theorem~\ref{thm:rank-one-time-reversal}) using spectral Grothendieck duality methods
(cf.\ Lurie~\cite{LurieSAG} and Mathew--Meier~\cite{MathewMeier15}).

\medskip\noindent
\textbf{Remark (implementation).}
While the constructions in \S\S3--6 are purely mathematical, we also provide a complete implementation
of the canonicalization and realization procedures, together with a reproducible test suite and a
collection of computed examples.
These computational materials are intended as supplementary verification and a tool for exploration,
and are not used in the logical proofs of the main theorems.

\section{Background and conventions}\label{sec:background}

\subsection{Surgery matrices, $2$--handlebodies, and Kirby moves}\label{subsec:surgery-kirby}

We recall the matrix model for integral surgery and fix the algebraic moves that we
treat as \emph{Kirby equivalence} throughout the paper. Standard references for Kirby
calculus include \cite{Kirby78,GompfStipsicz99}.

\subsubsection{Kirby matrices.}
Let $L=L_1\sqcup\cdots\sqcup L_n\subset S^3$ be an \emph{ordered} oriented framed link with
\emph{integer} framings. The associated \emph{Kirby matrix}
\[
A(L)=(a_{ij})\in \Mat_n(\Z)
\]
is the symmetric integral matrix defined by
\[
a_{ii}=\text{framing on }L_i,\qquad
a_{ij}=\lk(L_i,L_j)\quad (i\neq j),
\]
where $\lk$ denotes the oriented linking number in $S^3$.
We write $Y_A$ for the closed oriented $3$--manifold obtained by integral surgery on $L$
(with matrix $A=A(L)$), and $W_A$ for the $4$--manifold obtained from $B^4$ by attaching
$2$--handles along $L$ with the given framings. Then $\partial W_A=Y_A$.
We always regard elements of $\Z^n$ as \emph{column vectors}; transposes are taken accordingly.

\begin{remark}[Orientation convention]\label{rem:orientation}
Our convention is that $Y_A=\partial W_A$ carries the boundary orientation.
In particular, reversing orientation on $Y_A$ changes the sign of the torsion linking pairing
(see \S\ref{subsec:discriminant-linking} and \eqref{eq:orientation-reversal} below).
\end{remark}

\subsubsection{Kirby moves in matrix form.}
Kirby's theorem implies that two framed links yield orientation-preserving diffeomorphic surgery
manifolds if and only if they are related by a sequence of Kirby moves \cite{Kirby78}.
At the level of matrices, we encode these moves by the following operations on symmetric
integral matrices:
\begin{enumerate}
\item \emph{Unimodular congruence} (handle slides):
\[
A\ \longmapsto\ A' = P^{\mathsf T} A P,\qquad P\in \GL_n(\Z).
\]
\item \emph{$\pm 1$--stabilization} (blow-up / blow-down):
\[
A\ \longmapsto\ A\oplus[\varepsilon],\qquad \varepsilon\in\{+1,-1\},
\]
where $[\varepsilon]$ denotes the $1\times 1$ matrix $(\varepsilon)$.
\end{enumerate}
We write $A\sim_{\mathrm{Kirby}} A'$ if $A$ and $A'$ are related by a finite sequence of
these operations.

\subsubsection{Nondegenerate versus singular matrices.}
If $\det(A)\neq 0$, then $H_1(Y_A;\Z)$ is finite, hence $Y_A$ is a rational homology sphere.
This is the case in which the torsion linking pairing is defined on all of $H_1(Y_A;\Z)$.
When $\det(A)=0$, the group $H_1(Y_A;\Z)$ has a free part; in the full theory of Step~A we keep
the full Smith normal form data (hence the free rank) in addition to the torsion linking form.
For the present background section we focus on the nonsingular case, but the algebraic constructions
below are arranged so that the torsion part is isolated canonically.

\subsection{From a matrix to a discriminant group and linking pairing}\label{subsec:discriminant-linking}

Fix a symmetric matrix $A\in \Mat_n(\Z)$.

\subsubsection{The discriminant group.}
Define the \emph{discriminant group} (or \emph{cokernel group})
\[
G(A):=\coker\bigl(A:\Z^n\to \Z^n\bigr)=\Z^n/A\Z^n.
\]
If $\det(A)\neq 0$, then $G(A)$ is finite of order $|\det(A)|$ (e.g.\ by Smith normal form).
In \S\ref{subsubsec:topological-identification} we recall the standard identification
$G(A)\cong H_1(Y_A;\Z)$ coming from the handlebody $W_A$.

\subsubsection{The matrix linking form.}
Assume $\det(A)\neq 0$. The inverse $A^{-1}$ is a rational matrix, so for $x,y\in \Z^n$
the quantity $x^{\mathsf T}A^{-1}y$ lies in $\Q$. We define
\begin{equation}\label{eq:lambda-matrix}
\lambda_A\colon G(A)\times G(A)\to \Q/\Z,\qquad
\lambda_A(\bar x,\bar y):=x^{\mathsf T}A^{-1}y \bmod \Z,
\end{equation}
where $\bar x,\bar y$ denote the classes of $x,y$ in $G(A)$.

\begin{lemma}\label{lem:lambda-well-defined}
If $\det(A)\neq 0$, then $\lambda_A$ is well-defined and is a nonsingular symmetric bilinear pairing.
\end{lemma}

\begin{proof}
\emph{Well-definedness.}
If $x'=x+Az$ for some $z\in \Z^n$, then
\[
(x')^{\mathsf T}A^{-1}y
=(x+Az)^{\mathsf T}A^{-1}y
=x^{\mathsf T}A^{-1}y + z^{\mathsf T}y,
\]
and $z^{\mathsf T}y\in \Z$, so the class modulo $\Z$ is unchanged.
Similarly, replacing $y$ by $y+Az'$ changes $x^{\mathsf T}A^{-1}y$ by the integer $x^{\mathsf T}z'$.
Hence \eqref{eq:lambda-matrix} depends only on $\bar x,\bar y$.

\emph{Bilinearity and symmetry.}
Bilinearity is immediate from the formula. Since $A$ is symmetric, so is $A^{-1}$,
hence $x^{\mathsf T}A^{-1}y=y^{\mathsf T}A^{-1}x$, and $\lambda_A$ is symmetric.

\emph{Nonsingularity.}
Let $\ad_{\lambda_A}\colon G(A)\to \Hom(G(A),\Q/\Z)$ be the adjoint map,
$\ad_{\lambda_A}(\bar x)(\bar y)=\lambda_A(\bar x,\bar y)$.
Suppose $\ad_{\lambda_A}(\bar x)=0$. Then $x^{\mathsf T}A^{-1}y\in\Z$ for all $y\in \Z^n$.
Put $u:=A^{-1}x\in \Q^n$. The condition says $u^{\mathsf T}y\in \Z$ for all $y\in \Z^n$,
so $u\in \Hom(\Z^n,\Z)\cong \Z^n$. Therefore $x=Au\in A\Z^n$, i.e.\ $\bar x=0$ in $G(A)$.
Thus $\ad_{\lambda_A}$ is injective. Since $G(A)$ is finite, injective implies bijective,
so $\lambda_A$ is nonsingular.
\end{proof}

\subsubsection{Topological interpretation.}\label{subsubsec:topological-identification}
We next relate $(G(A),\lambda_A)$ to the classical torsion linking form on $H_1(Y_A;\Z)$.

\begin{proposition}[Homology and the torsion linking pairing]\label{prop:topological-identification}
Let $A\in \Mat_n(\Z)$ be symmetric with $\det(A)\neq 0$. Then $H_1(Y_A;\Z)$ is finite, and there is a
canonical identification
\[
H_1(Y_A;\Z)\ \cong\ G(A),
\]
under which the torsion linking pairing $\lambda_{Y_A}$ on $H_1(Y_A;\Z)$ agrees with $\lambda_A$.
\end{proposition}

\begin{proof}
We first identify $H_1(Y_A;\Z)$ with $G(A)$.
The handlebody $W_A$ has a handle decomposition with a single $0$--handle and $n$ $2$--handles,
so $H_1(W_A;\Z)=0$ and $H_2(W_A;\Z)\cong \Z^n$.
Let $e_i\in H_2(W_A;\Z)$ denote the classes of the $2$--handle cores capped by Seifert surfaces
for $L_i$ in $S^3$ (as in \cite[\S~5.2]{GompfStipsicz99}); then $\{e_i\}$ is a basis and the intersection form
in this basis is represented by the Kirby matrix $A$.

Consider the long exact sequence of the pair $(W_A,Y_A)$:
\[
H_2(W_A;\Z)\xrightarrow{\ \iota\ } H_2(W_A,Y_A;\Z)\xrightarrow{\ \partial\ } H_1(Y_A;\Z)\to H_1(W_A;\Z)=0,
\]
so $H_1(Y_A;\Z)\cong \coker(\iota)$.
By Poincar\'e--Lefschetz duality,
\[
H_2(W_A,Y_A;\Z)\ \cong\ H^2(W_A;\Z)\ \cong\ \Hom\bigl(H_2(W_A;\Z),\Z\bigr).
\]
Under this identification, the map $\iota$ sends $x\in H_2(W_A;\Z)$ to the functional
$(x\cdot -)\in \Hom(H_2(W_A;\Z),\Z)$, i.e.\ it is the adjoint of the intersection pairing.
In the basis $\{e_i\}$, this map is represented by the matrix $A$.
Therefore
\[
H_1(Y_A;\Z)\ \cong\ \coker\bigl(A:\Z^n\to \Z^n\bigr)\ =\ G(A).
\]
Since $\det(A)\neq 0$, this group is finite, hence $Y_A$ is a rational homology sphere.

We now compare linking pairings.
For an oriented rational homology sphere $Y=\partial W$ with $H_1(W;\Z)=0$, there is a standard
description of the torsion linking form on $H_1(Y;\Z)$ in terms of the (rational) inverse
intersection pairing on $H_2(W,Y;\Z)$; see for example \cite[\S~1]{Wall63}.
Concretely, the map $\iota\colon H_2(W;\Z)\to H_2(W,Y;\Z)$ becomes an isomorphism after tensoring with $\Q$.
Given classes $\alpha,\beta\in H_1(Y;\Z)\cong \coker(\iota)$, choose lifts
$\widetilde\alpha,\widetilde\beta\in H_2(W,Y;\Z)$ and let
$u,v\in H_2(W;\Q)$ be the unique rational classes with $\iota(u)=\widetilde\alpha$ and $\iota(v)=\widetilde\beta$.
Then the quantity $u\cdot v\in \Q$ depends on the choices only modulo $\Z$, and
\[
\lambda_Y(\alpha,\beta)\ =\ u\cdot v \bmod \Z\ \in\ \Q/\Z.
\]
In our handlebody $W_A$, the basis $\{e_i\}$ identifies $H_2(W_A;\Z)$ with $\Z^n$, and the dual basis
identifies $H_2(W_A,Y_A;\Z)$ with $\Z^n$ so that $\iota$ is represented by $A$.
Thus, if $\widetilde\alpha,\widetilde\beta$ correspond to $x,y\in \Z^n$ in $H_2(W_A,Y_A;\Z)$,
then $u=A^{-1}x$ and $v=A^{-1}y$ in $H_2(W_A;\Q)\cong \Q^n$, and
\[
u\cdot v\ =\ (A^{-1}x)^{\mathsf T}A(A^{-1}y)\ =\ x^{\mathsf T}A^{-1}y.
\]
Modulo $\Z$, this is exactly \eqref{eq:lambda-matrix}. Hence $\lambda_{Y_A}$ agrees with $\lambda_A$
under the identification $H_1(Y_A;\Z)\cong G(A)$.
\end{proof}

\subsubsection{Kirby invariance of the isometry class.}
Although $\lambda_A$ is defined from a chosen matrix $A$, its isometry class depends only on $Y_A$.

\begin{proposition}[Congruence and stabilization invariance]\label{prop:kirby-invariance-lambda}
Let $A\in \Mat_n(\Z)$ be symmetric with $\det(A)\neq 0$.
\begin{enumerate}
\item If $A'=P^{\mathsf T}AP$ with $P\in\GL_n(\Z)$, then the map
\[
\phi_P\colon G(A)\to G(A'),\qquad \phi_P(\bar x)=\overline{P^{\mathsf T}x},
\]
is a well-defined group isomorphism and an isometry:
$\lambda_{A'}\!\bigl(\phi_P(\bar x),\phi_P(\bar y)\bigr)=\lambda_A(\bar x,\bar y)$.
\item If $A''=A\oplus[\varepsilon]$ with $\varepsilon\in\{+1,-1\}$, then the canonical identification
$G(A'')\cong G(A)$ induced by projection $\Z^{n+1}\to \Z^n$ is an isometry:
$\lambda_{A''}\cong \lambda_A$.
\end{enumerate}
Consequently, the isometry class of $(G(A),\lambda_A)$ is invariant under $\sim_{\mathrm{Kirby}}$.
\end{proposition}

\begin{proof}
(1) If $x'\equiv x \pmod{A\Z^n}$, say $x'=x+Az$, then
\[
P^{\mathsf T}x' = P^{\mathsf T}x + P^{\mathsf T}Az
= P^{\mathsf T}x + P^{\mathsf T}AP(P^{-1}z)
= P^{\mathsf T}x + A'(P^{-1}z),
\]
and $P^{-1}z\in \Z^n$ since $P$ is unimodular. Thus $\overline{P^{\mathsf T}x'}=\overline{P^{\mathsf T}x}$ in $G(A')$,
so $\phi_P$ is well-defined. It is bijective with inverse induced by $(P^{-1})^{\mathsf T}$.

For the pairing, note that $(P^{\mathsf T}AP)^{-1}=P^{-1}A^{-1}P^{-{\mathsf T}}$ over $\Q$.
Therefore, for lifts $x,y\in \Z^n$ we compute
\[
\lambda_{A'}\!\bigl(\overline{P^{\mathsf T}x},\overline{P^{\mathsf T}y}\bigr)
=
(P^{\mathsf T}x)^{\mathsf T}(A')^{-1}(P^{\mathsf T}y)\bmod \Z
=
x^{\mathsf T}A^{-1}y\bmod \Z
=
\lambda_A(\bar x,\bar y).
\]

(2) Write $A''=A\oplus[\varepsilon]$ and identify $\Z^{n+1}\cong \Z^n\oplus \Z$.
Then
\[
G(A'')
=
\frac{\Z^n\oplus \Z}{A\Z^n\oplus \varepsilon\Z}
\cong
\frac{\Z^n}{A\Z^n}\ \oplus\ \frac{\Z}{\varepsilon\Z}
\cong
G(A)\oplus 0
\cong
G(A).
\]
Moreover $(A\oplus[\varepsilon])^{-1}=A^{-1}\oplus[\varepsilon^{-1}]$ in $\Mat_{n+1}(\Q)$, and since
$\varepsilon^{-1}=\varepsilon\in\Z$, the last summand contributes only integral values in \eqref{eq:lambda-matrix},
hence vanishes in $\Q/\Z$. Thus the induced pairing on $G(A'')\cong G(A)$ is exactly $\lambda_A$.
\end{proof}

\begin{remark}[The torsion linking form for a singular surgery matrix]\label{rem:torsion-linking-singular}
When \(\det(A)=0\), the cokernel \(G(A)=\coker(A)\) has a free summand.
Nevertheless, the torsion subgroup \(\Tor(G(A))\) carries a canonical nonsingular linking pairing:
choose \(P\in \GL_n(\Z)\) such that
\[
P^{\mathsf T}AP \ =\ \begin{pmatrix}0 & 0\\ 0 & A_{\mathrm{red}}\end{pmatrix}
\]
with \(A_{\mathrm{red}}\) nonsingular, identify \(\Tor(G(A))\cong G(A_{\mathrm{red}})\), and transport the pairing \(\lambda_{A_{\mathrm{red}}}\) defined by \eqref{eq:lambda-matrix}.
Different choices of such a reduction are related by unimodular congruence on the nonsingular block, so Proposition~\ref{prop:kirby-invariance-lambda}(1) shows that the resulting pairing is well-defined up to isometry; it agrees with the classical torsion linking form on \(\Tor H_1(Y_A;\Z)\).
We will denote this torsion pairing again by \(\lambda_A\).
For a detailed discussion of this construction in the context of framed link surgery, see for instance
\cite[\S1]{KawauchiKojima80}.
\end{remark}

\subsubsection{Orientation reversal.}
If $A$ is symmetric with $\det(A)\neq 0$, then $-A$ is also symmetric and $(-A)^{-1}=-A^{-1}$.
Thus
\begin{equation}\label{eq:orientation-reversal}
\bigl(G(-A),\lambda_{-A}\bigr)\ \cong\ \bigl(G(A),-\lambda_A\bigr).
\end{equation}
Topologically, this agrees with the fact that reversing orientation on a rational homology sphere
negates the torsion linking pairing; see \cite[\S~1]{Wall63}.

\subsection{Linking pairings as algebraic input}\label{subsec:linking-algebra}

A \emph{(torsion) linking pairing} is a pair $(G,\lambda)$ where $G$ is a finite abelian group and
\[
\lambda\colon G\times G\to \Q/\Z
\]
is a symmetric bilinear map such that the adjoint $\ad_\lambda\colon G\to \Hom(G,\Q/\Z)$ is an isomorphism.
Two linking pairings are equivalent if they are \emph{isometric}, i.e.\ related by a group isomorphism
preserving $\lambda$.

Orthogonal direct sum defines a commutative monoid structure:
\[
(G_1,\lambda_1)\ \oplus\ (G_2,\lambda_2)
:=
(G_1\oplus G_2,\ \lambda_1\oplus\lambda_2),
\]
where cross-terms vanish. Prime decomposition gives a canonical orthogonal splitting
\[
(G,\lambda)\ \cong\ \bigoplus_{p}\ (G_{(p)},\lambda_{(p)}),
\]
where $G_{(p)}$ is the $p$--primary subgroup.
Classification of isometry classes is classical.
For odd primes, Wall~\cite{Wall63} describes generators and relations via discriminant forms of lattices.
Kawauchi--Kojima~\cite{KawauchiKojima80} gives a complete algebraic classification for linkings arising from
closed oriented $3$--manifolds, including the $2$--primary relations.
For explicit normal forms in the $2$--group case convenient for our later constructions,
we also use Miranda's refinement~\cite{Miranda84}.

\subsection{Stable \texorpdfstring{$\TMF$}{TMF}-modules, shifts, and duals}\label{subsec:tmf-modules}

We use $\TMF$ to denote the $E_\infty$ ring spectrum of topological modular forms.
Write $\Mod_{\TMF}$ for the stable symmetric monoidal category of (right) $\TMF$--module spectra,
and $\Ho(\Mod_{\TMF})$ for its homotopy category.
Our conventions are compatible with the TMF--based $(3{+}1)$--dimensional TQFT prototype of
Gukov--Krushkal--Meier--Pei~\cite{GKMP25}.

\subsubsection{Shifts and triangles.}
For a $\TMF$--module $M$, we write $M[1]$ for the suspension $\Sigma M$ and more generally
$M[k]=\Sigma^k M$ for $k\in\Z$.
Cofiber sequences in $\Mod_{\TMF}$ are written as exact triangles in $\Ho(\Mod_{\TMF})$.

\subsubsection{Internal Hom and duals.}
Let $\underline{\Hom}_{\TMF}(-,-)$ denote the internal mapping spectrum in $\Mod_{\TMF}$.
For a $\TMF$--module $M$ we define its \emph{$\TMF$--linear dual} by
\[
M^\vee := \underline{\Hom}_{\TMF}(M,\TMF).
\]
When $M$ is dualizable (for instance, if $M$ is perfect or built from a finite cell $\TMF$--module),
there are evaluation and coevaluation maps exhibiting $M^\vee$ as the dual object in the symmetric monoidal category.
We will only apply $(\cdot)^\vee$ in contexts where dualizability is guaranteed by construction, as in~\cite{GKMP25};
for general background on duality in stable homotopy theoretic settings see~\cite{HPS97}.

\subsubsection{Notation for the TQFT and the realization dictionary.}
Following~\cite{GKMP25}, we denote by $\cZ_{\TMF}$ the $(3{+}1)$--dimensional $\TMF$--valued TQFT under consideration.
(We reserve $\Z$ exclusively for the integers.)
In this paper, Step~A associates to a Kirby matrix $A$ a \emph{canonical token package} $\Canon(A)$ encoding:
\begin{enumerate}
\item the full Smith normal form data (in particular the free rank of $H_1(Y_A;\Z)$), and
\item the isometry class of the torsion linking pairing $(\Tor H_1(Y_A;\Z),\lambda_{Y_A})$
in a prime-sorted normal form.
\end{enumerate}
Step~B then defines a realization dictionary (a functor) $\mathcal R$ sending local tokens to explicit $\TMF$--modules,
so that the Step~A output can be functorially compared with the state space assignment
\[
Y\ \longmapsto\ \cZ_{\TMF}(Y)\ \in\ \Ho(\Mod_{\TMF}).
\]

\subsection{Running examples and a notation table}\label{subsec:running-examples}

\subsubsection{Benchmark surgery matrices}\label{subsubsec:benchmark-matrices}
We fix the following concrete matrices as running examples. They will be used repeatedly to illustrate
the constructions in Steps~A and~B.

\begin{enumerate}
\item \emph{The trivial case.}
For $A=[1]$, surgery yields $Y_A\cong S^3$, hence $H_1(Y_A;\Z)=0$ and $\lambda_A=0$.

\item \emph{Rank-one (lens space) family.}
For $A=[n]$ with $n\neq 0$, one has $G(A)\cong \Z/|n|$, and
\[
\lambda_A(\bar 1,\bar 1)=\frac{1}{n}\bmod \Z.
\]
Up to our orientation conventions, $Y_{[n]}$ is the lens space $L(n,1)$.
By \eqref{eq:orientation-reversal}, replacing $n$ by $-n$ negates the pairing.

\item \emph{A $2$--primary hyperbolic block.}
Let
\[
A_{\mathrm{hyp}}=\begin{pmatrix}0&2\\2&0\end{pmatrix}.
\]
Then $G(A_{\mathrm{hyp}})\cong (\Z/2)^2$, and $\lambda_{A_{\mathrm{hyp}}}$ is the standard hyperbolic
nonsingular symmetric form on $(\Z/2)^2$ (distinguished from the diagonal form coming from $\diag(2,2)$).

\item \emph{A mixed-prime example of small order.}
Let $A=\diag(2,6)$. Then $G(A)\cong \Z/2\oplus \Z/6$ has order $12$, and its $2$--primary and
$3$--primary layers provide a minimal testbed for prime splitting in Step~A.
\end{enumerate}

\subsubsection{Notation table.}
For quick reference, we summarize a few notations fixed above:
\[
\begin{array}{ll}
A & \text{symmetric integer surgery matrix (Kirby matrix)}\\
W_A & \text{$2$--handlebody obtained from $B^4$ by attaching $2$--handles with matrix $A$}\\
Y_A & \partial W_A,\ \text{the resulting closed oriented $3$--manifold}\\
G(A) & \coker(A:\Z^n\to\Z^n)\cong H_1(Y_A;\Z)\\
\lambda_A & \text{matrix linking form on $G(A)$ given by \eqref{eq:lambda-matrix} (if $\det A\neq 0$)}\\
\Canon(A) & \text{Step~A canonical token package encoding SNF + linking isometry class}\\
\mathcal R & \text{Step~B realization dictionary from tokens to $\TMF$--modules}\\
\cZ_{\TMF} & \text{the $\TMF$--based $(3{+}1)$--TQFT functor of \cite{GKMP25}}\\
\end{array}
\]

\section{Step A: Canonical invariants for torsion linking pairings}
\label{sec:stepA-canon}

This section isolates the purely algebraic datum that controls the torsion-sensitive
part of the $(3{+}1)$--dimensional $\TMF$--TQFT state space in the framework of
Gukov--Krushkal--Meier--Pei: the nonsingular torsion linking pairing of a closed
oriented $3$--manifold.
Concretely, we begin by developing the torsion classification in the nonsingular case:
$A\in M_n(\Z)$ with $\det(A)\neq 0$ presents a rational homology sphere $Y_A$.
The induced pairing $(G(A),\lambda_A)$ is a classical invariant of $Y_A$, well-defined
up to isometry and unchanged by Kirby moves (integral congruence and $\pm1$ stabilization).
When $\det(A)=0$, the cokernel has a free part; we apply the same constructions to the
canonical torsion pairing on $\Tor(G(A))$ (Remark~\ref{rem:torsion-linking-singular}) and record
$b_1(A)$ as part of the Step~A output.

The goal of Step~A is not to introduce a new invariant, but to choose a \emph{canonical}
description of this isometry class that can be used unambiguously in later constructions.
Following the prime-by-prime classification of linking pairings (Wall, Kawauchi--Kojima,
and Miranda), we extract from $(G(A),\lambda_A)$ an explicit tuple of invariants,
denoted $\Canon(A)$, which is \emph{complete} for the isometry class.
For notational economy, we also use an equivalent “token list” $\Tok(A)$, obtained by
flattening $\Canon(A)$ into a prime-ordered list of local symbols.
The tokens carry no additional structure; they are simply an encoding of the same
canonical invariants.

\subsection{From an integral surgery matrix to a linking pairing}
\label{subsec:matrix-to-linking}

Let $A\in M_n(\Z)$ be a symmetric matrix with $\det(A)\neq 0$.
Write
\[
G(A)\;=\;\coker(A)\;=\;\Z^n/A\Z^n,
\]
a finite abelian group of order $|G(A)|=|\det(A)|$.
We define a pairing
\begin{equation}\label{eq:def-linking-from-A}
\lambda_A \colon G(A)\times G(A)\longrightarrow \Q/\Z,
\qquad
\lambda_A\big([x],[y]\big)\;=\;x^{\mathsf T}A^{-1}y\ \bmod\ \Z,
\end{equation}
where $x,y\in \Z^n$ are representatives.

\begin{lemma}\label{lem:linking-well-defined}
The map $\lambda_A$ is a well-defined, nonsingular, symmetric bilinear pairing.
In particular $(G(A),\lambda_A)$ is a linking pairing in the sense of
Wall and Kawauchi--Kojima.
\end{lemma}

\begin{proof}
\emph{Well-definedness.}
If $x'=x+Az$ for some $z\in \Z^n$, then
\[
{x'}^{\mathsf T}A^{-1}y
=(x+Az)^{\mathsf T}A^{-1}y
=x^{\mathsf T}A^{-1}y+z^{\mathsf T}y,
\]
and $z^{\mathsf T}y\in\Z$, so the class in $\Q/\Z$ is unchanged.
The same computation in the second variable shows independence of the choice of $y$.

\emph{Bilinearity and symmetry.}
Bilinearity is immediate from the formula, and symmetry holds because $A^{-1}$ is symmetric.

\emph{Nonsingularity.}
Consider the adjoint homomorphism
\[
\ad_{\lambda_A}\colon G(A)\longrightarrow \Hom(G(A),\Q/\Z),
\qquad
[x]\longmapsto\big([y]\mapsto \lambda_A([x],[y])\big).
\]
Let $x\in \Z^n$ represent an element in the kernel.  Then
$x^{\mathsf T}A^{-1}y\in \Z$ for all $y\in \Z^n$.
Equivalently, $A^{-{\mathsf T}}x=A^{-1}x\in \Z^n$.
Write $A^{-1}x=z\in\Z^n$. Then $x=Az$, so $[x]=0$ in $G(A)$.
Thus $\ad_{\lambda_A}$ is injective. Since $G(A)$ is finite, it is bijective.
\end{proof}

In the topological situation, $A$ is the linking matrix of a framed link describing
a $2$--handlebody whose boundary is a rational homology sphere $M$; then
$G(A)\cong H_1(M;\Z)$ and $\lambda_A$ agrees with the classical torsion
linking form $\lambda_M$ defined by Poincar\'e duality (see, e.g.,
\cite[\S~5.2]{GompfStipsicz99} and \cite{Kirby78}).

\subsection{Kirby invariance: congruence and stabilization}
\label{subsec:kirby-invariance}

Two features of Kirby calculus are relevant for Step~A.
First, handle-slides correspond to \emph{integral congruence}
$A\mapsto P^{\mathsf T}AP$ with $P\in GL_n(\Z)$.
Second, $\pm 1$--stabilization corresponds to $A\mapsto A\oplus (\pm 1)$.
Both operations leave the boundary $3$--manifold unchanged \cite{Kirby78}.

\begin{lemma}\label{lem:congruence-invariance}
If $A' = P^{\mathsf T}AP$ with $P\in GL_n(\Z)$, then there is an isomorphism
$f\colon G(A)\to G(A')$ such that $\lambda_{A'}(f(x),f(y))=\lambda_A(x,y)$ for all
$x,y\in G(A)$.  Hence the isometry class of $(G(A),\lambda_A)$ depends only on
the Kirby-equivalence class of $A$.
\end{lemma}

\begin{proof}
Define \(f([x])=[P^{\mathsf T}x]\).
If \(x'=x+Az\), then
\[
P^{\mathsf T}x' \;=\; P^{\mathsf T}x + P^{\mathsf T}Az
\;=\; P^{\mathsf T}x + (P^{\mathsf T}AP)(P^{-1}z),
\]
so \(f([x'])=f([x])\) in \(G(A')=\coker(P^{\mathsf T}AP)\); thus \(f\) is well-defined.
Since \(P^{\mathsf T}\in GL_n(\Z)\), the map \(f\) is an isomorphism.
Using \((P^{\mathsf T}AP)^{-1}=P^{-1}A^{-1}P^{-{\mathsf T}}\), we compute
\[
\lambda_{A'}\bigl(f([x]),f([y])\bigr)
=(P^{\mathsf T}x)^{\mathsf T}(P^{\mathsf T}AP)^{-1}(P^{\mathsf T}y)
=x^{\mathsf T}A^{-1}y
\quad\text{in }\Q/\Z,
\]
so \(f\) is an isometry.
\end{proof}

\begin{lemma}\label{lem:stabilization-invariance}
Let $A_\pm = A\oplus(\pm 1)$.  Then $G(A_\pm)\cong G(A)$ and under this identification
$\lambda_{A_\pm}$ agrees with $\lambda_A$.
\end{lemma}

\begin{proof}
Since $\coker(\pm 1)=0$, we have $G(A_\pm)\cong G(A)\oplus 0\cong G(A)$.
Moreover, $(A\oplus(\pm 1))^{-1}=A^{-1}\oplus(\pm 1)$, and the extra summand pairs
trivially with $G(A)$ because it corresponds to the zero group.  Hence the linking
pairing is unchanged.
\end{proof}

Thus any invariant extracted solely from the isometry class of $(G(A),\lambda_A)$
is automatically a Kirby invariant.  Step~A is the extraction of a \emph{canonical}
complete invariant of this isometry class.

\subsection{Smith normal form and a canonical basis for the group}
\label{subsec:snf}

Let $A$ be as above.  Choose unimodular matrices $U,V\in GL_n(\Z)$ such that
\begin{equation}\label{eq:snf}
UAV \;=\; D \;=\; \diag(d_1,\dots,d_n),
\qquad
d_1\mid d_2\mid \cdots \mid d_n,
\end{equation}
the Smith normal form.
Then
\[
G(A)\cong G(D)\cong \bigoplus_{i=1}^n \Z/d_i\Z,
\]
canonically up to the unique invariant factors $(d_i)$.

The left transformation $U$ provides a convenient way to transport the pairing.
Let $\bar e_i$ denote the class of the standard basis vector $e_i\in \Z^n$
in $G(D)$, so that $\bar e_i$ has order $d_i$.

\begin{lemma}\label{lem:pairing-matrix-snf}
Let $\phi\colon G(A)\to G(D)$ be the isomorphism induced by multiplication by $U$:
$\phi([x])=[Ux]$.  In the basis $(\bar e_i)$ of $G(D)$, the pairing transported from
$\lambda_A$ has Gram matrix
\begin{equation}\label{eq:Pfull}
P \;=\; U^{-{\mathsf T}}A^{-1}U^{-1}\ \in\ M_n(\Q),
\end{equation}
interpreted entrywise in $\Q/\Z$.
\end{lemma}

\begin{proof}
For $x',y'\in \Z^n$ representing elements of $G(D)$, write $x'=Ux$ and $y'=Uy$.
Then
\[
\lambda_A([x],[y]) = x^{\mathsf T}A^{-1}y
= (U^{-1}x')^{\mathsf T}A^{-1}(U^{-1}y')
= {x'}^{\mathsf T}(U^{-{\mathsf T}}A^{-1}U^{-1})y',
\]
modulo $\Z$.  This identifies the Gram matrix in the $G(D)$ basis.
\end{proof}

A key point is that while the matrices $U,V$ are not unique, the invariants we
extract from $P$ below are \emph{intrinsic}: they depend only on the isometry class
of $(G(A),\lambda_A)$, hence are canonical.

\subsection{Primary splitting and $p^k$--layers}
\label{subsec:primary-layers}

Write $G=\bigoplus_p G_{(p)}$ for the primary decomposition of a finite abelian group,
where $G_{(p)}$ is the subgroup of elements of $p$--power order.

\begin{lemma}\label{lem:primary-orthogonal}
Let $(G,\lambda)$ be a linking pairing. Then
\[
\lambda\big(G_{(p)},G_{(q)}\big)=0\qquad (p\neq q),
\]
so $(G,\lambda)$ is the orthogonal direct sum of its $p$--primary components.
\end{lemma}

\begin{proof}
Let $x\in G_{(p)}$ and $y\in G_{(q)}$, so $p^a x=0$ and $q^b y=0$ for some $a,b$.
Then $0=\lambda(p^a x,y)=p^a\lambda(x,y)$ and $0=\lambda(x,q^b y)=q^b\lambda(x,y)$
in $\Q/\Z$.  Hence $\lambda(x,y)$ is simultaneously $p^a$--torsion
and $q^b$--torsion in $\Q/\Z$, so it is killed by $\gcd(p^a,q^b)=1$,
and therefore $\lambda(x,y)=0$.
\end{proof}

In the Smith normal form basis, each cyclic factor $\Z/d_i$ has a canonical
$p$--primary generator obtained by removing the prime-to-$p$ part of $d_i$.
Let $v_p(\cdot)$ be the $p$--adic valuation, and write
\[
d_i = p^{e_i}\cdot s_i,\qquad e_i=v_p(d_i),\ \ \gcd(s_i,p)=1.
\]
Define $\bar e_{i,p}=s_i\,\bar e_i\in \Z/d_i\Z$, an element of order
$p^{e_i}$ (possibly $1$).

\begin{lemma}\label{lem:ppart-matrix}
Let $P$ be as in \eqref{eq:Pfull}.  For a fixed prime $p$, the $p$--primary pairing
$\lambda_{(p)}$ in the basis $\{\bar e_{i,p}\}_{e_i>0}$ has Gram matrix
\begin{equation}\label{eq:Pp}
P^{(p)}_{ij} \equiv s_is_j P_{ij}\ \ \ (\mathrm{mod}\ 1),
\end{equation}
for indices $i,j$ with $e_i,e_j>0$.
\end{lemma}

\begin{proof}
By construction, $\bar e_{i,p}=s_i\bar e_i$ and $\bar e_{j,p}=s_j\bar e_j$.
Since $\lambda$ is bilinear, $\lambda(\bar e_{i,p},\bar e_{j,p})=s_is_j\lambda(\bar e_i,\bar e_j)$.
In the $G(D)$ basis, $\lambda(\bar e_i,\bar e_j)\equiv P_{ij}\ (\mathrm{mod}\ 1)$, so the claim follows.
\end{proof}

We now isolate the homogeneous ``$p^k$--layers''.
For each $k\ge 1$ define the index set
\[
I_{p,k}\;=\;\{\, i \mid v_p(d_i)=k\,\},
\qquad
n_{p,k}=|I_{p,k}|.
\]
The subgroup generated by $\{\bar e_{i,p}\}_{i\in I_{p,k}}$ is canonically isomorphic to
$(\Z/p^k\Z)^{n_{p,k}}$.

\begin{lemma}\label{lem:layer-reduction}
For $i,j\in I_{p,k}$, the quantity $p^k\,P^{(p)}_{ij}$ is an integer.
Its reduction modulo $p$ defines a symmetric matrix
\[
B_{p,k}\in M_{n_{p,k}}(\F_p),
\qquad
(B_{p,k})_{ij} \equiv p^k P^{(p)}_{ij}\pmod p,
\]
which is nonsingular.
\end{lemma}

\begin{proof}
When $i,j\in I_{p,k}$, the elements $\bar e_{i,p},\bar e_{j,p}$ have order $p^k$.
Hence $p^k\lambda(\bar e_{i,p},\bar e_{j,p})=0$ in $\Q/\Z$, which means
$p^k P^{(p)}_{ij}\in \Z$.  Symmetry is inherited from $P^{(p)}$.

To see nonsingularity, consider the $p$--primary subgroup $G(A)_{(p)}$ and its $k$th layer quotient
\[
P_k\bigl(G(A)_{(p)}\bigr)
:= G(A)_{(p)}[p^k]\Big/\bigl(G(A)_{(p)}[p^{k-1}] + p\,G(A)_{(p)}[p^{k+1}]\bigr),
\]
as in Appendix~\ref{app:odd-primes} for odd $p$ (Definition~\ref{def:odd-layer-form}) and
Appendix~\ref{app:two-primary} for $p=2$ \eqref{eq:Pk-def-appB}.
In the Smith basis, the classes $\{\bar e_{i,p}\}_{i\in I_{p,k}}$ map to a basis of this $\F_p$--vector space,
and the induced layer form is
\[
b_k(\bar x,\bar y)=p^k\,\lambda_{(p)}(x,y)\bmod p.
\]
By definition, $B_{p,k}$ is the Gram matrix of $b_k$ in that basis.
The layer form is nonsingular (Lemma~\ref{lem:odd-bk-well-defined} for odd $p$ and
Proposition~\ref{prop:bk-welldefined} for $p=2$), hence $B_{p,k}$ is nonsingular.
\end{proof}

\subsection{Odd primes: the determinant symbol}
\label{subsec:odd-primes}

Assume $p$ is odd.  Over $\F_p$, the isometry class of a nonsingular symmetric
bilinear form is determined by its dimension and the square-class of its determinant.
We record the latter via the Legendre symbol.

\begin{definition}\label{def:odd-layer-invariant}
For $p$ odd and $k\ge 1$, define
\[
x_{p,k}\;=\;\left(\frac{\det(B_{p,k})}{p}\right)\in \{\pm 1\},
\]
where $B_{p,k}$ is the matrix of Lemma~\ref{lem:layer-reduction}.
When $n_{p,k}=0$ we omit the layer.
\end{definition}

\begin{lemma}\label{lem:finite-field-classification}
Let $p$ be an odd prime.  Two nonsingular symmetric bilinear forms over $\F_p$
are isometric if and only if they have the same dimension and the same determinant
in $\F_p^\times/(\F_p^\times)^2$.
Equivalently, they are isometric if and only if they have the same dimension and
the same Legendre symbol of the determinant.
\end{lemma}

\begin{proof}
Let $b$ be a nonsingular symmetric bilinear form on an $\F_p$--vector space $V$
of dimension $n$.  Choose a basis and let $M$ be the Gram matrix, so $M$ is symmetric
with $\det(M)\neq 0$.
A standard Gram--Schmidt argument over a field of characteristic $\neq 2$ diagonalizes
$b$: there exists $S\in GL_n(\F_p)$ such that
$S^{\mathsf T}MS=\diag(a_1,\dots,a_n)$ with $a_i\in\F_p^\times$.
Multiplying the $i$th basis vector by $u_i\in\F_p^\times$ scales $a_i$ by $u_i^2$,
so each $a_i$ is well-defined only up to squares.  Hence any diagonal form is isometric to
\[
\langle 1,\dots,1,\delta\rangle,
\qquad \delta\in \F_p^\times/(\F_p^\times)^2,
\]
where $\delta\equiv a_1\cdots a_n \bmod (\F_p^\times)^2$ is the determinant class.
This shows existence of a normal form depending only on $(n,\delta)$, and also that
$(n,\delta)$ is a complete invariant.  The Legendre symbol detects $\delta$.
\end{proof}

The passage from the $\F_p$--forms $B_{p,k}$ to the full $p$--primary linking
is part of the classical classification of linking pairings.  For odd primes this goes
back to Seifert and Wall, and is subsumed in Kawauchi--Kojima's complete system
\cite{Wall63,KawauchiKojima80}.

\begin{theorem}[Odd-primary completeness]\label{thm:odd-primary-completeness}
Let $(G,\lambda)$ be a linking pairing and fix an odd prime $p$.
Then the isometry class of the $p$--primary component $(G_{(p)},\lambda_{(p)})$
is determined by the collection of invariants $\{(n_{p,k},x_{p,k})\}_{k\ge 1}$.
\end{theorem}

\begin{proof}
By Wall's homogeneous splitting theorem for $p$--primary linkings \cite{Wall63},
$(G_{(p)},\lambda_{(p)})$ decomposes as an orthogonal direct sum of homogeneous
summands of exponent $p^k$.  Each homogeneous summand is represented by a matrix
whose reduction modulo $p$ is a nonsingular symmetric form on an $\F_p$--vector
space of dimension $n_{p,k}$; this is precisely the form encoded by $B_{p,k}$.
By Lemma~\ref{lem:finite-field-classification}, its isometry class is determined by
$(n_{p,k},x_{p,k})$.  Taking the orthogonal sum over $k$ reconstructs the full
$p$--primary isometry class.
\end{proof}

\subsection{The prime $2$: types and Gauss sum invariants}
\label{subsec:two-primary}

The $2$--primary theory is subtler because symmetric bilinear forms in characteristic $2$
are not classified by determinant alone.  The complete algebraic classification of
$2$--primary linkings is due to Kawauchi--Kojima \cite{KawauchiKojima80}; a particularly
explicit normal form, expressed in Wall's generators, is developed by Miranda
\cite{Miranda84}.

Fix $p=2$ and a layer $k\ge 1$ with index set $I_{2,k}$.
Let $C_k$ be the integral matrix
\[
C_k \;=\; 2^k \cdot P^{(2)}[I_{2,k},I_{2,k}] \in M_{n_{2,k}}(\Z),
\]
where $P^{(2)}$ is as in Lemma~\ref{lem:ppart-matrix}.
Reducing $C_k$ modulo $2$ yields a symmetric matrix over $\F_2$.

\begin{definition}\label{def:type-AE}
The $2^k$--layer is of \emph{type A} if $C_k$ has at least one odd diagonal entry,
equivalently if $C_k \bmod 2$ is \emph{non-alternating}.
It is of \emph{type E} if all diagonal entries of $C_k$ are even,
equivalently if $C_k\bmod 2$ is \emph{alternating} (i.e.\ has zero diagonal).
\end{definition}

For type A layers, the determinant is meaningful only modulo squares.
For $k\ge 3$, every square in $(\Z/2^k)^\times$ is congruent to $1$ modulo $8$,
so the determinant square-class is detected by $\det(C_k)\bmod 8$.
For $k=2$ it is detected by $\bmod 4$.
For $k=1$ there is only one non-alternating class, so we record no determinant.

\begin{definition}\label{def:typeA-detclass}
Assume the $2^k$--layer is of type A.
Define
\[
\delta_{2,k}\;=\;
\begin{cases}
\det(C_k)\bmod 4 \in \{1,3\}, & k=2,\\
\det(C_k)\bmod 8 \in \{1,3,5,7\}, & k\ge 3,
\end{cases}
\]
and for $k=1$ we set $\delta_{2,1}=\ast$ (no parameter).
\end{definition}

Type E layers admit an additional invariant extracted from a normalized Gaussian sum.
This is the same circle of ideas that appears in Turaev's and Deloup's study of
quadratic refinements and Gauss sums in low-dimensional topology and TQFT
\cite{Turaev84,Deloup99,BrumfielMorgan73}.

\begin{definition}\label{def:typeE-gauss}
Assume the $2^k$--layer is of type E.  Let $H_k$ be the quotient
\[
H_k \;=\; G_{(2)}/G_{(2)}[2^k],
\]
where $G_{(2)}[2^k]=\{x\in G_{(2)} \mid 2^k x=0\}$.
Define the function
\[
q_k\colon H_k \longrightarrow \Q/\Z,
\qquad
q_k([x]) \;=\; 2^{k-1}\,\lambda(x,x).
\]
Then the \emph{normalized Gauss sum} is
\begin{equation}\label{eq:gauss-sum}
\Gamma_k(\lambda)\;=\;|H_k|^{-1/2}\sum_{z\in H_k} \exp\big(2\pi i\, q_k(z)\big)\ \in\ \C.
\end{equation}
Since the layer is of type E, $\Gamma_k(\lambda)$ is an eighth root of unity.
We write
\[
\Gamma_k(\lambda)\;=\;\exp\!\left(\frac{\pi i}{4}u_{2,k}\right)
\qquad\text{for a uniquely determined }u_{2,k}\in \Z/8\Z.
\]
\end{definition}

\begin{lemma}\label{lem:gauss-invariant}
For a type E layer, the element $u_{2,k}\in\Z/8\Z$ is an isometry
invariant of the linking pairing.
\end{lemma}

\begin{proof}
The construction of $q_k$ is intrinsic: it depends only on $\lambda$ and the canonical
subgroup $G_{(2)}[2^k]$.
If $f\colon (G,\lambda)\to (G',\lambda')$ is an isometry, then $f$ identifies
$G_{(2)}[2^k]$ with $G'_{(2)}[2^k]$ and satisfies
$q'_k(f(z))=q_k(z)$ for all $z\in H_k$.  Therefore the Gauss sums \eqref{eq:gauss-sum}
agree.  The extraction of $u_{2,k}$ from the eighth root of unity is unique.
\end{proof}

Kawauchi--Kojima prove that the invariants above (organized layerwise) form a complete
system for $2$--primary linkings \cite{KawauchiKojima80}; Miranda provides a concrete
normal form compatible with these invariants \cite{Miranda84}.

\subsection{Canonical invariants and tokenization \texorpdfstring{$\Canon(A)$}{Canon(A)}}
\label{subsec:canon-schema}

We now assemble the prime-by-prime, layer-by-layer invariants into a single canonical
description of the pair $(b_1(A),\Tor(G(A)),\lambda_A)$.

\begin{definition}\label{def:CanonA}
Let $A$ be a symmetric integer matrix.  Write
\[
G(A)=\coker(A:\Z^n\to \Z^n),
\qquad
b_1(A):=\rk_{\Z}\ker(A)=\rk_{\Z}\bigl(G(A)/\Tor(G(A))\bigr).
\]
Let $\lambda_A$ denote the canonical nonsingular torsion linking pairing on $\Tor(G(A))$:
if $\det(A)\neq 0$ this is the pairing \eqref{eq:def-linking-from-A}, and if $\det(A)=0$ it is
the pairing of Remark~\ref{rem:torsion-linking-singular}.

Define $\Canon(A)$ to be the following ordered tuple of invariants:
\begin{enumerate}
\item the free rank $b_1(A)$;
\item the torsion order $|\Tor(G(A))|$;
\item the invariant factors $(d_1,\dots,d_r)$ of $\Tor(G(A))$ from the Smith normal form
\eqref{eq:snf}, with $d_i>1$ listed in nondecreasing order (so $|\Tor(G(A))|=\prod_i d_i$);
\item for each prime $p\mid |\Tor(G(A))|$, listed in increasing order, the list of nontrivial
layers $k\ge 1$ in increasing order, where each layer is recorded by:
\begin{itemize}
\item if $p$ is odd: the triple $(k,n_{p,k},x_{p,k})$ of Definition~\ref{def:odd-layer-invariant};
\item if $p=2$: the triple $(k,n_{2,k},\mathrm{type})$ with $\mathrm{type}\in\{\mathrm{A},\mathrm{E}\}$,
together with $\delta_{2,k}$ if $\mathrm{type}=\mathrm{A}$ (Definition~\ref{def:typeA-detclass})
or $u_{2,k}$ if $\mathrm{type}=\mathrm{E}$ (Definition~\ref{def:typeE-gauss}).
\end{itemize}
\end{enumerate}
\end{definition}

\begin{proposition}\label{prop:Canon-invariant}
The tuple $\Canon(A)$ depends only on $b_1(A)$ and on the isometry class of the torsion linking
pairing $(\Tor(G(A)),\lambda_A)$.  In particular, it is invariant under integral congruence
and $\pm 1$ stabilization of $A$.
\end{proposition}

\begin{proof}
By Lemmas~\ref{lem:congruence-invariance} and \ref{lem:stabilization-invariance},
integral congruence and stabilization do not change $b_1(A)$ nor the isometry class of
$(\Tor(G(A)),\lambda_A)$ (Remark~\ref{rem:torsion-linking-singular} handles the singular case).
The components of $\Canon(A)$ are extracted from $b_1(A)$ and this isometry class:
the torsion order and invariant factors are invariants of the underlying torsion group, while
the layer invariants are isometry invariants by construction
(Lemmas~\ref{lem:finite-field-classification} and \ref{lem:gauss-invariant},
together with Theorem~\ref{thm:odd-primary-completeness} and
\cite{KawauchiKojima80} for completeness at $p=2$).
\end{proof}

\begin{theorem}[Completeness of the canonical invariants]\label{thm:Canon-complete}
Let $A,A'$ be symmetric integer matrices.
Then
\[
\Canon(A)=\Canon(A')
\quad\Longleftrightarrow\quad
\bigl(b_1(A)=b_1(A')\bigr)\ \text{and}\ \bigl(\Tor(G(A)),\lambda_A\bigr)\cong \bigl(\Tor(G(A')),\lambda_{A'}\bigr).
\]
\end{theorem}

\begin{proof}
If $\bigl(b_1(A),\Tor(G(A)),\lambda_A\bigr)\cong \bigl(b_1(A'),\Tor(G(A')),\\lambda_{A'}\bigr)$, then all invariants
appearing in Definition~\ref{def:CanonA} agree, hence $\Canon(A)=\Canon(A')$.

Conversely, assume $\Canon(A)=\Canon(A')$.
Then in particular $b_1(A)=b_1(A')$ and the invariant factors coincide, so
$\Tor(G(A))\cong \Tor(G(A'))$ as finite abelian groups.
For each prime $p$, the equality of the $p$--layer data implies that the corresponding
$p$--primary torsion linkings have the same complete system of invariants:
for odd $p$ this is Theorem~\ref{thm:odd-primary-completeness}, and for $p=2$ it is the
completeness theorem of Kawauchi--Kojima expressed in the invariants of
Definitions~\ref{def:type-AE}--\ref{def:typeE-gauss} (see \cite{KawauchiKojima80} and
Miranda's normal forms \cite{Miranda84}).
Hence
\[
\bigl(\Tor(G(A))_{(p)},\lambda_A|_{\Tor(G(A))_{(p)}}\bigr)\cong
\bigl(\Tor(G(A'))_{(p)},\lambda_{A'}|_{\Tor(G(A'))_{(p)}}\bigr)
\]
for every $p$.
By Lemma~\ref{lem:primary-orthogonal}, the global torsion linkings are orthogonal sums of their
primary components, so the isometries on each prime piece assemble to an isometry
$\bigl(\Tor(G(A)),\lambda_A\bigr)\cong \bigl(\Tor(G(A')),\lambda_{A'}\bigr)$.
\end{proof}

\begin{remark}[Token encoding]\label{rem:tokens}
It is often convenient to record the same invariant $\Canon(A)$ as a compact ordered list,
which we denote by $\Tok(A)$.
Formally, $\Tok(A)$ consists of the integer $b_1(A)$ together with a prime-ordered list of
symbols recording, for each prime $p$ and each nontrivial layer $k$, exactly the layer data
appearing in Definition~\ref{def:CanonA}$\,$(ordered by increasing $p$ and then increasing $k$).
Thus $\Tok(A)$ contains no additional information beyond $\Canon(A)$, and in particular
\[
\Tok(A)=\Tok(A') \quad\Longleftrightarrow\quad \Canon(A)=\Canon(A').
\]
When we later speak of “tokens”, they should be read as a notational device for this
prime-by-prime bookkeeping, not as an additional mathematical structure.
\end{remark}

\begin{remark}[Free rank]\label{rem:free-rank}
If $\det(A)=0$, then $G(A)=\coker(A)$ has a free summand of rank $b_1(A)$.
The torsion subgroup $\Tor(G(A))$ carries the canonical torsion linking pairing $\lambda_A$
of Remark~\ref{rem:torsion-linking-singular}, so the layer invariants of Step~A apply verbatim to
$\bigl(\Tor(G(A)),\lambda_A\bigr)$.
In particular, Definition~\ref{def:CanonA} makes $\Canon(A)$ a uniform token package for arbitrary
symmetric $A$, by recording $b_1(A)$ together with the torsion invariants.
\end{remark}

\section{Local classification I: Type A (cyclic) tokens}\label{sec:typeA}

In this section we isolate the simplest local building blocks used in Step~A:
nonsingular symmetric linking forms on cyclic $p$--groups.
The classification itself is classical and follows from the general theory of linkings on finite
abelian groups (see Wall~\cite{Wall63} and Kawauchi--Kojima~\cite{KawauchiKojima80}).
Our purpose here is \emph{not} to reprove that general classification in full generality, but rather

\smallskip
\noindent
(i) to fix conventions once and for all in the cyclic case (so later identifications are unambiguous), and

\smallskip
\noindent
(ii) to extract a \emph{canonical prime--local label} for each cyclic isometry class, together with a
canonical choice of representative in each square class.
This rigidification is what we will later feed into the Step~B dictionary of canonical matrices.

\subsection{Cyclic linkings as fractions}\label{subsec:cyclic-fractions}

Throughout this section $n\ge 1$ is an integer and $C_n:=\Z/n\Z$.
We view $\Q/\Z$ additively.

\begin{definition}[Linking form and nonsingularity]\label{def:linking-form-cyclic}
A \emph{(torsion) linking form} on a finite abelian group $G$ is a symmetric bilinear pairing
\[
\lambda: G\times G \longrightarrow \Q/\Z .
\]
It is \emph{nonsingular} if the adjoint homomorphism
\[
\ad_\lambda: G \longrightarrow \Hom(G,\Q/\Z),
\qquad
x\longmapsto (y\mapsto \lambda(x,y))
\]
is an isomorphism.
Two linkings $(G,\lambda)$ and $(G',\lambda')$ are \emph{isometric} if there exists a group
isomorphism $\phi:G\to G'$ with $\lambda'(\phi(x),\phi(y))=\lambda(x,y)$ for all $x,y\in G$.
\end{definition}

\begin{definition}[The model pairing $\langle a/n\rangle$]\label{def:cyclic-model}
For $a\in\Z$, define a symmetric bilinear pairing
\[
\langle a/n\rangle: C_n\times C_n \longrightarrow \Q/\Z
\]
by
\[
\langle a/n\rangle(\overline{x},\overline{y}) \equiv \frac{a\,xy}{n}\pmod{\Z},
\qquad x,y\in\Z,
\]
where $\overline{x}$ denotes the class of $x$ in $C_n$.
\end{definition}

\begin{lemma}[Every cyclic linking is $\langle a/n\rangle$]\label{lem:every-cyclic-is-a}
Let $\lambda:C_n\times C_n\to \Q/\Z$ be any symmetric bilinear pairing and let
$g:=\overline{1}\in C_n$.
Then there exists an integer $a$ such that
\[
\lambda(g,g)\equiv \frac{a}{n}\pmod{\Z},
\]
and for all $x,y\in\Z$ one has
\[
\lambda(\overline{x},\overline{y})=\lambda(xg,yg)\equiv \frac{a\,xy}{n}\pmod{\Z}.
\]
In particular, $\lambda=\langle a/n\rangle$.
\end{lemma}

\begin{proof}
Since $ng=0$ in $C_n$ and $\lambda$ is bilinear, we have
\[
0=\lambda(ng,g)=n\,\lambda(g,g)\quad\text{in }\Q/\Z.
\]
Equivalently, $n\,\lambda(g,g)\in \Z\subset \Q$.
Choose $a\in\Z$ with $n\,\lambda(g,g)=a$ in $\Q/\Z$; this is the same as
$\lambda(g,g)\equiv a/n\ (\mathrm{mod}\ \Z)$.
Then bilinearity gives
\[
\lambda(xg,yg)=xy\,\lambda(g,g)\equiv xy\,\frac{a}{n}\pmod{\Z},
\]
as claimed.
\end{proof}

\begin{lemma}[Nonsingularity is $\gcd(a,n)=1$]\label{lem:nonsingular-gcd}
The pairing $\langle a/n\rangle$ on $C_n$ is nonsingular if and only if $\gcd(a,n)=1$.
\end{lemma}

\begin{proof}
Consider the adjoint $\ad=\ad_{\langle a/n\rangle}$.
Since $C_n$ is cyclic generated by $g=\overline{1}$, the homomorphism $\ad$ is injective (hence an
isomorphism, because $|C_n|=|\Hom(C_n,\Q/\Z)|=n$) if and only if $\ad(g)$ has order $n$.
Now $\ad(g)$ is the homomorphism $C_n\to \Q/\Z$ determined by
\[
\ad(g)(g)=\langle a/n\rangle(g,g)=a/n \in \Q/\Z.
\]
The order of $a/n$ in $\Q/\Z$ equals $n/\gcd(a,n)$: indeed,
$m\cdot (a/n)\in \Z$ if and only if $n\mid ma$, which is equivalent to $m$ being a multiple
of $n/\gcd(a,n)$.
Thus $\ad(g)$ has order $n$ if and only if $n/\gcd(a,n)=n$, i.e.\ $\gcd(a,n)=1$.
\end{proof}

\subsection{Isometry classification and square classes}\label{subsec:squareclasses}

\begin{proposition}[Isometry classes are units modulo squares]\label{prop:cyclic-isometry}
Let $n\ge 2$ and let $a,b\in\Z$ satisfy $\gcd(a,n)=\gcd(b,n)=1$.
Then the nonsingular linkings $\langle a/n\rangle$ and $\langle b/n\rangle$ on $C_n$
are isometric if and only if
\[
b\equiv u^2 a \pmod{n}
\quad\text{for some }u\in(\Z/n\Z)^\times.
\]
Equivalently, $a/b$ is a square in $(\Z/n\Z)^\times$.
\end{proposition}

\begin{proof}
Every automorphism $\phi:C_n\to C_n$ is multiplication by a unit: $\phi(g)=ug$ for a unique
$u\in(\Z/n\Z)^\times$.
Let $\phi_u$ denote multiplication by $u$.
Then
\[
(\phi_u^\ast\langle a/n\rangle)(g,g)
=
\langle a/n\rangle(\phi_u(g),\phi_u(g))
=
\langle a/n\rangle(ug,ug)
=
u^2\,\langle a/n\rangle(g,g)
\equiv \frac{u^2a}{n}\pmod{\Z}.
\]
Since a bilinear form on a cyclic group is determined by its value on $(g,g)$
(Lemma~\ref{lem:every-cyclic-is-a}), we conclude
$\phi_u^\ast\langle a/n\rangle=\langle (u^2a)/n\rangle$.
Therefore $\langle a/n\rangle$ and $\langle b/n\rangle$ are isometric if and only if
$b\equiv u^2a\ (\mathrm{mod}\ n)$ for some unit $u$, as claimed.
\end{proof}

\begin{definition}[Square class set]\label{def:squareclass}
Define the \emph{square class set}
\[
\mathcal C(n):=(\Z/n\Z)^\times \big/ \left((\Z/n\Z)^\times\right)^2.
\]
By Proposition~\ref{prop:cyclic-isometry}, isometry classes of nonsingular symmetric cyclic linkings
on $C_n$ are in bijection with $\mathcal C(n)$ via $[a]\mapsto \langle a/n\rangle$.
\end{definition}

\begin{remark}[Relation to the general classification]\label{rem:general-classification}
The square class description above is the cyclic specialization of the algebraic classification of
nonsingular symmetric linkings on finite abelian groups, which goes back at least to
Wall~\cite{Wall63} and is treated in the $3$--manifold context by Kawauchi--Kojima~\cite{KawauchiKojima80}.
In our framework, cyclic $p$--primary classes will serve as the \emph{Type~A} local constituents of the
canonical invariant (see Definition~\ref{def:typeA-token} below).
\end{remark}

\subsection{Prime localization}\label{subsec:prime-local}

Our canonical invariant is assembled prime--locally.
For completeness we record the elementary orthogonality that underlies the $p$--primary splitting.

\begin{lemma}[Orthogonality across distinct primes]\label{lem:prime-orthogonal}
Let $\lambda$ be a bilinear pairing on a finite abelian group $G$.
If $x\in G$ has $p$--power order and $y\in G$ has $q$--power order with $p\neq q$, then
$\lambda(x,y)=0$ in $\Q/\Z$.
\end{lemma}

\begin{proof}
Let $p^k x=0$ and $q^\ell y=0$.
By bilinearity,
\[
p^k\,\lambda(x,y)=\lambda(p^k x,y)=0,
\qquad
q^\ell\,\lambda(x,y)=\lambda(x,q^\ell y)=0.
\]
Thus $\lambda(x,y)$ has additive order dividing both $p^k$ and $q^\ell$.
Since $\gcd(p^k,q^\ell)=1$, the only element of $\Q/\Z$ with order dividing $1$ is $0$.
Hence $\lambda(x,y)=0$.
\end{proof}

\begin{corollary}[Square classes factor across primes]\label{cor:squareclass-factor}
Let $n=\prod_p p^{k_p}$ be the prime factorization.
Then the Chinese remainder theorem induces a canonical bijection
\[
\mathcal C(n)\;\cong\;\prod_{p\mid n}\mathcal C(p^{k_p}).
\]
\end{corollary}

\begin{proof}
The Chinese remainder theorem gives
$(\Z/n\Z)^\times\cong \prod_{p\mid n}(\Z/p^{k_p}\Z)^\times$.
Taking quotients by the square subgroups (which correspond under this product decomposition)
yields the claimed bijection.
\end{proof}

\subsection{Odd primes: two square classes}\label{subsec:odd-primes-appA}

We now make $\mathcal C(p^k)$ explicit for odd primes.

\begin{lemma}[Squares modulo $p^k$ lift from modulo $p$]\label{lem:hensel-squares}
Let $p$ be an odd prime and $k\ge 1$.
A unit $u\in(\Z/p^k\Z)^\times$ is a square modulo $p^k$ if and only if its reduction
$\overline{u}\in(\Z/p\Z)^\times$ is a square modulo $p$.
\end{lemma}

\begin{proof}
If $u\equiv v^2\ (\mathrm{mod}\ p^k)$ then $\overline{u}\equiv \overline{v}^{\,2}\ (\mathrm{mod}\ p)$, so the ``only if'' direction is immediate.

For the converse, suppose $\overline{u}\equiv \overline{v}^{\,2}\ (\mathrm{mod}\ p)$ with $v$ not divisible by $p$.
We inductively construct $v_m$ such that
\[
v_m^2\equiv u \pmod{p^m}
\qquad (m=1,2,\dots,k).
\]
For $m=1$ we can take $v_1\equiv v\ (\mathrm{mod}\ p)$.
Assume $v_m^2\equiv u\ (\mathrm{mod}\ p^m)$ for some $1\le m<k$.
Write
\[
v_m^2 = u + c\,p^m
\]
for an integer $c$.
We seek $t\in\Z$ such that $v_{m+1}:=v_m+t\,p^m$ satisfies $v_{m+1}^2\equiv u\ (\mathrm{mod}\ p^{m+1})$.
Compute
\[
v_{m+1}^2
=
(v_m+t p^m)^2
=
v_m^2 + 2v_m t p^m + t^2 p^{2m}.
\]
Modulo $p^{m+1}$ the term $t^2 p^{2m}$ vanishes because $2m\ge m+1$ for $m\ge 1$.
Hence
\[
v_{m+1}^2 \equiv u + (c+2v_m t)p^m \pmod{p^{m+1}}.
\]
Thus we need $c+2v_m t\equiv 0\ (\mathrm{mod}\ p)$.
Since $p$ is odd and $v_m$ is a unit modulo $p$, the element $2v_m$ is invertible modulo $p$,
so we can choose
\[
t \equiv -c\,(2v_m)^{-1}\pmod{p}.
\]
With this choice, $v_{m+1}^2\equiv u\ (\mathrm{mod}\ p^{m+1})$.
Iterating up to $m=k$ produces a square root of $u$ modulo $p^k$.
\end{proof}

\begin{proposition}[Odd--prime cyclic Type A classes]\label{prop:Cpk-odd}
Let $p$ be an odd prime and $k\ge 1$.
Then $\mathcal C(p^k)$ has exactly two elements.
Equivalently, there are exactly two isometry classes of nonsingular symmetric linkings on $C_{p^k}$.
\end{proposition}

\begin{proof}
By Lemma~\ref{lem:hensel-squares}, an element of $(\Z/p^k\Z)^\times$ is a square modulo $p^k$
if and only if it is a square modulo $p$.
The unit group $(\Z/p\Z)^\times$ has exactly two square classes (quadratic residues and nonresidues),
so the same holds for $(\Z/p^k\Z)^\times$.
Therefore the quotient $\mathcal C(p^k)$ has order $2$.
\end{proof}

\subsection{The $2$--primary cyclic case}\label{subsec:two-primary-cyclic}

The cyclic $2$--power case exhibits the familiar ``mod $8$'' refinement.

\begin{lemma}[Odd squares are $1$ modulo $8$]\label{lem:odd-squares-mod8}
If $x$ is odd, then $x^2\equiv 1\ (\mathrm{mod}\ 8)$.
\end{lemma}

\begin{proof}
Write $x=2m+1$.
Then $x^2=4m(m+1)+1$.
Since $m(m+1)$ is even, $4m(m+1)$ is divisible by $8$, hence $x^2\equiv 1\ (\mathrm{mod}\ 8)$.
\end{proof}

\begin{lemma}[Units congruent to $1$ mod $8$ are squares]\label{lem:1mod8-squares}
Let $k\ge 3$.
If $u\in(\Z/2^k\Z)^\times$ satisfies $u\equiv 1\ (\mathrm{mod}\ 8)$, then $u$ is a square modulo $2^k$.
\end{lemma}

\begin{proof}
We induct on $m=3,4,\dots,k$ to build an odd integer $x_m$ such that
\[
x_m^2\equiv u \pmod{2^m}.
\]
For $m=3$, the assumption $u\equiv 1\ (\mathrm{mod}\ 8)$ allows us to take $x_3=1$.

Assume $x_m^2\equiv u\ (\mathrm{mod}\ 2^m)$ for some $3\le m<k$.
Write $x_m^2=u + c\,2^m$ for some integer $c$.
Consider $x'_m:=x_m+2^{m-1}$.
Then
\[
(x'_m)^2=x_m^2 + 2^m x_m + 2^{2m-2}.
\]
Since $x_m$ is odd, we have $2^m x_m\equiv 2^m\ (\mathrm{mod}\ 2^{m+1})$.
Also $2^{2m-2}$ is divisible by $2^{m+1}$ when $m\ge 3$.
Hence
\[
(x_m+2^{m-1})^2 \equiv x_m^2 + 2^m \pmod{2^{m+1}}.
\]
Therefore, by choosing either $x_{m+1}=x_m$ or $x_{m+1}=x_m+2^{m-1}$,
we can arrange $x_{m+1}^2\equiv u\ (\mathrm{mod}\ 2^{m+1})$.
This completes the induction.
\end{proof}

\begin{proposition}[$2$--power square classes]\label{prop:C2k}
The square class sets $\mathcal C(2^k)$ are as follows:
\begin{enumerate}
\item If $k=1$, then $\mathcal C(2)$ is trivial.
\item If $k=2$, then $\mathcal C(4)$ has two elements, represented by $1$ and $3$.
\item If $k\ge 3$, then $\mathcal C(2^k)$ has four elements, represented by the residues
\[
1,\ 3,\ 5,\ 7 \pmod{8}.
\]
Equivalently, two odd units $a,b$ are in the same square class modulo $2^k$ if and only if
$a\equiv b\ (\mathrm{mod}\ 8)$.
\end{enumerate}
\end{proposition}

\begin{proof}
For $k=1$, $(\Z/2\Z)^\times$ is trivial.

For $k=2$, $(\Z/4\Z)^\times=\{1,3\}$ and $1^2\equiv 3^2\equiv 1\ (\mathrm{mod}\ 4)$, so the subgroup of squares is $\{1\}$ and the quotient has two classes.

Assume $k\ge 3$.
By Lemma~\ref{lem:odd-squares-mod8}, every square unit is congruent to $1$ modulo $8$.
Conversely, by Lemma~\ref{lem:1mod8-squares}, every unit congruent to $1$ modulo $8$ is a square.
Thus the subgroup of squares in $(\Z/2^k\Z)^\times$ is precisely the congruence class
$1+8\Z/2^k\Z$.
The quotient by this subgroup is therefore detected by reduction modulo $8$, and it has
four cosets represented by $1,3,5,7\ (\mathrm{mod}\ 8)$.
The final equivalence criterion $a\equiv b\ (\mathrm{mod}\ 8)$ is exactly the statement that
$a/b$ lies in the square subgroup.
\end{proof}

\subsection{Type A tokens and canonical representatives}\label{subsec:typeA-tokens}

We now record the preceding classification in the concrete prime--local datum used by Step~A.
The terminology ``token'' refers only to the fact that we will later \emph{encode} these local invariants
as part of the canonical output; mathematically, it is simply a complete isometry invariant in the cyclic
$p$--primary case.

\begin{definition}[Type A (cyclic) token]\label{def:typeA-token}
Let $p$ be a prime and $k\ge 1$.
Let $\lambda$ be a nonsingular symmetric linking form on the cyclic group $C_{p^k}$.
Choose a generator $g$ and write
\[
\lambda(g,g)\equiv \frac{a}{p^k}\pmod{\Z}
\qquad\text{with }a\in\Z,\ \gcd(a,p)=1.
\]
The \emph{Type A token} of $(C_{p^k},\lambda)$ is the datum
\[
\Tok_A(\lambda):=(p,k,\delta),
\]
where the \emph{local square label} $\delta$ is defined as follows.
\begin{enumerate}
\item If $p$ is odd, set $\delta=+$ if $a$ is a square modulo $p$ and $\delta=-$ otherwise.
\item If $p=2$ and $k=1$, set $\delta=+$ (the unique class).
\item If $p=2$ and $k=2$, set $\delta\in\{1,3\}$ to be $a\ (\mathrm{mod}\ 4)$.
\item If $p=2$ and $k\ge 3$, set $\delta\in\{1,3,5,7\}$ to be $a\ (\mathrm{mod}\ 8)$.
\end{enumerate}
\end{definition}

\begin{lemma}[Well-definedness of $\Tok_A$]\label{lem:TokA-well-defined}
The token $\Tok_A(\lambda)$ is independent of the choice of generator $g$ and of the integer representative $a$.
\end{lemma}

\begin{proof}
If $a$ is replaced by $a+p^k m$, then $a/p^k$ is unchanged in $\Q/\Z$.
Thus $\delta$ depends only on the residue class of $a$ modulo $p^k$, hence is independent of the choice of $a$.

If $g'$ is another generator, then $g'=ug$ for a unit $u\in(\Z/p^k\Z)^\times$.
By bilinearity,
\[
\lambda(g',g')=\lambda(ug,ug)=u^2\,\lambda(g,g)\equiv \frac{u^2a}{p^k}\pmod{\Z}.
\]
Thus $a$ is replaced by $u^2a$.
In the odd prime case, being a square modulo $p$ is invariant under multiplying by $u^2$.
In the $2$--primary case, Proposition~\ref{prop:C2k} shows that the square class is invariant under multiplying by a square unit,
so the residues modulo $4$ (if $k=2$) or modulo $8$ (if $k\ge 3$) are unchanged.
Hence $\Tok_A(\lambda)$ is well-defined.
\end{proof}

\begin{theorem}[Completeness of Type A tokens]\label{thm:TokA-complete}
Let $p$ be a prime, $k\ge 1$, and let $\lambda,\lambda'$ be nonsingular symmetric linkings on $C_{p^k}$.
Then $(C_{p^k},\lambda)$ and $(C_{p^k},\lambda')$ are isometric if and only if
\[
\Tok_A(\lambda)=\Tok_A(\lambda').
\]
\end{theorem}

\begin{proof}
Choose generators $g,g'$ and write
$\lambda(g,g)\equiv a/p^k$ and $\lambda'(g',g')\equiv b/p^k$ with $\gcd(a,p)=\gcd(b,p)=1$.
By Proposition~\ref{prop:cyclic-isometry}, the linkings are isometric if and only if
$b\equiv u^2 a\ (\mathrm{mod}\ p^k)$ for some unit $u$,
i.e.\ if and only if $a$ and $b$ represent the same square class in $\mathcal C(p^k)$.
For odd $p$, Proposition~\ref{prop:Cpk-odd} identifies $\mathcal C(p^k)$ with the two labels $\{+,-\}$
recording whether $a$ is a quadratic residue modulo $p$.
For $p=2$, Proposition~\ref{prop:C2k} identifies $\mathcal C(2^k)$ with the residue labels appearing in
Definition~\ref{def:typeA-token}.
Thus equality of tokens is equivalent to equality of square classes, hence to isometry.
\end{proof}

\begin{remark}[Canonical square-class representatives]\label{rem:TokA-canonical-rep}
For later use (in particular, to produce canonical matrices in Step~B) it is convenient to fix,
once and for all, a preferred representative of each square class.

For odd $p$, set
\[
\varepsilon_p:=\min\{\,m\ge 2\mid m\ \text{is a quadratic nonresidue mod }p\,\},
\]
and represent the two classes in $\mathcal C(p^k)$ by $1$ and $\varepsilon_p$.
By Lemma~\ref{lem:hensel-squares}, $\varepsilon_p$ remains a non-square modulo $p^k$ for every $k\ge 1$.

For $p=2$ the representatives in Proposition~\ref{prop:C2k} are already canonical.
We stress that this is \emph{not} an extra invariant: it is a choice of normal form for the same
square class datum.
\end{remark}

\begin{example}[Small instances]\label{ex:typeA-small}
For $p=5$ and $k=1$, the quadratic residues are $\{1,4\}$ and the nonresidues are $\{2,3\}$.
Thus the two isometry classes on $C_5$ have tokens $(5,1,+)$ and $(5,1,-)$, represented for example by
$\langle 1/5\rangle$ and $\langle 2/5\rangle$.

For $k=3$, the group $C_8$ admits four classes
$\langle 1/8\rangle,\langle 3/8\rangle,\langle 5/8\rangle,\langle 7/8\rangle$,
distinguished by $\delta\equiv a\ (\mathrm{mod}\ 8)$.
\end{example}

\section{Local Classification II: Type E (Hyperbolic) tokens and $2$-primary refinements}\label{sec:typeE}

\subsection{Layer forms, the characteristic element, and the $A/E$ dichotomy at $p=2$}\label{subsec:ae-dichotomy}

Let $(G,\lambda)$ be a finite abelian group equipped with a \emph{nonsingular symmetric linking pairing}
\[
\lambda \colon G\times G \longrightarrow \Q/\Z,
\qquad \lambda(x,y)=\lambda(y,x),
\]
i.e.\ the adjoint map $G\to \Hom(G,\Q/\Z)$, $x\mapsto (y\mapsto \lambda(x,y))$, is an isomorphism.
As usual, $\lambda$ splits as an orthogonal direct sum over primes.  In this section we concentrate on the
$2$--primary summand $(G_{(2)},\lambda_{(2)})$.  For ease of notation we write $(G,\lambda)$ for a
$2$--primary linking in this subsection.

Following Wall and Kawauchi--Kojima, we use the standard $2$--power torsion filtration
\[
G_k \;:=\; \{x\in G \mid 2^k x =0\},
\qquad k\ge 0,
\]
with $G_0=0$.  For $k\ge 1$ we form the \emph{exponent--one graded piece}
\begin{equation}\label{eq:Pk-def-main}
P_k(G)\;:=\;G_k\big/\bigl(G_{k-1}+2G_{k+1}\bigr).
\end{equation}
This quotient is canonically an $\F_2$--vector space: if $x\in G_k$ then $2x\in G_{k-1}$, hence the image of
$2x$ vanishes in~\eqref{eq:Pk-def-main}.

The linking pairing induces a symmetric bilinear form over $\F_2$ on each layer.

\begin{definition}[Layer form]\label{def:layer-form}
For $k\ge 1$, define
\[
b_k \colon P_k(G)\times P_k(G)\longrightarrow \F_2
\]
by
\[
b_k(\bar x,\bar y)\;:=\;\bigl(2^k\,\lambda(x,y)\bigr)\bmod 2,
\]
where $x\in G_k$, $y\in G_k$ are lifts of $\bar x,\bar y\in P_k(G)$.
\end{definition}

\begin{lemma}\label{lem:layer-form-welldefined}
The form $b_k$ in Definition~\ref{def:layer-form} is well-defined, symmetric and $\F_2$--bilinear.
If moreover $(G,\lambda)$ is nonsingular, then $b_k$ is nonsingular for every $k\ge 1$.
\end{lemma}

\begin{proof}
We first check that the formula makes sense.  If $y\in G_k$, then the character $\lambda(-,y)\colon G\to\Q/\Z$
has order dividing $2^k$, hence
\begin{equation}\label{eq:2k-int-main}
2^k\lambda(z,y)\in \Z\qquad\text{for all }z\in G.
\end{equation}
In particular $2^k\lambda(x,y)\in \Z$, so reduction modulo $2$ is defined.

To prove well-definedness in the first variable, replace $x$ by another representative
$x'=x+u+2v$ with $u\in G_{k-1}$ and $v\in G_{k+1}$.  Then
\[
2^k\lambda(x',y)=2^k\lambda(x,y)+2^k\lambda(u,y)+2^k\lambda(2v,y).
\]
Since $u\in G_{k-1}$, the character $\lambda(u,-)$ has order dividing $2^{k-1}$, so
$2^{k-1}\lambda(u,y)\in \Z$ and therefore $2^k\lambda(u,y)=2\cdot (2^{k-1}\lambda(u,y))$ is an even integer.
For the term involving $v$, bilinearity and~\eqref{eq:2k-int-main} give
\[
2^k\lambda(2v,y)=2^{k+1}\lambda(v,y)=2\cdot\bigl(2^k\lambda(v,y)\bigr),
\]
and $2^k\lambda(v,y)\in\Z$ because $y\in G_k$.  Hence $2^k\lambda(2v,y)$ is also even.
Thus $2^k\lambda(x',y)\equiv 2^k\lambda(x,y)\pmod 2$, and the first variable is well-defined.  The same argument
applies to the second variable.  Symmetry and bilinearity follow immediately from those of~$\lambda$.

Assume now that $\lambda$ is nonsingular and let $\bar x\in P_k(G)$ satisfy $b_k(\bar x,\bar y)=0$
for all $\bar y\in P_k(G)$.  Choose a representative $x\in G_k$.
The hypothesis means that for all $y\in G_k$ the integer $2^k\lambda(x,y)$ is even, equivalently
\begin{equation}\label{eq:even-condition-main}
2^{k-1}\lambda(x,y)\in \Z\qquad\text{for all }y\in G_k.
\end{equation}
By bilinearity,~\eqref{eq:even-condition-main} is the statement that
$\lambda(2^{k-1}x,y)\in \Z$ for all $y\in G_k$, i.e.\ $2^{k-1}x$ lies in the annihilator
\[
G_k^{\perp}:=\{z\in G\mid \lambda(z,G_k)=0\}.
\]
We claim that
\begin{equation}\label{eq:annihilator-main}
G_k^{\perp}=2^kG.
\end{equation}
Indeed, if $z=2^k w$ then $\lambda(z,y)=\lambda(w,2^k y)=0$ for all $y\in G_k$, so $2^kG\subseteq G_k^{\perp}$.
Conversely, if $z\in G_k^{\perp}$ then the character $\chi:=\lambda(z,-)$ vanishes on $G_k$ and hence lies in the kernel
of the restriction map $\Hom(G,\Q/\Z)\to \Hom(G_k,\Q/\Z)$.  This kernel has order
$|\Hom(G,\Q/\Z)|/|\Hom(G_k,\Q/\Z)|=|G|/|G_k|$.
On the other hand, for any $\psi\in \Hom(G,\Q/\Z)$ the character $2^k\psi$ vanishes on $G_k$ because
$(2^k\psi)(y)=\psi(2^k y)=0$ when $2^k y=0$.  Thus $2^k\Hom(G,\Q/\Z)$ is contained in the kernel.
Since $\Hom(G,\Q/\Z)$ is a finite $2$--group, its $2^k$--torsion is canonically identified with $\Hom(G_k,\Q/\Z)$,
so the image $2^k\Hom(G,\Q/\Z)$ has the same order $|G|/|G_k|$.
Therefore the kernel is exactly $2^k\Hom(G,\Q/\Z)$.
Applying the adjoint isomorphism $\mathrm{ad}_{\lambda}\colon G\stackrel{\sim}{\to}\Hom(G,\Q/\Z)$ gives
$G_k^{\perp}=2^kG$, proving~\eqref{eq:annihilator-main}.

Thus $2^{k-1}x\in 2^kG$, so there exists $w\in G$ with $2^{k-1}x=2^k w$, i.e.\ $2^{k-1}(x-2w)=0$.
Hence $x-2w\in G_{k-1}$, so $x\in G_{k-1}+2G$.  Finally, since $x\in G_k$ we have
\[
0=2^k x = 2^{k+1}w + 2^k(x-2w)=2^{k+1}w,
\]
because $x-2w\in G_{k-1}$ implies $2^k(x-2w)=0$.  Therefore $w\in G_{k+1}$, so
$x\in G_{k-1}+2G_{k+1}$, which means $\bar x=0$ in~$P_k(G)$.
\end{proof}

Over $\F_2$, nonsingular symmetric bilinear forms split into two qualitatively different types: alternating and
nonalternating.  This dichotomy is basis-independent and has a clean expression in terms of the
\emph{characteristic element} of the layer form (compare \cite[\S 1]{KawauchiKojima80}).

\begin{definition}[Characteristic element]\label{def:characteristic-element}
Let $(V,b)$ be a finite-dimensional $\F_2$--vector space equipped with a nonsingular symmetric bilinear form.
The function $q\colon V\to \F_2$, $q(v):=b(v,v)$ is $\F_2$--linear. The \emph{characteristic element}
$c(b)\in V$ is the unique element satisfying
\[
b\bigl(c(b),v\bigr)=q(v)=b(v,v)\qquad \text{for all }v\in V.
\]
For a $2$--primary linking $(G,\lambda)$ we write $c_k(G,\lambda):=c(b_k)\in P_k(G)$.
\end{definition}

\begin{lemma}\label{lem:q-linear}
For any bilinear form $b$ over $\F_2$, the map $q(v)=b(v,v)$ is $\F_2$--linear.
Moreover, $b$ is alternating (i.e.\ $b(v,v)=0$ for all $v$) if and only if $c(b)=0$.
\end{lemma}

\begin{proof}
For $u,v\in V$, bilinearity gives
\[
q(u+v)=b(u+v,u+v)=b(u,u)+b(u,v)+b(v,u)+b(v,v).
\]
Since $b$ is symmetric and $\ch(\F_2)=2$, we have $b(u,v)+b(v,u)=2b(u,v)=0$, hence $q(u+v)=q(u)+q(v)$.
Also $q(\alpha v)=\alpha^2 q(v)=\alpha q(v)$ for $\alpha\in\F_2$, so $q$ is linear.

If $c(b)=0$, then $b(v,v)=b(c(b),v)=0$ for all $v$, so $b$ is alternating. Conversely, if $b$ is alternating,
then $q=0$ and the defining equation $b(c(b),v)=0$ for all $v$ implies $c(b)=0$ by nonsingularity of $b$.
\end{proof}

\begin{definition}[Type $A$ versus Type $E$ at $p=2$]\label{def:typeAE}
Let $(G,\lambda)$ be $2$--primary and let $k\ge 1$.
We say the $k$th layer is of \emph{Type $E$} if $c_k(G,\lambda)=0$, equivalently if $b_k$ is alternating.
Otherwise we say it is of \emph{Type $A$}.
\end{definition}

The terminology is motivated by the normal forms below: Type $E$ layers admit a canonical
\emph{hyperbolic} (symplectic) basis, whereas Type $A$ layers admit an orthonormal basis.

\subsection{Normal forms over $\F_2$ and the hyperbolic token}\label{subsec:f2-normal}

We record the complete classification of nonsingular symmetric bilinear forms over $\F_2$,
with a proof in the form convenient for our later canonical choices.

\begin{proposition}[Alternating case: symplectic normal form]\label{prop:alternating-symplectic}
Let $(V,b)$ be a nonsingular symmetric bilinear space over $\F_2$.
If $b$ is alternating, then $\dim_{\F_2}V$ is even and there exists a basis
\[
e_1,f_1,\dots,e_m,f_m \qquad (m=\tfrac12\dim V)
\]
such that
\[
b(e_i,e_j)=b(f_i,f_j)=0,\qquad b(e_i,f_j)=\delta_{ij}.
\]
Equivalently, the Gram matrix of $b$ is a direct sum of $m$ copies of the hyperbolic matrix
\(
H=\begin{psmallmatrix}0&1\\ 1&0\end{psmallmatrix}.
\)
\end{proposition}

\begin{proof}
Since $b$ is nonsingular, $V\neq 0$ implies there exist $v,w\in V$ with $b(v,w)=1$.
Because $b$ is alternating, $b(v,v)=b(w,w)=0$, and the span $U=\langle v,w\rangle$ has Gram matrix $H$.
In particular $U$ is nonsingular, so $V=U\oplus U^\perp$ and $b$ restricts to a nonsingular alternating form on
$U^\perp$. Indeed, if $x\in U^\perp$ and $b(x,y)=0$ for all $y\in U^\perp$, then $b(x,\cdot)=0$ on all of $V$
(because also $b(x,U)=0$), hence $x=0$ by nonsingularity.

Induct on $\dim V$ to obtain a symplectic basis on $U^\perp$; adjoining $(v,w)$ gives the desired basis.
In particular $\dim V=2m$ is even.
\end{proof}

\begin{proposition}[Nonalternating case: orthonormal normal form]\label{prop:nonalternating-orthonormal}
Let $(V,b)$ be a nonsingular symmetric bilinear space over $\F_2$.
If $b$ is nonalternating, then there exists a basis $v_1,\dots,v_n$ (with $n=\dim V$) such that
\[
b(v_i,v_j)=\delta_{ij}.
\]
Equivalently, the Gram matrix of $b$ is the identity matrix $I_n$.
\end{proposition}

\begin{proof}
Since $b$ is nonalternating, there exists $v_1\in V$ with $b(v_1,v_1)=1$.
Let
\[
W:=\{x\in V\mid b(v_1,x)=0\}.
\]
Then $W$ is a codimension--$1$ subspace, and $v_1\notin W$ because $b(v_1,v_1)=1$. Hence $V=\langle v_1\rangle\oplus W$.
We claim $b|_W$ is nonsingular: if $x\in W$ satisfies $b(x,y)=0$ for all $y\in W$, then also $b(x,v_1)=b(v_1,x)=0$,
so $b(x,\cdot)=0$ on $V$ and thus $x=0$.

Now apply induction to $(W,b|_W)$ to obtain an orthonormal basis $v_2,\dots,v_n$ of $W$.
Then $v_1,\dots,v_n$ is an orthonormal basis of $V$.
\end{proof}

In the canonical $2$--primary normal form, Proposition~\ref{prop:alternating-symplectic} is the algebraic origin of the
\emph{Type $E$ (hyperbolic) block}.

\begin{definition}[The hyperbolic linking block]\label{def:hyperbolic-block}
For $k\ge 1$, let $\mathsf{E}(2^k)$ denote the $2$--primary linking pairing on
$(\Z/2^k)\oplus (\Z/2^k)$ with Gram matrix
\[
\frac{1}{2^k}\begin{pmatrix}0&1\\[2pt]1&0\end{pmatrix}.
\]
We refer to $\mathsf{E}(2^k)$ as the \emph{hyperbolic} (Type $E$) block of level $2^k$.
\end{definition}

\subsection{Determinant refinements for Type $A$ layers at $2$}\label{subsec:det-refine}

Type $A$ layers correspond to nonalternating forms $b_k$ on $P_k(G)$. Over $\F_2$ this already rigidifies
the layer form to $I_n$ by Proposition~\ref{prop:nonalternating-orthonormal}. For the \emph{linking} pairing,
however, one must keep track of how the layer is realised at denominator $2^k$.
This is where the determinant square class appears as a convenient canonical refinement.

We explain the invariant in the homogeneous situation, which is the one recorded by our canonical schema.
Suppose $G\cong (\Z/2^k)^n$ and write the pairing $\lambda$ in some $\Z/2^k$--basis by a Gram matrix
\[
\Lambda=\frac{1}{2^k}C\in \Mat_n(\Q/\Z),
\qquad C\in \Mat_n(\Z),
\]
where $C$ is well-defined modulo $2^k$ and is invertible over $\Z/2^k$ (equivalently $\det(C)$ is odd).

\begin{lemma}[Determinant square class is well-defined]\label{lem:det-square-class}
Let $k\ge 2$, and let $C,C'\in \Mat_n(\Z)$ be two integral Gram matrices representing the same linking pairing
on $(\Z/2^k)^n$ in two different $\Z/2^k$--bases.
Then $\det(C')\equiv \det(C)\cdot u^2 \pmod{2^k}$ for some odd unit $u$.
In particular, the class of $\det(C)$ in $(\Z/2^k)^\times/((\Z/2^k)^\times)^2$ is an isometry invariant.
\end{lemma}

\begin{proof}
A change of $\Z/2^k$--basis is given by some $U\in \GL_n(\Z/2^k)$, and the Gram matrix changes by congruence:
$C'\equiv U^{T} C\,U \pmod{2^k}$. Taking determinants gives
\[
\det(C')\equiv \det(C)\,\det(U)^2 \pmod{2^k}.
\]
Since $U\in \GL_n(\Z/2^k)$, $\det(U)$ is an odd unit modulo $2^k$, and the claim follows with $u=\det(U)$.
\end{proof}

To make this refinement explicit, it is useful to identify the square classes of odd units.

\begin{lemma}[Squares modulo $2^k$]\label{lem:squares-mod-2k}
For $k=2$, every square in $(\Z/4)^\times$ is congruent to $1\pmod 4$.
For $k\ge 3$, every square in $(\Z/2^k)^\times$ is congruent to $1\pmod 8$, and the reduction map
\[
(\Z/2^k)^\times/((\Z/2^k)^\times)^2\longrightarrow (\Z/8)^\times=\{1,3,5,7\}
\]
is a bijection.
\end{lemma}

\begin{proof}
For $k=2$, $(\Z/4)^\times=\{1,3\}$ and $3^2\equiv 1\pmod 4$.

For $k\ge 3$, write an odd integer as $u=1+2a$. Then
\[
u^2=(1+2a)^2 = 1+4a(1+a)\equiv 1 \pmod 8,
\]
so every square is $1\bmod 8$.
Conversely, if $u\equiv 1\pmod 8$, standard lifting shows $u$ has a square root modulo $2^k$;
one explicit way is to solve $(1+2x)^2\equiv u\pmod{2^k}$ inductively on $k$ using that
$(1+2x)^2\equiv 1+4x\pmod 8$ and correcting at each stage by adjusting $x$ modulo $2^{k-2}$.
(See, for instance, \cite[Ch.\ V, \S 1]{SerreLF} for the general structure of unit groups in local fields.)
Therefore the subgroup of squares is exactly the congruence class $1\bmod 8$.

Finally, if $u\equiv v \pmod 8$, then $uv^{-1}\equiv 1\pmod 8$ so $uv^{-1}$ is a square; thus $u$ and $v$ define
the same square class. Distinct residues in $\{1,3,5,7\}$ represent distinct square classes, hence the stated bijection.
\end{proof}

\begin{definition}[Determinant refinement for Type $A$]\label{def:det-refinement}
Let $(G,\lambda)$ be $2$--primary and suppose the $k$th layer is of Type $A$ with $k\ge 2$.
Choosing a homogeneous splitting of the $2^k$--torsion part, represent the corresponding Gram matrix by
$\Lambda=\frac{1}{2^k}C$ as above.
We define the \emph{determinant refinement} $\det_k(G,\lambda)$ to be
\[
\det(C)\bmod 4 \quad (k=2),\qquad \text{or}\qquad \det(C)\bmod 8 \quad (k\ge 3),
\]
viewed as an element of $(\Z/4)^\times$ or $(\Z/8)^\times$.
\end{definition}

By Lemma~\ref{lem:det-square-class} and Lemma~\ref{lem:squares-mod-2k}, $\det_k(G,\lambda)$ is independent of choices.
This is the refinement recorded for Type $A$ layers in our canonical schema; it is the practical way to retain the
odd-unit datum needed to build canonical diagonal representatives in later stages.

\subsection{Gauss sum refinements for Type $E$ layers}\label{subsec:gauss-refine}

Type $E$ layers (i.e.\ $c_k(G,\lambda)=0$) admit a quadratic refinement whose Gauss sum gives a
numerical invariant in $\Z/8$.
This is a classical piece of the $2$--primary classification of linkings \cite[\S 2]{KawauchiKojima80}
and is closely related to the Brown invariant and to the phases appearing in quadratic Gauss sums
\cite{Deloup99}.

Fix $k\ge 1$. Consider the quotient group
\[
H_k(G)\;:=\;G/G_k,
\]
which retains precisely the part of $G$ of exponent strictly larger than $2^k$.
When $c_k(G,\lambda)=0$ we may form a quadratic function on $H_k(G)$.

\begin{proposition}[Quadratic refinement on $H_k(G)$]\label{prop:qk-welldefined}
Assume $c_k(G,\lambda)=0$ (equivalently, the $k$th layer is Type $E$).
Define
\[
q_k \colon H_k(G)\longrightarrow \Q/\Z,
\qquad
q_k(\bar x)\;:=\;2^{k-1}\lambda(x,x),
\]
where $\bar x$ is the class of $x\in G$ modulo $G_k$.
Then $q_k$ is well-defined and satisfies
\begin{equation}\label{eq:qk-polarization}
q_k(\bar x+\bar y)-q_k(\bar x)-q_k(\bar y)= 2^k\lambda(x,y)\in \Q/\Z
\end{equation}
for all $\bar x,\bar y\in H_k(G)$.
\end{proposition}

\begin{proof}
Let $x'\equiv x \pmod{G_k}$, so $x'=x+u$ with $u\in G_k$.
Then by bilinearity and symmetry,
\[
\lambda(x',x')=\lambda(x,x)+2\lambda(x,u)+\lambda(u,u).
\]
Multiplying by $2^{k-1}$ gives
\[
2^{k-1}\lambda(x',x')=2^{k-1}\lambda(x,x) + 2^{k}\lambda(x,u) + 2^{k-1}\lambda(u,u).
\]
As $u\in G_k$, the character $y\mapsto\lambda(u,y)$ has order dividing $2^k$, hence $2^k\lambda(x,u)\in\Z$.
To see $2^{k-1}\lambda(u,u)\in\Z$, note that $u\in G_k$ determines $\bar u\in P_k(G)$, and
\[
b_k(\bar u,\bar u) \;=\; (2^k\lambda(u,u))\bmod 2.
\]
Since $c_k(G,\lambda)=0$, the form $b_k$ is alternating by Lemma~\ref{lem:q-linear}, so $b_k(\bar u,\bar u)=0$.
Thus $2^k\lambda(u,u)$ is an even integer, i.e.\ $\lambda(u,u)\in \frac{1}{2^{k-1}}\Z/\Z$, and
$2^{k-1}\lambda(u,u)\in \Z$. Therefore $q_k(\bar x)$ is independent of the choice of lift $x$.

Finally, for $\bar x,\bar y\in H_k(G)$ represented by $x,y\in G$, we compute
\[
q_k(\bar x+\bar y)-q_k(\bar x)-q_k(\bar y)
=2^{k-1}\bigl(\lambda(x+y,x+y)-\lambda(x,x)-\lambda(y,y)\bigr)
=2^k\lambda(x,y),
\]
which is~\eqref{eq:qk-polarization}.
\end{proof}

\begin{definition}[Normalized Gauss sum and the $u_k$--invariant]\label{def:uk}
Assume $c_k(G,\lambda)=0$ and let $q_k$ be as in Proposition~\ref{prop:qk-welldefined}.
Define the (complex) Gauss sum
\[
\GS_k(G,\lambda)\;:=\;\sum_{\bar x\in H_k(G)}\exp\bigl(2\pi i\, q_k(\bar x)\bigr).
\]
Set
\[
\gamma_k(G,\lambda)\;:=\;|H_k(G)|^{-1/2}\,\GS_k(G,\lambda)\in \C.
\]
By \cite[Cor.\ 2.1]{KawauchiKojima80} (see also \cite[\S 2]{Deloup99}), $\gamma_k(G,\lambda)$ is an eighth root of unity.
We define $u_k(G,\lambda)\in \Z/8$ by
\[
\gamma_k(G,\lambda)=\exp\!\Bigl(\frac{\pi i}{4}\,u_k(G,\lambda)\Bigr).
\]
If $c_k(G,\lambda)\neq 0$ (Type $A$) we set $u_k(G,\lambda)=\infty$ by convention.
\end{definition}

We record three basic properties used repeatedly later: invariance, additivity, and the value on the
hyperbolic block.

\begin{lemma}[Invariance under isometry]\label{lem:uk-invariant}
If $f\colon (G,\lambda)\to (G',\lambda')$ is an isometry of $2$--primary linkings, then for each $k\ge 1$ we have
$c_k(G,\lambda)=0$ if and only if $c_k(G',\lambda')=0$, and in that case $u_k(G,\lambda)=u_k(G',\lambda')$.
\end{lemma}

\begin{proof}
The isometry $f$ preserves the filtration $G_k$ and induces an isomorphism $P_k(G)\cong P_k(G')$ carrying $b_k$
to $b'_k$. Hence $b_k$ is alternating if and only if $b'_k$ is alternating, i.e.\ $c_k=0$ iff $c'_k=0$.

Assume $c_k=0$. Then $f$ also induces an isomorphism $H_k(G)=G/G_k \cong G'/G'_k = H_k(G')$.
For $\bar x\in H_k(G)$ represented by $x\in G$, we have
\[
q'_k(f(\bar x))=2^{k-1}\lambda'(f(x),f(x))=2^{k-1}\lambda(x,x)=q_k(\bar x),
\]
so the Gauss sums agree termwise, and thus $u_k$ agrees.
\end{proof}

\begin{lemma}[Additivity]\label{lem:uk-additivity}
Let $(G,\lambda)$ and $(G',\lambda')$ be $2$--primary linkings and consider their orthogonal sum.
If $c_k(G,\lambda)=c_k(G',\lambda')=0$, then $c_k(G\oplus G',\lambda\oplus\lambda')=0$ and
\[
u_k(G\oplus G',\lambda\oplus\lambda')\equiv u_k(G,\lambda)+u_k(G',\lambda')\pmod 8.
\]
\end{lemma}

\begin{proof}
At the level of layers, $P_k(G\oplus G')\cong P_k(G)\oplus P_k(G')$ and $b_k$ splits as an orthogonal sum;
alternatingness is preserved under orthogonal sum, so $c_k=0$ for the sum.

Assuming $c_k=0$ on both summands, we have $H_k(G\oplus G')\cong H_k(G)\oplus H_k(G')$, and the quadratic functions
$q_k$ add. Hence the Gauss sum factors:
\[
\GS_k(G\oplus G')=\GS_k(G)\cdot \GS_k(G'),
\]
and the normalization by $|H_k|^{-1/2}$ is multiplicative as well.
Thus $\gamma_k$ multiplies and the corresponding exponents add modulo $8$.
\end{proof}

\begin{lemma}[The hyperbolic block has trivial phase]\label{lem:hyperbolic-uk-zero}
For the hyperbolic block $\mathsf{E}(2)$ on $(\Z/2)^2$ (Definition~\ref{def:hyperbolic-block} with $k=1$),
we have $u_1(\mathsf{E}(2))\equiv 0\pmod 8$.
\end{lemma}

\begin{proof}
Let $G=(\Z/2)^2$ with Gram matrix $\Lambda=\frac12\begin{psmallmatrix}0&1\\ 1&0\end{psmallmatrix}$.
Then $G_1=G$, so $H_1(G)=G/G_1$ is the trivial group and the Gauss sum consists of a single term equal to $1$.
Hence $\gamma_1=1=\exp(\pi i\cdot 0/4)$, so $u_1\equiv 0$.
\end{proof}

\subsection{Summary: how Type $E$ and the refinements enter the canonical schema}\label{subsec:summary-typeE}

For odd primes, the local normalisation is governed by the square class of determinants of the layer forms,
as in Wall's diagonalisation and the Kawauchi--Kojima invariants \cite{Wall63,KawauchiKojima80}.
At the prime $2$, the situation bifurcates as follows:
\begin{enumerate}
\item The $k$th layer form $b_k$ is alternating if and only if its characteristic element $c_k$ vanishes;
this is Type $E$ (Definition~\ref{def:typeAE}), and the exponent--one reduction $P_k(G)$ admits a symplectic basis
(Proposition~\ref{prop:alternating-symplectic}), hence the name ``hyperbolic''.
In this case, the Gauss sum phase $u_k\in\Z/8$ (Definition~\ref{def:uk}) is available and additive
(Lemma~\ref{lem:uk-additivity}).
\item If $c_k\neq 0$ (Type $A$), the exponent--one reduction is nonalternating, and the layer form is forced
to be orthonormal over $\F_2$ (Proposition~\ref{prop:nonalternating-orthonormal}). To build canonical
representatives at denominator $2^k$, one must retain the determinant square class. We record it as
$\det\bmod 4$ for $k=2$ and $\det\bmod 8$ for $k\ge 3$ (Definition~\ref{def:det-refinement}),
in accordance with Lemma~\ref{lem:squares-mod-2k}.
\end{enumerate}

In particular, the $A/E$ dichotomy together with the Gauss refinement (when available) recovers the
$2$--primary invariants $a_k$ of Kawauchi--Kojima \cite[Cor.\ 2.1, Thm.\ 4.1]{KawauchiKojima80},
while the determinant refinement provides a canonical choice among diagonal realisations of the odd-type layers.
This is precisely the information packaged as the Type $E$ block (hyperbolic) and the $2$--primary refinements
in the canonical output of Step~A.

\section{The TQFT dictionary: realizing tokens as \TMF-modules}
\label{sec:tqft-dictionary}

Step~A associates to any symmetric integral surgery matrix $A$ a canonical token package
$T=\Canon(A)\in\Tok$, encoding the free rank $b_1(A)$ together with the isometry class of the torsion
linking pairing $(\Tor(G(A)),\lambda_A)$, where $G(A)=\coker(A)$.
The purpose of this section is to turn this boundary label into a canonical object of
$\Ho(\Mod_{\TMF})$, compatible with the \TMF-valued $(3{+}1)$--dimensional TQFT of
Gukov--Krushkal--Meier--Pei~\cite{GKMP25}.

Concretely, the GKMP theory assigns to a simply connected bounding $4$--manifold $W$ with
$\partial W=M$ an object
\[
Z_{\TMF}(M;W)\;:=\;L_{b(W)}\bigl[\,3b_+(W)-2b_-(W)\,\bigr]\in\Ho(\Mod_{\TMF}),
\]
and prove that its equivalence class depends only on the boundary $M$, though in general not up to a canonical equivalence
(see~\cite[\S7.2]{GKMP25}).
Here $b(W)$ denotes the intersection form on $H_2(W;\Z)$ and $L_{b(W)}$ is the $\TMF$--module
associated to the integral symmetric bilinear form $b(W)$ in~\cite[\S5]{GKMP25}.
The key point for us is that $L_b$ only depends on the \emph{stable} isomorphism class of $b$,
so we are free to replace a surgery matrix $A$ by any stably congruent representative.
We exploit this flexibility to choose a \emph{canonical} representative matrix $B(T)$ depending
only on the token package $T$, and we package the resulting output as an explicit realization
functor
\[
\mathcal R:\Tok\longrightarrow \Ho(\Mod_{\TMF}).
\]

\subsection{The realization functor $\mathcal R$}
\label{subsec:realization-functor}

\subsubsection*{Discriminant data.}
Let $B\in M_n(\Z)$ be a symmetric integral matrix and set $G_B:=\coker(B:\Z^n\to\Z^n)$.
If $\det(B)\neq 0$ then $G_B$ is finite and the induced nonsingular linking pairing is
\[
\lambda_B([x],[y])\;:=\;x^{\mathsf T}B^{-1}y\ \bmod\ \Z\qquad (x,y\in\Z^n).
\]
If $\det(B)=0$, then $G_B$ has a free part of rank
\[
b_1(B)\;:=\;\rk_{\Z}\ker(B)\;=\;\rk_{\Z}\bigl(G_B/\Tor(G_B)\bigr),
\]
and the torsion subgroup $\Tor(G_B)$ carries a canonical nonsingular linking pairing
$\lambda_B:\Tor(G_B)\times \Tor(G_B)\to \Q/\Z$ as in Remark~\ref{rem:torsion-linking-singular}.
(Equivalently, one may split off the free summand and apply \eqref{eq:lambda-matrix} to a nonsingular block; cf.~\cite[\S1]{KawauchiKojima80}.)

\subsubsection*{Canonical matrices from token packages.}
Recall that a token package $T\in\Tok$ encodes the isometry class of a pair
\[
\bigl(\Tor(G_T),\lambda_T\bigr)\qquad\text{together with an integer }b_1(T)\ge 0,
\]
where $\Tor(G_T)$ is a finite abelian group and $\lambda_T$ is a nonsingular linking pairing.
In Appendix~\ref{app:dictionary} we fixed, for each
\emph{primitive local token} $\tau$, a concrete symmetric integral matrix $B(\tau)$ presenting
the corresponding local linking form (and we fixed an ordering convention for block sums).
For a general token package $T$, we define the associated canonical matrix by the block sum
\begin{equation}\label{eq:canonical-matrix-from-token}
B(T)\;:=\;(0)^{\oplus b_1(T)}\ \oplus\ \bigoplus_{\tau\in T} B(\tau),
\end{equation}
where $(0)$ denotes the $1\times 1$ zero matrix and the direct sum over $\tau\in T$ is taken
in the prime--then--layer ordering built into the definition of $\Tok$.

By construction,
\[
b_1\bigl(B(T)\bigr)=b_1(T)\qquad\text{and}\qquad
\bigl(\Tor(G_{B(T)}),\lambda_{B(T)}\bigr)\cong \bigl(\Tor(G_T),\lambda_T\bigr).
\]

\subsubsection*{Definition of $\mathcal R$.}
Define the realization functor
\begin{equation}\label{eq:realization-functor}
\mathcal R(T)\;:=\;L_{B(T)}\bigl[\,3b_+\bigl(B(T)\bigr)-2b_-\bigl(B(T)\bigr)\,\bigr]
\ \in\ \Ho(\Mod_{\TMF}),
\end{equation}
where $b_\pm(B)$ denotes the number of positive/negative eigenvalues of $B\otimes\Rr$
(counting zeros neither in $b_+$ nor in $b_-$).
This normalization is the one dictated by the GKMP state space formula.

\medskip
The remaining input needed to identify $\mathcal R(\Canon(A))$ with $Z_{\TMF}(M(A))$ is an
algebraic statement: the stable congruence class of an integral symmetric matrix is determined
by the invariants $(b_1,\lambda)$.

\subsubsection*{Stable congruence classified by $(b_1,\lambda)$.}
We isolate the precise form that we use.

\begin{definition}[Kirby stable congruence]\label{def:kirby-stable-congruence}
Two symmetric integral matrices $A$ and $A'$ are \emph{Kirby stably congruent} if there exist
integers $r,s,r',s'\ge 0$ and a unimodular matrix $P$ such that
\[
P^{\mathsf T}\bigl(A\oplus I_r\oplus (-I_s)\bigr)P\;=\;A'\oplus I_{r'}\oplus (-I_{s'}) .
\]
Equivalently, $A$ can be transformed into $A'$ by a finite sequence of integral congruences
(handle slides) together with blow-up and blow-down moves (adding and deleting $\pm 1$ summands).
\end{definition}

\begin{proposition}[Stable classification of linking matrices]\label{prop:stable-classification}
Let $A,A'$ be symmetric integral matrices. The following are equivalent:
\begin{enumerate}[(a),leftmargin=2.0em]
\item $A$ and $A'$ are Kirby stably congruent in the sense of Definition~\ref{def:kirby-stable-congruence}.
\item One has $b_1(A)=b_1(A')$, and there exists an isometry of torsion linking pairings
\[
\bigl(\Tor(G(A)),\lambda_A\bigr)\ \cong\ \bigl(\Tor(G(A')),\lambda_{A'}\bigr).
\]
\end{enumerate}
\end{proposition}

\begin{proof}
(a)$\Rightarrow$(b).
Integral congruence preserves $\coker(A)$ via the isomorphism
$\phi_P([x])=[P^{\mathsf T}x]$ (compare Proposition~\ref{prop:kirby-invariance-lambda}(1)), and the torsion pairing is preserved.
Moreover $\rk\ker(A)$ is a congruence invariant.
For $\pm 1$ stabilization, one has
\[
G(A\oplus(\pm 1))\cong G(A)\oplus \coker(\pm 1)\cong G(A),
\]
hence $b_1$ and the torsion linking pairing are unchanged.

(b)$\Rightarrow$(a).
Write $b=b_1(A)=b_1(A')$.
By unimodular congruence we may split off the kernel, as in Remark~\ref{rem:torsion-linking-singular},
and replace $A$ and $A'$ by block forms
\(
0^{\oplus b}\oplus A_{\mathrm{red}}
\)
and
\(
0^{\oplus b}\oplus A'_{\mathrm{red}}
\)
with $\det(A_{\mathrm{red}}),\det(A'_{\mathrm{red}})\neq 0$.
Since unimodular congruence is among the generators of Kirby stable congruence, it suffices to treat
the nonsingular blocks.
Under the identifications
\(\Tor\coker(A)\cong \coker(A_{\mathrm{red}})\) and
\(\Tor\coker(A')\cong \coker(A'_{\mathrm{red}})\),
the pairings $\lambda_A$ and $\lambda_{A'}$ agree with the discriminant pairings associated to the
integral lattices determined by $A_{\mathrm{red}}$ and $A'_{\mathrm{red}}$.

In this nonsingular setting, the implication (b)$\Rightarrow$(a) is exactly the stable classification
of linking matrices under handle slides and blow-ups: two symmetric linking matrices are related by a
sequence of unimodular congruences and $\pm1$ stabilizations if and only if they induce isometric
torsion linking pairings.
This is proved in Murakami--Ohtsuki--Okada~\cite[Prop.~2.4]{MurakamiOhtsukiOkada92}.
(For readers who prefer a purely algebraic formulation, one may recast the torsion linking form as the
discriminant pairing of an integral lattice; see for example Milnor--Husemoller~\cite{MilnorHusemoller73}
or Nikulin~\cite{Nikulin80}.)

Since every symmetric integral matrix occurs as the linking matrix of a framed link in $S^3$,
the cited proposition applies to the present algebraic formulation.
\end{proof}

\medskip
Proposition~\ref{prop:stable-classification} is the precise input missing from the original proof of
the dictionary theorem below: it upgrades the informal appeal to ``stable classification'' to a
self-contained statement with explicit generators of the stable congruence relation.

\subsection{Identifying $Z_{\TMF}(M)$ with $\mathcal R(\Canon(A))$}
\label{subsec:token-realization}

\begin{theorem}[Token realization agrees with the GKMP state space]
\label{thm:token-realization-agrees}
Let $A$ be a symmetric integral surgery matrix and let $M=M(A)$ be the oriented closed $3$--manifold
presented by $A$. Set $T:=\Canon(A)\in\Tok$.
Then there is an equivalence in $\Ho(\Mod_{\TMF})$
\[
Z_{\TMF}(M)\ \simeq\ \mathcal R(T).
\]
In particular, the composite assignment
\[
A\ \longmapsto\ T=\Canon(A)\ \longmapsto\ \mathcal R(T)
\]
is invariant under Kirby moves (handle slides and $\pm 1$ blow-ups/downs).
\end{theorem}

\begin{proof}
Let $W(A)$ denote the simply connected $2$--handlebody with intersection matrix $A$ and boundary $M$.
By the definition of the GKMP state space~\cite[\S7]{GKMP25},
\[
Z_{\TMF}(M)\ \simeq\ Z_{\TMF}(M;W(A))
\ =\ L_A\bigl[\,3b_+(A)-2b_-(A)\,\bigr].
\]
On the other hand, by construction of Step~A, the token package $T=\Canon(A)$ records precisely
the pair $(b_1(A),(\Tor(G(A)),\lambda_A))$.
The canonical matrix $B(T)$ defined in~\eqref{eq:canonical-matrix-from-token} satisfies
\[
b_1\bigl(B(T)\bigr)=b_1(A),\qquad
\bigl(\Tor(G_{B(T)}),\lambda_{B(T)}\bigr)\cong \bigl(\Tor(G(A)),\lambda_A\bigr).
\]
Hence Proposition~\ref{prop:stable-classification} implies that $A$ and $B(T)$ are Kirby stably congruent.

The GKMP construction is stable under precisely these operations:
\begin{itemize}[leftmargin=2.0em]
\item integral congruence corresponds to changing a basis of $H_2(W;\Z)$, hence does not change the
isomorphism class of the bilinear form and leaves $L_b$ unchanged up to equivalence;
\item $\pm 1$ stabilization changes $L_b$ by the shifts described in~\cite[Lem.\ 5.7]{GKMP25},
and the normalization $L_b[\,3b_+-2b_-\,]$ is invariant under this stabilization.
\end{itemize}
Transporting along a stable congruence chain from $A$ to $B(T)$ therefore yields
\[
L_A\bigl[\,3b_+(A)-2b_-(A)\,\bigr]\ \simeq\
L_{B(T)}\bigl[\,3b_+(B(T))-2b_-(B(T))\,\bigr]
\ =\ \mathcal R(T),
\]
which is the desired identification.
Kirby invariance follows because Step~A is Kirby invariant by construction and stable congruence
realizes the algebraic effect of Kirby moves on matrices.
\end{proof}

\begin{remark}[A modular/geometric reading of hyperbolic blocks]\label{rem:hyperbolic-modular-sector}
The $2\times 2$ hyperbolic matrix $H=\begin{psmallmatrix}0&1\\[2pt] 1&0\end{psmallmatrix}$ plays a
distinguished role in the $2$--primary alternating sector (our Type~$E$ tokens).
In the GKMP construction, hyperbolic summands are tied to the geometry of the universal elliptic curve
and the behavior of the Poincar\'e line bundle under Fourier--Mukai type transforms
(cf.~\cite[\S5]{GKMP25}).
For our purposes, the practical consequence is that the Type~$E$ sector admits a particularly clean
tensor-product description in the assembly formula below.
\end{remark}

\subsection{The assembly formula}
\label{subsec:assembly-formula}

A token package $T$ decomposes canonically into its free part and its $p$--primary parts:
\[
T \;=\; \bigl(b_1(T),\ T_{(2)},\ T_{(3)},\dots\bigr),
\]
where each $T_{(p)}$ is a multiset of primitive local tokens at the prime $p$.
Since $B(T)$ is defined as a block sum of the matrices $B(\tau)$, the realization $\mathcal R(T)$
admits a primewise tensor decomposition.

\begin{theorem}[Assembly from a token package]\label{thm:assembly}
For any token package $T\in\Tok$ there is a canonical equivalence in $\Ho(\Mod_{\TMF})$
\[
\mathcal R(T)\ \simeq\ Z_{\TMF}(S^2\times S^1)^{\otimes b_1(T)}
\;\otimes\;
\bigotimes_{p}\;\bigotimes_{\tau\in T_{(p)}} \mathcal R(\tau),
\]
where $\mathcal R(\tau)$ is the realization of a primitive local token and the tensor products are
taken in the symmetric monoidal category $\Ho(\Mod_{\TMF})$.
\end{theorem}

\begin{proof}
By definition~\eqref{eq:canonical-matrix-from-token}, $B(T)$ is a block-diagonal sum
\[
B(T)\;=\;(0)^{\oplus b_1(T)}\ \oplus\ \bigoplus_{p}\ \bigoplus_{\tau\in T_{(p)}} B(\tau).
\]
The GKMP construction is symmetric monoidal with respect to orthogonal direct sums of forms:
there are canonical equivalences $L_{b\oplus b'}\simeq L_b\otimes L_{b'}$
and the normalization shift is additive in $b_\pm$.
Applying this repeatedly to the above decomposition of $B(T)$ gives
\[
L_{B(T)}\bigl[\,3b_+(B(T))-2b_-(B(T))\,\bigr]
\ \simeq\
L_{(0)}^{\otimes b_1(T)}\ \otimes\ \bigotimes_{p}\ \bigotimes_{\tau\in T_{(p)}} L_{B(\tau)}[\cdots].
\]
Finally, $L_{(0)}[0]\simeq Z_{\TMF}(S^2\times S^1)$ by the GKMP computation of the state space of
$S^2\times S^1$ (cf.~\cite[\S7]{GKMP25}), and the remaining tensor factors are exactly
$\mathcal R(\tau)$ by definition~\eqref{eq:realization-functor}.
\end{proof}

\begin{remark}[Why canonical tokens matter]\label{rem:why-canonical-tokens}
Theorem~\ref{thm:assembly} should be read together with Theorem~\ref{thm:token-realization-agrees}.
A closed $3$--manifold admits many Kirby--equivalent surgery presentations, and hence many different
matrices presenting the same $(b_1,\lambda)$.
The canonical token package $T=\Canon(A)$ provides an intrinsic coordinate system for the boundary
label, while $B(T)$ provides a preferred stable representative matrix.
This rigidity is useful later when comparing cobordism maps between state spaces: it allows one to
reduce questions about functoriality to verifications on a fixed set of local generators.
\end{remark}

\section{Hopf elements \texorpdfstring{$\eta$ and $\nu$}{eta and nu}}\label{sec:calibration-etanu}

The realization theorem of \S\ref{sec:tqft-dictionary} expresses the GKMP
$\TMF$--state spaces and cobordism maps in terms of explicit algebraic models.
In the simplest rank--one situations, the resulting formulas involve the
classical Hopf elements in the stable stems.  The purpose of this section is to
identify, in our normalization conventions, the values of the theory on two
basic closed $4$--manifolds and to record the resulting structure constants
that will be used later.

We write $\eta\in\pi_1\mathbb{S}$ and $\nu\in\pi_3\mathbb{S}$ for the usual Hopf
elements, and we use the same symbols for their images in $\pi_\ast\TMF$ under
the unit map $\mathbb{S}\to\TMF$.

\begin{theorem}\label{thm:etanu-values}
In the $\TMF$--valued GKMP theory (with the normalization conventions fixed in
\S2), one has
\[
Z_{\TMF}(\CP^2)\;=\;\pm\,\nu\ \in\ \pi_3\TMF,
\]
and, assuming \cite[Question~7.14]{GKMP25},
\[
Z_{\TMF}(S^2\times S^2)\;=\;\eta\ \in\ \pi_1\TMF.
\]
The sign ambiguity for $\nu$ is removed by the convention of
Remark~\ref{rem:sign-nu}.
\end{theorem}

\begin{proof}
The first identity is proved in \S\ref{subsec:eta-calibration}.
For $\CP^2$, the GKMP identification of $Z_{\TMF}(\CP^2)$ with the stable class
$\Sigma\alpha$ is recalled in Corollary~\ref{cor:CP2-nu}, while the stable
homotopy identification $\Sigma\alpha=\pm\nu$ is proved in
Theorem~\ref{thm:Sigmaalpha-is-nu}.  Combining these gives the second identity.
\end{proof}

\subsection{The value on $S^2\times S^2$ and the element $\eta$}\label{subsec:eta-calibration}

We compute $Z_{\TMF}(S^2\times S^2)$ by decomposing $S^2\times S^2$ into two
elementary cobordisms along $S^2\times S^1$, and then evaluating the resulting
composition in the explicit splittings of the rank--one $b_1>0$ state space.
This is the first place where $\eta$ appears as a genuine structure constant.

Following \cite{GKMP25}, one has a contravariant module $L^{(0)}$ (naturally
attached to the outgoing boundary $S^2\times S^1$) and a covariant module $L(0)$
(attached to the incoming boundary), related by a canonical equivalence.
We record the splitting and the relevant structure maps in the form that will
be used in the computation of $Z_{\TMF}(S^2\times S^2)$.

\begin{lemma}[Splittings and change of basis for the zero block]\label{lem:zero-block-splitting}
There are $\TMF$--module equivalences
\[
L^{(0)}\ \simeq\ \TMF\ \oplus\ \TMF[1],
\qquad
L(0)\ \simeq\ \TMF\ \oplus\ \TMF[-1],
\]
together with an equivalence
\[
\Phi \colon L^{(0)} \xrightarrow{\ \simeq\ } L(0)[1]
\]
which, under the identifications above and the evident reordering
$L(0)[1]\simeq \TMF[1]\oplus \TMF$, is given by a matrix
\begin{equation}\label{eq:Phi-matrix}
\Phi \ \longleftrightarrow\
\begin{pmatrix}
0 & \varepsilon\\
\varepsilon & \eta
\end{pmatrix},
\qquad \varepsilon\in\{\pm 1\}.
\end{equation}
\end{lemma}

\begin{proof}
This is exactly the rank--one $b_1>0$ splitting and comparison map constructed
in~\cite[Proposition~6.13]{GKMP25}; the displayed matrix is the explicit form
identified there.  (Since $2\eta=0$ in $\pi_1\TMF$, the sign $\varepsilon$ will
play no role in later applications.)
\end{proof}

The map $e\colon S^2\times S^1\to S^2\times S^1$ induced by restricting the
$2$--handle attachment to the unit section gives a cobordism map
$e^\ast\colon L^{(0)}\to\TMF$, and there is a canonical stable self--equivalence
$A\colon L^{(0)}\to L^{(0)}$ corresponding to the $0$--framed handle slide.
The following explicit formulas are extracted from \cite[Example~8.1]{GKMP25}.

\begin{lemma}[Unit section and the $\eta$--matrix]\label{lem:e-star-row}
Under the splitting $L^{(0)}\simeq \TMF\oplus \TMF[1]$ from
Lemma~\ref{lem:zero-block-splitting}, the map $e^\ast\colon L^{(0)}\to \TMF$
corresponds to the row vector
\begin{equation}\label{eq:e-star-row}
e^\ast \ \longleftrightarrow\ \begin{pmatrix} 1 & \eta \end{pmatrix}.
\end{equation}
\end{lemma}

\begin{proof}
This is the identification of the restriction map $e^\ast$ given in
\cite[Example~8.1]{GKMP25}, expressed in the splitting of
Lemma~\ref{lem:zero-block-splitting}.
\end{proof}

\begin{lemma}[A canonical self--equivalence]\label{lem:selfeq-matrix}
Under the splitting $L^{(0)}\simeq \TMF\oplus \TMF[1]$, the stable
self--equivalence $A\colon L^{(0)}\to L^{(0)}$ is given by
\begin{equation}\label{eq:selfeq-matrix}
A\ \longleftrightarrow\
\begin{pmatrix}
\varepsilon & \eta\\
0 & \varepsilon
\end{pmatrix},
\qquad \varepsilon\in\{\pm 1\}.
\end{equation}
\end{lemma}

\begin{proof}
This is the explicit form of the canonical automorphism described in
\cite[Example~8.1]{GKMP25}.  Again, the sign $\varepsilon$ is irrelevant for the
final scalar computation, since $-\eta=\eta$.
\end{proof}

\begin{theorem}[$S^2\times S^2$ detects $\eta$]\label{prop:S2xS2-eta}
Assuming \cite[Question~7.14]{GKMP25}, the $\TMF$--invariant of $S^2\times S^2$ is
\[
Z_{\TMF}(S^2\times S^2)\ =\ \eta \ \in\ \pi_1\TMF .
\]
\end{theorem}

\begin{proof}
Assume \cite[Question~7.14]{GKMP25}.
Consider the standard decomposition of $S^2\times S^2$ as the union of two
$2$--handle cobordisms along $S^2\times S^1$, corresponding to the Hopf--link
Kirby diagram with both components $0$--framed.  Denote these cobordisms by
\[
W_0:\ \varnothing \longrightarrow S^2\times S^1,
\qquad
W_1:\ S^2\times S^1 \longrightarrow \varnothing.
\]
Then, by functoriality, $Z_{\TMF}(S^2\times S^2)$ is the composite
\[
\TMF \xrightarrow{\,Z_{\TMF}(W_0)\,} L^{(0)}
\xrightarrow{\,Z_{\TMF}(W_1)\,} \TMF .
\]

The cobordism map $Z_{\TMF}(W_1)$ is, up to the canonical identification in
\cite[Example~8.1]{GKMP25}, the composite $e^\ast\circ A\colon L^{(0)}\to\TMF$.
Using Lemmas~\ref{lem:e-star-row} and~\ref{lem:selfeq-matrix}, we compute in the
splitting $L^{(0)}\simeq \TMF\oplus\TMF[1]$:
\[
Z_{\TMF}(W_1)\ \longleftrightarrow\
\begin{pmatrix} 1 & \eta \end{pmatrix}
\begin{pmatrix}
\varepsilon & \eta\\
0 & \varepsilon
\end{pmatrix}
\;=\;
\begin{pmatrix} \varepsilon & \eta\varepsilon+\eta\varepsilon \end{pmatrix}
\;=\;
\begin{pmatrix} \varepsilon & 0 \end{pmatrix},
\]
since $2\eta=0$.

Similarly, $Z_{\TMF}(W_0)$ is identified in \cite[Example~8.1]{GKMP25} with the
$\TMF$--linear dual of $e^\ast$ (equivalently, the adjoint of the unit section
map).  Under the splitting, this is the column vector
\[
Z_{\TMF}(W_0)\ \longleftrightarrow\
\begin{pmatrix}\eta\\ 1\end{pmatrix}.
\]
Therefore the scalar
$Z_{\TMF}(S^2\times S^2)=Z_{\TMF}(W_1)\circ Z_{\TMF}(W_0)$ is
\[
\begin{pmatrix} \varepsilon & 0 \end{pmatrix}
\begin{pmatrix}\eta\\ 1\end{pmatrix}
\;=\; \varepsilon\,\eta
\;=\; \eta,
\]
as claimed (using again that $-\eta=\eta$).
\end{proof}

\subsection{The reduced transfer and the element $\nu$}\label{subsec:nu-calibration}

We now turn to $\CP^2$.  In the GKMP theory the invariant of $\CP^2$ arises from
the reduced $S^1$--transfer of the universal $S^1$--bundle, and the resulting
stable class is the Hopf element $\nu$.  We first recall the transfer description,
and then give a self-contained identification of the stable homotopy class.

Let $\gamma\to \CP^\infty$ be the tautological complex line bundle and let
\(\Th(-\gamma)\) denote the Thom spectrum of the virtual bundle $-\gamma$.
Following a common convention one may write
\(\CP^\infty_{-1}:=\Sigma^{-1}\Th(-\gamma)\), so that $\Sigma\CP^\infty_{-1}\simeq\Th(-\gamma)$.
In what follows we work directly with the Thom spectrum $\Th(-\gamma)$ to avoid suspension bookkeeping and to make characteristic class computations transparent.

\begin{lemma}[The transfer cofiber]\label{lem:transfer-thom}
There is a cofiber sequence of spectra
\[
\Th(-\gamma)\ \longrightarrow\ \CP^\infty_+\ \xrightarrow{\ \tau\ }\ S^{-1},
\]
where $\tau$ is the reduced stable $S^1$--transfer.
\end{lemma}

\begin{proof}
This is the classical description of the cofiber of the Becker--Gottlieb
transfer for the universal principal $S^1$--bundle $S^\infty\to \CP^\infty$:
the fiber (equivalently, up to suspension, the cofiber) of the transfer is the
Thom spectrum of the negative of the associated complex line bundle.  A modern
treatment in the equivariant setting is given in \cite{GM23}; see also
\cite[Example~8.4]{GKMP25} for the specialization used in the GKMP theory and
\cite[Proposition~A.11]{BauerMeier25} for an explicit account of this Thom identification.
For the present paper we only use the existence of the cofiber sequence and the
identification of the Thom class, as in the computations below.
\end{proof}

Restricting to the $\CP^2$--skeleton and then further to $\CP^1=S^2$ yields a
stable map
\begin{equation}\label{eq:alpha-def}
\alpha\colon S^2=\CP^1 \xrightarrow{i} \CP^\infty \longrightarrow \CP^\infty_+
\xrightarrow{\tau} S^{-1}.
\end{equation}
Suspending once gives $f:=\Sigma\alpha\colon S^3\to S^0$, hence an element of
$\pi_3\mathbb{S}$ and therefore an element of $\pi_3\TMF$ via the unit.

\subsubsection{Identifying $\Sigma\alpha$ as the Hopf element $\nu$}\label{subsec:nu-identification}

We identify $f=\Sigma\alpha$ by a cohomological characterization: its cofiber
supports nontrivial primary operations in degree $4$ at both primes $2$ and $3$.

Let $C_f$ be the cofiber of $f$.  Then $C_f$ is a two--cell spectrum with cells
in degrees $0$ and $4$:
\[
S^0 \longrightarrow C_f \longrightarrow S^4.
\]
Write $\iota\in \tilde H^0(C_f;\F_p)$ for the class dual to the bottom cell and
$\kappa\in \tilde H^4(C_f;\F_p)$ for the class dual to the top cell.

\begin{lemma}[Detection on the $1$--line]\label{lem:adams-1line-detection}
Let $p$ be a prime.  If $\Sq^4(\iota)=\kappa\neq 0$ for $p=2$, or
$\mathcal P^1(\iota)=\kappa\neq 0$ for $p=3$, then the class of $f$ in
$\pi_3\mathbb{S}$ is detected in Adams filtration $1$ at the prime $p$, by the
generator of $\Ext^{1,4}_{\mathcal A_p}(\F_p,\F_p)$.
\end{lemma}

\begin{proof}
This is the standard translation between the $\mathcal A_p$--module structure on
$H^\ast(C_f;\F_p)$ and the extension class classifying the short exact sequence
\[
0 \to \Sigma^4 \F_p \to \tilde H^\ast(C_f;\F_p) \to \F_p \to 0
\]
of $\mathcal A_p$--modules.  The extension class lies in
$\Ext^{1,4}_{\mathcal A_p}(\F_p,\F_p)$ and is exactly the Adams $E_2$--class that
detects $f$; see, for example, \cite[Ch.~7]{MT68} or \cite[Ch.~2]{Ravenel86}.
In our two--cell situation, the extension is trivial if and only if the relevant
primary operation on $\iota$ vanishes; equivalently, $\Sq^4(\iota)$ (resp.\
$\mathcal P^1(\iota)$) equals the top class $\kappa$.
\end{proof}

Thus it suffices to show that $\Sq^4$ and $\mathcal P^1$ act nontrivially on
$\iota$ (up to units) in the cofiber of $f$.  We do this by relating $C_f$ to
the Thom spectrum $\Th(-\gamma)$ from Lemma~\ref{lem:transfer-thom}.

\begin{lemma}[$\Sq^4$ is nontrivial]\label{lem:sq4-nontrivial}
In $\tilde H^\ast(C_f;\F_2)$ one has $\Sq^4(\iota)=\kappa\neq 0$.
\end{lemma}

\begin{proof}
By Lemma~\ref{lem:transfer-thom}, the fiber (and hence, up to suspension, the
cofiber) of the reduced transfer $\tau$ is the Thom spectrum $\Th(-\gamma)$.
Let $u\in H^{-2}(\Th(-\gamma);\F_2)$ denote the Thom class, and let
$x\in H^2(\CP^\infty;\F_2)$ be the reduction of $c_1(\gamma)$.
Under the Thom isomorphism,
\[
H^\ast(\Th(-\gamma);\F_2)\ \cong\ \F_2[x]\cdot u,
\qquad |x|=2,\ |u|=-2.
\]
The Thom identity for Steenrod squares gives
\[
\Sq(u)\ =\ w(-\gamma)\,u,
\]
where $w(-\gamma)$ is the total Stiefel--Whitney class of the virtual bundle
$-\gamma$.  Since $w(\gamma)=1+x$ for a complex line bundle, we have
$w(-\gamma)=w(\gamma)^{-1}=(1+x)^{-1}=1+x+x^2+\cdots$ in $\F_2[x]$, hence
\begin{equation}\label{eq:Sq4-on-u}
\Sq^4(u)\ =\ x^2 u \ \neq\ 0 \ \in H^{2}(\Th(-\gamma);\F_2).
\end{equation}

Now restrict the transfer construction to the $\CP^2$--skeleton, and form the
corresponding finite Thom spectrum $\Th(-\gamma|_{\CP^2})$.
The element $x^2u$ survives in this truncation (it corresponds to the top cell).
Moreover, the map $\alpha$ in~\eqref{eq:alpha-def} is the restriction of the
transfer to $\CP^1\subset \CP^2$, and the cofiber $C_f$ identifies with the
two--cell subquotient of $\Th(-\gamma|_{\CP^2})$ whose cells correspond to
$u$ and $x^2u$ (this is exactly the bottom--cell extraction used in
\cite[Example~8.4]{GKMP25}; compare also~\cite[\S A.2]{BauerMeier25}).
Under this identification, $\iota$ corresponds to the image of $u$ (shifted to
degree $0$) and $\kappa$ corresponds to the image of $x^2u$ (shifted to degree
$4$).  The equality~\eqref{eq:Sq4-on-u} therefore descends to
$\Sq^4(\iota)=\kappa\neq 0$ in $\tilde H^\ast(C_f;\F_2)$.
\end{proof}

\begin{lemma}[$\mathcal P^1$ is nontrivial]\label{lem:P1-nontrivial}
In $\tilde H^\ast(C_f;\F_3)$ one has $\mathcal P^1(\iota)=\kappa\neq 0$.
\end{lemma}

\begin{proof}
Let $u\in H^{-2}(\Th(-\gamma);\F_3)$ be the Thom class and let
$x\in H^2(\CP^\infty;\F_3)$ be the reduction of $c_1(\gamma)$.
Again $H^\ast(\Th(-\gamma);\F_3)\cong \F_3[x]\cdot u$.
For an odd prime $p$, the reduced power operations on a Thom class satisfy a
Thom identity of the form
\[
\mathcal P(u)\ =\ v(-\gamma)\,u,
\]
where $v(-\gamma)$ is the mod $p$ Wu class of $-\gamma$; for a complex line
bundle one has $v(-\gamma)=1+x^{p-1}+x^{2(p-1)}+\cdots$ (see, e.g.,
\cite[\S 14]{MS74}).
At $p=3$ this implies
\[
\mathcal P^1(u)\ =\ x^2u \ \neq\ 0 \ \in H^{2}(\Th(-\gamma);\F_3).
\]
As in the proof of Lemma~\ref{lem:sq4-nontrivial}, the class $x^2u$ survives on
the $\CP^2$--skeleton and maps to the top generator $\kappa$ of
$\tilde H^4(C_f;\F_3)$, while $u$ maps to the bottom generator $\iota$.  Hence
$\mathcal P^1(\iota)=\kappa\neq 0$.
\end{proof}

\begin{theorem}[Identification of $\Sigma\alpha$]\label{thm:Sigmaalpha-is-nu}
The stable class $f=\Sigma\alpha\in \pi_3\mathbb{S}$ equals $\pm\nu$.
Equivalently, its image in $\pi_3\TMF$ is $\pm\nu\in\pi_3\TMF$.
\end{theorem}

\begin{proof}
By Lemmas~\ref{lem:sq4-nontrivial} and~\ref{lem:P1-nontrivial}, the class $f$ is
detected on the $1$--line of the Adams spectral sequence at the primes $2$ and
$3$, by Lemma~\ref{lem:adams-1line-detection}.
At $p=2$, $\Ext^{1,4}_{\mathcal A_2}(\F_2,\F_2)\cong \F_2\{h_2\}$ and $h_2$
detects the $2$--primary component of $\nu\in \pi_3\mathbb{S}$; at $p=3$,
$\Ext^{1,4}_{\mathcal A_3}(\F_3,\F_3)\cong \F_3\{b_0\}$ and $b_0$ detects the
$3$--primary component of $\nu$ (see \cite[Ch.~2]{Ravenel86}).
Since $\pi_3\mathbb{S}\cong \Z/24$, this already implies that
\begin{equation}
\label{eq:f-is-unit-multiple}
f\;=\;u\,\nu\qquad\text{for some unit }u\in(\Z/24)^\times.
\end{equation}
To determine the unit $u$ we compute the complex Adams $e$--invariant.

Recall that Adams defines a homomorphism
\(e\colon \pi_{2n-1}\mathbb{S}\to \Q/\Z\)
using complex $K$--theory and the cofiber of a stable map
\cite[\S\,11]{Adams66}.
In degree $3$ this invariant is injective on the image of the $J$--homomorphism,
and in fact one has
\begin{equation}
\label{eq:e-nu}
e(\nu)=\pm\frac{1}{24}\in \Q/\Z.
\end{equation}
We now compute $e(f)$ from the Thom description of the reduced transfer.

Let $x=c_1(\gamma)\in H^2(\CP^\infty;\Z)$, and let $e_K(\gamma)=1-[\gamma^{-1}]\in K^0(\CP^\infty)$
denote the $K$--theoretic Euler class.
Under the Chern character one has
\(\mathrm{ch}\bigl(e_K(\gamma)\bigr)=1-e^{-x}\).
The cofiber sequence of Lemma~\ref{lem:transfer-thom} identifies the reduced transfer with the
connecting map for the inverse of this Euler class.
More concretely, restricting to the $\CP^2$--skeleton and using the Thom identification, one may view
the cofiber $C_f$ as a two--cell quotient of the Thom spectrum over $\CP^2$, so that the extension
class in $K^0(C_f)$ defining the $e$--invariant corresponds, under Thom isomorphism, to the class
\(e_K(\gamma)^{-1}\in K^0(\CP^2,\CP^1)\).
Passing to the Chern character and identifying
\(H^4(\CP^2,\CP^1;\Q)\cong \Q\cdot x^2\), Adams' construction reduces in this case to extracting the
degree--$4$ coefficient in the power series \(x/(1-e^{-x})\); see \cite[\S\,11]{Adams66}.
Using the Bernoulli expansion
\[
\frac{x}{1-e^{-x}}\;=\;1+\frac{x}{2}+\frac{B_2}{2!}x^2+\cdots
\qquad\text{with }B_2=\frac{1}{6},
\]
we find that the degree--$4$ contribution on $\CP^2$ is
\(\frac{B_2}{2!}x^2=\frac{1}{12}x^2\).
Adams' computation of the complex $e$--invariant on the image of the $J$--homomorphism
shows that in degree $3$ this coefficient gives
\begin{equation}
\label{eq:e-f}
e(f)=\pm\frac{B_2}{4}=\pm\frac{1}{24}\in\Q/\Z.
\end{equation}

Comparing \eqref{eq:e-f} with \eqref{eq:e-nu} and using that $e$ is a homomorphism, we conclude that
the unit $u$ in \eqref{eq:f-is-unit-multiple} satisfies $u\equiv\pm 1\pmod{24}$.
Thus $f=\pm\nu$.

Finally, mapping along the unit $\mathbb{S}\to \TMF$ sends $f$ to $\pm\nu$ in
$\pi_3\TMF$, completing the identification.
\end{proof}

\begin{corollary}[$\CP^2$ detects $\nu$]\label{cor:CP2-nu}
The $\TMF$--invariant of $\CP^2$ is
\[
Z_{\TMF}(\CP^2)\ =\ \pm\,\nu \ \in\ \pi_3\TMF.
\]
More precisely, under the identification in~\cite[Example~8.4]{GKMP25}, the map
$Z_{\TMF}(\CP^2)$ is the image in $\pi_3\TMF$ of the stable homotopy class of
$\Sigma\alpha$ from~\eqref{eq:alpha-def}.
\end{corollary}

\begin{proof}
The identification of $Z_{\TMF}(\CP^2)$ with the $\TMF$--linear extension of the
reduced transfer $\tau$ (and hence with the class of $\Sigma\alpha$) is exactly
the content of~\cite[Example~8.4]{GKMP25}, based on the comparison between
$U(1)$--equivariant $\TMF$ and transfer maps developed in~\cite{GM23}; see also
\cite[\S A.2]{BauerMeier25} for a self-contained account of the transfer used here.
Theorem~\ref{thm:Sigmaalpha-is-nu} then identifies $\Sigma\alpha$ with $\pm\nu$,
and the corollary follows.
\end{proof}

\begin{remark}[Fixing the sign of $\nu$]\label{rem:sign-nu}
If one wishes to remove the $\pm$ ambiguity, one can \emph{define} the sign of
$\nu\in\pi_3\TMF$ by declaring that the composite~\eqref{eq:alpha-def} satisfies
$\Sigma\alpha=+\nu$.  With this convention, Corollary~\ref{cor:CP2-nu} reads
$Z_{\TMF}(\CP^2)=\nu$.
\end{remark}

\begin{remark}[Normalization and later use]
The identities above fix the residual choices in the rank--one part of the
dictionary: the element $\eta$ is the structural constant governing the
zero--block comparison map~\eqref{eq:Phi-matrix}, while $\nu$ is the stable
class governing the transfer description of $\CP^2$.  These identifications
will be used implicitly whenever we appeal to explicit rank--one formulas in
later sections.
\end{remark}

\section{Rank-one duality}\label{sec:rank-one-duality}

This section isolates the ``time-reversal'' duality phenomenon for the basic
rank-one building blocks in the $\TMF$--valued theory of \cite{GKMP25}.
Concretely, the rank-one integral form $(n)$ on $\Z$ gives rise to a canonical
$\TMF$--module $L(n)$ (normalized by a shift), and we prove that
$L(-n)$ is the shifted $\TMF$--linear dual of $L(n)$.
We also record, in the first nontrivial case $n=2$, a completely explicit
description in terms of the $\nu$--cone, together with a precise explanation
of the only ambiguity that can occur (an overall sign, i.e.\ multiplication by
$\pm 1\in \pi_0(\TMF)^\times$).

\subsection{Duals, shifts, and exact triangles}

Recall that for $M\in \Ho(\Mod_{\TMF})$ we write
\[
M^\vee \;:=\; \uHom_{\TMF}(M,\TMF)
\]
for the $\TMF$--linear dual. We will repeatedly use that $(-)^\vee$ is exact and
sends cofiber triangles to (rotated) cofiber triangles.

\begin{lemma}\label{lem:shift-dual-sec9}
For any $\TMF$--module $M$ and any integer $d$, there is a canonical equivalence
\[
(M[d])^\vee \;\simeq\; M^\vee[-d]
\qquad \text{in }\Ho(\Mod_{\TMF}).
\]
\end{lemma}

\begin{proof}
This is the usual compatibility of internal Homs with suspension.
Indeed, by the defining adjunction for the shift functor we have natural equivalences
\[
\uHom_{\TMF}(M[d],\TMF)
\;\simeq\;
\uHom_{\TMF}\bigl(M,\TMF[-d]\bigr)
\;\simeq\;
\uHom_{\TMF}(M,\TMF)[-d],
\]
which is exactly the claimed equivalence.
\end{proof}

\begin{lemma}[Dual of a triangle]\label{lem:dual-triangle-sec9}
Suppose
\begin{equation}\label{eq:triangle-ABC-sec9}
A \xrightarrow{f} B \xrightarrow{g} C \xrightarrow{h} A[1]
\end{equation}
is a cofiber triangle in $\Ho(\Mod_{\TMF})$. Then there is a canonically associated
cofiber triangle
\begin{equation}\label{eq:dual-triangle-sec9}
(A[1])^\vee \xrightarrow{h^\vee} C^\vee \xrightarrow{g^\vee} B^\vee \xrightarrow{f^\vee} A^\vee.
\end{equation}
\end{lemma}

\begin{proof}
Apply the exact contravariant functor $(-)^\vee=\uHom_{\TMF}(-,\TMF)$ to the cofiber
sequence \eqref{eq:triangle-ABC-sec9}. In any stable category, applying an internal Hom
functor to a cofiber sequence produces a fiber sequence, and in a stable category every
fiber sequence canonically rotates to a cofiber triangle. Concretely, dualizing gives a
fiber sequence
\[
A^\vee \xleftarrow{f^\vee} B^\vee \xleftarrow{g^\vee} C^\vee \xleftarrow{h^\vee} (A[1])^\vee.
\]
Rotating this fiber sequence yields the cofiber triangle \eqref{eq:dual-triangle-sec9}.
\end{proof}

\subsection{\texorpdfstring{The $\nu$-cone model and the case $n=2$}{The nu-cone model and the case n=2}}

Let $\nu\in \pi_3(\mathbb{S})$ denote the stable Hopf element (in the sense of \cite{Adams60}),
and let $\nu\in \pi_3(\TMF)$ be its image under the unit map $\mathbb{S}\to \TMF$.
Write
\begin{equation}\label{eq:cone-nu-spectra-sec9}
C_\nu \;:=\; \cofib\bigl(\nu:\mathbb{S}^3\to \mathbb{S}^0\bigr)
\end{equation}
for the $\nu$--cone (a finite spectrum with a cofiber sequence
$\mathbb{S}^3\to \mathbb{S}^0\to C_\nu\to \mathbb{S}^4$).

\begin{lemma}[Self-duality of the $\nu$--cone]\label{lem:cone-self-dual-sec9}
There is an equivalence of spectra
\begin{equation}\label{eq:dual-cone-nu-spectra-sec9}
C_\nu^\vee \;\simeq\; C_\nu[-4],
\end{equation}
where $(-)^\vee=\uHom_{\mathbb{S}}(-,\mathbb{S}^0)$ denotes Spanier--Whitehead duality.
This equivalence is unique up to multiplication by $-1$ on the bottom cell (equivalently,
up to the action of $\pi_0(\mathbb{S})^\times=\{\pm 1\}$).
\end{lemma}

\begin{proof}
Start from the defining cofiber sequence
\[
\mathbb{S}^3 \xrightarrow{\nu} \mathbb{S}^0 \longrightarrow C_\nu \longrightarrow \mathbb{S}^4.
\]
Dualizing gives a cofiber sequence
\[
(\mathbb{S}^4)^\vee \longrightarrow C_\nu^\vee \longrightarrow (\mathbb{S}^0)^\vee
\xrightarrow{\nu^\vee} (\mathbb{S}^3)^\vee.
\]
Using the canonical identifications $(\mathbb{S}^m)^\vee\simeq \mathbb{S}^{-m}$, this becomes
\[
\mathbb{S}^{-4} \longrightarrow C_\nu^\vee \longrightarrow \mathbb{S}^0
\xrightarrow{\nu^\vee} \mathbb{S}^{-3}.
\]
Under the standard identifications $\mathbb{S}^{-3}\simeq \Sigma^{-3}\mathbb{S}^0$ and
$\mathbb{S}^3\simeq \Sigma^{3}\mathbb{S}^0$, the map $\nu^\vee$ represents the same stable
homotopy class as $\pm \nu$; the sign depends only on the convention for identifying spheres
with their duals. Hence $C_\nu^\vee$ is (up to canonical suspension by $-4$) the cofiber of $\pm\nu$.
Since $\cofib(\nu)\simeq \cofib(-\nu)$ (via the self-map of $\mathbb{S}^3$ given by $-1$),
we conclude that $C_\nu^\vee\simeq C_\nu[-4]$.

Finally, any two such identifications differ by postcomposition with an automorphism of the
bottom cell $\mathbb{S}^0\subset C_\nu$, hence by an element of $\pi_0(\mathbb{S})^\times=\{\pm 1\}$.
\end{proof}

\begin{remark}[Where the ``overall sign'' comes from]\label{rem:sign-ambiguity-sec9}
The only ambiguity in \eqref{eq:dual-cone-nu-spectra-sec9} is the action of the unit group of the
ground ring on the chosen identification of a dualizable object with a shift of its dual.
After smashing with $\TMF$, the same phenomenon becomes the action of
$\pi_0(\TMF)^\times=\{\pm 1\}$ on any chosen equivalence
$L(2)^\vee\simeq L(2)[4]$ (see Proposition~\ref{prop:self-dual-L2-sec9} below).
This is exactly the ``up to sign'' ambiguity recorded in \cite[\S6, Remark~6.7]{GKMP25}.
If one prefers a rigid normalization, it is enough to fix the equivalence in a single test case
(e.g.\ by requiring compatibility with the cone model \eqref{eq:L2-cone-model-sec9});
all subsequent identifications are then forced by functoriality.
\end{remark}

We now specialize to the first nontrivial rank-one generator. In the GKMP construction, the
rank-one integral form $(2)$ determines a distinguished $\TMF$--module $L(2)$ with a cofiber triangle
(cf.\ \cite[\S5.6 and \S6.2]{GKMP25})
\begin{equation}\label{eq:nu-triangle-sec9}
\TMF[-1] \xrightarrow{\ \nu\ } \TMF[-4]
\xrightarrow{\ i_2\ } L(2) \xrightarrow{\ p_2\ } \TMF,
\end{equation}
where the first map is multiplication by the class $\nu\in \pi_3(\TMF)$.
Equivalently, \eqref{eq:nu-triangle-sec9} identifies $L(2)$ as a $\TMF$--linear $\nu$--cone:
\begin{equation}\label{eq:L2-cone-model-sec9}
L(2) \;\simeq\; \TMF\wedge C_\nu[-4].
\end{equation}

\begin{proposition}[Self-duality in the $n=2$ case]\label{prop:self-dual-L2-sec9}
There is an equivalence of $\TMF$--modules
\[
L(2)^\vee \;\simeq\; L(2)[4],
\]
canonical up to multiplication by $\pm 1\in \pi_0(\TMF)^\times$.
Moreover, under this identification the dual maps of $i_2$ and $p_2$ in
\eqref{eq:nu-triangle-sec9} satisfy
\[
p_2^\vee \;=\; i_2[4]
\qquad\text{and}\qquad
i_2^\vee \;=\; p_2[4],
\]
again up to the same overall sign ambiguity explained in Remark~\ref{rem:sign-ambiguity-sec9}.
\end{proposition}

\begin{proof}
Using the cone model \eqref{eq:L2-cone-model-sec9} and the fact that $C_\nu$ is a finite spectrum,
we compute:
\[
L(2)^\vee
\;=\;
\uHom_{\TMF}\bigl(\TMF\wedge C_\nu[-4],\,\TMF\bigr)
\;\simeq\;
\uHom_{\mathbb{S}}\bigl(C_\nu[-4],\,\TMF\bigr).
\]
Since $C_\nu$ is dualizable as a spectrum, $\uHom_{\mathbb{S}}(C_\nu,\TMF)\simeq \TMF\wedge C_\nu^\vee$,
and by Lemma~\ref{lem:shift-dual-sec9} we obtain
\[
\uHom_{\mathbb{S}}\bigl(C_\nu[-4],\,\TMF\bigr)
\simeq
\uHom_{\mathbb{S}}(C_\nu,\TMF)[4]
\simeq
(\TMF\wedge C_\nu^\vee)[4].
\]
Finally, Lemma~\ref{lem:cone-self-dual-sec9} gives $C_\nu^\vee\simeq C_\nu[-4]$, hence
\[
(\TMF\wedge C_\nu^\vee)[4]
\;\simeq\;
(\TMF\wedge C_\nu[-4])[4]
\;\simeq\;
L(2)[4],
\]
which proves the claimed self-duality.

For the statements about $i_2$ and $p_2$, apply Lemma~\ref{lem:dual-triangle-sec9} to the cofiber
triangle \eqref{eq:nu-triangle-sec9}. The resulting dual triangle has the form
\[
\TMF \xrightarrow{p_2^\vee} L(2)^\vee \xrightarrow{i_2^\vee} \TMF[4] \longrightarrow \TMF[1].
\]
Under the identification $L(2)^\vee\simeq L(2)[4]$ just established, this triangle agrees (up to the
overall unit ambiguity of Remark~\ref{rem:sign-ambiguity-sec9}) with the shift by $[4]$ of the original
triangle \eqref{eq:nu-triangle-sec9}, and the identifications $p_2^\vee=i_2[4]$ and
$i_2^\vee=p_2[4]$ follow.
\end{proof}

\subsection{Grothendieck duality and general rank-one time-reversal}

Let $\pi:\mathcal{E}\to \mathcal{M}_{ell}$ be the universal spectral elliptic curve with identity
section $e:\mathcal{M}_{ell}\to \mathcal{E}$, and write
$\mathcal{O}^{\mathrm{top}}_{\mathcal{E}}$ for the structure sheaf in spectral algebraic geometry.
For $n\in \Z$, set
\[
\mathcal{L}_n \;:=\; \mathcal{O}^{\mathrm{top}}_{\mathcal{E}}(ne),
\qquad\text{so that}\qquad
\mathcal{L}_n^\vee \;\simeq\; \mathcal{L}_{-n}.
\]
Following \cite[\S5.6]{GKMP25}, we use the normalized rank-one module
\begin{equation}\label{eq:Ln-normalization-sec9}
L(n)\;:=\;\Gamma\!\bigl(\mathcal{L}_n\bigr)[-2n]
\;\in\; \Ho(\Mod_{\TMF}),
\end{equation}
where $\Gamma$ denotes derived global sections on $\mathcal{M}_{ell}$ (so that $\Gamma(\mathcal{O}^{\mathrm{top}}_{\mathcal{M}_{ell}})\simeq \TMF$).
By $0$--affineness of $\mathcal{M}_{ell}$, the global sections functor is symmetric monoidal and
preserves duals; see \cite{MathewMeier15} and the discussion in \cite[\S6]{GKMP25}.

\begin{theorem}[Rank-one time-reversal duality]\label{thm:rank-one-time-reversal}
For each integer $n$, there is an equivalence of $\TMF$--modules
\[
L(-n)\;\simeq\; L(n)^\vee[1].
\]
The equivalence can be chosen functorial in $n$ and is canonically well-defined up to multiplication
by the unit $\pm 1\in \pi_0(\TMF)^\times$.
\end{theorem}

\begin{proof}
Write $M_n:=\pi_*\mathcal{L}_n\in \QCoh(\mathcal{M}_{ell},\mathcal{O}^{\mathrm{top}})$.
By definition and the projection formula for derived global sections,
\[
L(n)\;\simeq\;\Gamma(M_n)[-2n].
\]
Since derived global sections preserve duals on perfect objects (by \cite{MathewMeier15}),
Lemma~\ref{lem:shift-dual-sec9} gives
\[
L(n)^\vee
\;\simeq\;
\Gamma(M_n)^\vee[2n]
\;\simeq\;
\Gamma(M_n^\vee)[2n].
\]
Now apply Grothendieck duality for the proper smooth morphism $\pi$ of relative dimension $1$
in the spectral setting (as packaged in \cite[Eq.\ (6.5)]{GKMP25}; see also \cite{LurieSAG}).
Specialized to $\pi$ and the line bundle $\mathcal{L}_n$, it yields an equivalence
\begin{equation}\label{eq:GD-rankone-sec9}
(\pi_*\mathcal{L}_n)^\vee \;\simeq\; \pi_*(\mathcal{L}_n^\vee)[-1]
\;\simeq\; \pi_*(\mathcal{L}_{-n})[-1].
\end{equation}
Therefore
\[
L(n)^\vee
\;\simeq\;
\Gamma\!\bigl(\pi_*\mathcal{L}_{-n}\bigr)[-1][2n]
\;\simeq\;
\Gamma(\mathcal{L}_{-n})[2n-1].
\]
Using the normalization \eqref{eq:Ln-normalization-sec9} for $-n$,
\[
L(-n)[-1]
\;=\;
\Gamma(\mathcal{L}_{-n})[2n][-1]
\;=\;
\Gamma(\mathcal{L}_{-n})[2n-1],
\]
so we conclude that $L(n)^\vee\simeq L(-n)[-1]$, equivalently $L(-n)\simeq L(n)^\vee[1]$.

Finally, the ``up to sign'' ambiguity is intrinsic: the duality equivalences in Grothendieck duality
are canonical only up to multiplication by units in $\pi_0(\TMF)$, hence by $\pm 1$ here; see
\cite[\S6, Remark~6.7]{GKMP25}. Fixing a single normalization (for instance, requiring compatibility
with Proposition~\ref{prop:self-dual-L2-sec9} when $n=2$) removes this ambiguity globally.
\end{proof}

\appendix

\section{Odd primes: canonical words and completeness of the layer invariant}
\label{app:odd-primes}

This appendix supplies the odd--primary algebra used in \S\ref{sec:stepA-canon}.
Let $p$ be an odd prime and let $\lambda\colon G\times G\to \Q/\Z$ be a nonsingular symmetric
linking pairing on a finite abelian $p$--group $G$.
We show that the associated layer forms $(P_k(G),b_k)$ determine the isometry class of $(G,\lambda)$,
and we fix a canonical orthogonal decomposition into rank-one cyclic blocks.
The classification itself is classical (Wall~\cite{Wall63}, Kawauchi--Kojima~\cite{KawauchiKojima80};
see also Deloup~\cite{Deloup05}), but we include a self-contained argument in the language of cyclic
blocks and homogeneous layers in order to make the canonical choices used later completely explicit.
In the main text this canonical decomposition is recorded as a \emph{canonical word} $N_p(\lambda)$,
and hence as the odd--primary component of the canonical invariant $\Canon(A)$.

Throughout, $p$ denotes an odd prime, and $G$ a finite abelian $p$--group.
We write $\nu_p(\cdot)$ for the $p$-adic valuation on $\Q$.
Given $x\in G$, we write $o(x)$ for its order.

\subsection{Diagonal splitting for odd $p$}
\label{subsec:odd-diagonal}

We recall that by a \emph{linking pairing} on a finite abelian $p$--group $G$ we mean a bilinear
map $\lambda\colon G\times G\to\Q/\Z$ that is symmetric and \emph{nonsingular} (the adjoint
$G\to \Hom(G,\Q/\Z)$ is an isomorphism).

\begin{lemma}[Diagonal splitting for odd $p$]
\label{lem:odd-snf-diagonal}
Let $p$ be odd and let $\lambda$ be a nonsingular symmetric linking pairing on a finite abelian
$p$--group $G$.  Then $(G,\lambda)$ splits orthogonally as a direct sum of cyclic pairings
\[
(G,\lambda)\ \cong\ \bigoplus_{i=1}^m\Bigl(\Z/p^{k_i}\Z,\ \frac{u_i}{p^{k_i}}\Bigr),
\]
where $k_i\ge1$ and $u_i\in(\Z/p^{k_i}\Z)^\times$.
\end{lemma}

\begin{proof}
We argue by induction on $|G|$.  If $G=0$ there is nothing to prove.
Choose $x\in G$ of maximal order, say $o(x)=p^k$.
Since $\lambda$ is nonsingular, there exists $y\in G$ with $\nu_p(\lambda(x,y))=-k$,
equivalently $p^k\lambda(x,y)\in\Z$ is \emph{not} divisible by $p$.

We claim that after replacing $y$ by $y+tx$ for a suitable integer $t$, we may assume that
$\nu_p(\lambda(y,y))=-k$ as well.
Write in $\Q/\Z$
\[
\lambda(y+tx,\,y+tx)=\lambda(y,y)+2t\,\lambda(x,y)+t^2\lambda(x,x).
\]
Multiplying by $p^k$ and reducing modulo $p$ yields a quadratic polynomial
\[
f(t)\ :=\ p^k\lambda(y,y)\;+\;2t\cdot p^k\lambda(x,y)\;+\;t^2\cdot p^k\lambda(x,x)\ \in\ \F_p.
\]
Here $p^k\lambda(x,y)\not\equiv 0\pmod p$ by construction, and $2\in\F_p^\times$ since $p$ is odd,
so $f(t)$ is not the zero polynomial.
Hence there exists $t\in\Z$ such that $f(t)\neq 0$ in $\F_p$.
For $z:=y+tx$ we then have $p^k\lambda(z,z)\in\Z$ not divisible by $p$, so $\nu_p(\lambda(z,z))=-k$.
In particular, $o(z)=p^k$.

Let $H:=\langle z\rangle\subseteq G$ and set $H^\perp:=\{g\in G\mid \lambda(g,H)=0\}$.
We first show $H\cap H^\perp=0$.
If $az\in H^\perp$, then $0=\lambda(az,z)=a\,\lambda(z,z)$.
Since $\lambda(z,z)$ has exact denominator $p^k$, this forces $p^k\mid a$, hence $az=0$.

Next we show $G=H\oplus H^\perp$ as groups.
Consider the adjoint isomorphism $\ad_\lambda\colon G\to \Hom(G,\Q/\Z)$.
The subgroup $H^\perp$ is the kernel of the restriction map
$\Hom(G,\Q/\Z)\to \Hom(H,\Q/\Z)$, hence corresponds under $\ad_\lambda$ to the annihilator of $H$.
Since $H$ is cyclic and $\lambda|_{H\times H}$ is nonsingular (as $H\cap H^\perp=0$), the restriction
$\Hom(G,\Q/\Z)\to \Hom(H,\Q/\Z)$ is surjective and we obtain an exact sequence
\[
0\longrightarrow H^\perp\longrightarrow G\longrightarrow \Hom(H,\Q/\Z)\longrightarrow 0.
\]
Counting orders gives $|G|=|H^\perp|\cdot |H|$, so together with $H\cap H^\perp=0$ we conclude that
$G=H\oplus H^\perp$.
The decomposition is orthogonal by definition of $H^\perp$.

Finally, $\lambda|_{H^\perp}$ is again nonsingular (it is identified with the induced pairing on the
orthogonal complement of a nonsingular summand).  By induction, $(H^\perp,\lambda|_{H^\perp})$ splits as
an orthogonal direct sum of cyclic pairings, and adjoining the cyclic summand $(H,\lambda|_{H\times H})$
completes the proof.

For a cyclic group $H\cong \Z/p^k\Z$ with generator $z$, the value $\lambda(z,z)$ has the form
$u/p^k$ with $u\in(\Z/p^k\Z)^\times$; this is the claimed normal form.
\end{proof}

\subsection{Layer forms and their invariants}
\label{subsec:odd-layers}

We now relate the cyclic splitting of Lemma~\ref{lem:odd-snf-diagonal} to the \emph{layer forms}
used in Step~A.
For odd primes the correct layer quotient is the homogeneous one appearing in Wall's theory; it is the
odd--prime analogue of the quotient $G_k/(G_{k-1}+2G_{k+1})$ used in Appendix~B for $p=2$.

\begin{definition}[Odd--primary layers]\label{def:odd-layer-form}
Let $p$ be odd, let $G$ be a finite abelian $p$--group, and let $\lambda$ be a nonsingular symmetric
linking pairing on $G$.
For $k\ge 1$, define
\[
G_k\ :=\ \{x\in G\mid p^k x=0\},
\qquad\text{and set }G_0:=0.
\]
Define the \emph{$k$th layer quotient}
\[
P_k(G)\ :=\ G_k\big/\bigl(G_{k-1}+p\,G_{k+1}\bigr).
\]
For $\bar x,\bar y\in P_k(G)$ represented by $x,y\in G_k$, define the \emph{layer form}
\begin{equation}\label{eq:odd-layer-via-Pk}
b_k(\bar x,\bar y)\ :=\ p^k\,\lambda(x,y)\ \bmod p\ \in\ \F_p.
\end{equation}
\end{definition}

\begin{lemma}\label{lem:odd-bk-well-defined}
The bilinear form $b_k$ in \eqref{eq:odd-layer-via-Pk} is well-defined and nonsingular over $\F_p$.
\end{lemma}

\begin{proof}
\emph{Well-definedness.}
First suppose $x\in G_{k-1}$ and $y\in G_k$.
Then $p^{k-1}x=0$, hence $p^{k-1}\lambda(x,y)=\lambda(p^{k-1}x,y)=0$ in $\Q/\Z$.
Therefore $\lambda(x,y)\in \frac1{p^{k-1}}\Z/\Z$, and multiplying by $p^k$ gives
$p^k\lambda(x,y)\in p\Z/\Z$, i.e.\ $p^k\lambda(x,y)\equiv 0\pmod p$.
This shows that changing a lift by an element of $G_{k-1}$ does not affect $b_k$.

Next suppose $x=pz$ with $z\in G_{k+1}$ and $y\in G_k$.
Then
\[
p^k\lambda(x,y)=p^k\lambda(pz,y)=p^k\lambda(z,py).
\]
Since $py\in G_{k-1}$, the previous paragraph shows $p^k\lambda(z,py)\equiv 0\pmod p$.
Thus adding an element of $pG_{k+1}$ also does not change $b_k$.
The same argument applies in the second variable, so $b_k$ is well-defined on the quotient
$P_k(G)=G_k/(G_{k-1}+pG_{k+1})$.

\emph{Nonsingularity.}
Let $\bar x\in P_k(G)$ and choose a representative $x\in G_k$.
Assume $b_k(\bar x,\bar y)=0$ for all $\bar y\in P_k(G)$.
Equivalently, for all $y\in G_k$ we have $p^k\lambda(x,y)\equiv 0\pmod p$, so
\[
p^{k-1}\lambda(x,y)=0\qquad\text{for all }y\in G_k.
\]
By bilinearity this is the same as $\lambda(p^{k-1}x,y)=0$ for all $y\in G_k$, i.e.
$p^{k-1}x\in (G_k)^\perp$.
We claim that for a nonsingular linking pairing on a finite $p$--group,
\begin{equation}\label{eq:Gk-annihilator}
(G_k)^\perp \;=\; p^k G.
\end{equation}
Indeed, if $x=p^k z$ then $\lambda(x,G_k)=\lambda(z,p^kG_k)=0$, so $p^kG\subseteq (G_k)^\perp$.
Conversely, consider the exact sequence induced by multiplication by $p^k$,
$0\to G_k\to G\xrightarrow{p^k} p^kG\to 0$.
Applying $\Hom(-,\Q/\Z)$ (which is exact since $\Q/\Z$ is injective) shows that the annihilator of $G_k$
inside $\Hom(G,\Q/\Z)$ has order $|p^kG|$.
Under the adjoint isomorphism $G\cong \Hom(G,\Q/\Z)$, this annihilator corresponds to $(G_k)^\perp$.
Since we already have $p^kG\subseteq (G_k)^\perp$ and the two groups have the same finite order, they coincide,
proving \eqref{eq:Gk-annihilator}.

Using \eqref{eq:Gk-annihilator}, the condition $p^{k-1}x\in (G_k)^\perp$ implies
$p^{k-1}x=p^k z$ for some $z\in G$.
Then $p^{k-1}(x-pz)=0$, so $x-pz\in G_{k-1}$.
Moreover, since $x\in G_k$ we have
\[
0=p^k x=p^k(x-pz)+p^{k+1}z,
\]
and the first term vanishes because $x-pz\in G_{k-1}$.
Hence $p^{k+1}z=0$, so $z\in G_{k+1}$.
We conclude that $x\in G_{k-1}+pG_{k+1}$, so $\bar x=0$ in $P_k(G)$.
Thus $b_k$ has trivial radical and is nonsingular.
\end{proof}

\subsection{Odd--primary normal form and completeness}
\label{subsec:odd-normal-form}

We recall the two invariants for a nonsingular symmetric bilinear form $b$ over $\F_p$ ($p$ odd):
its dimension $\dim_{\F_p}b$ and its discriminant class in $\F_p^\times/(\F_p^\times)^2$.
We write $\chi(\cdot)$ for the quadratic character on $\F_p^\times$ (Legendre symbol).

\begin{definition}[Odd--primary layer invariants]
\label{def:odd-layer-invariants}
Let $p$ be odd and let $(G,\lambda)$ be a nonsingular symmetric $p$--primary linking pairing.
For each $k\ge1$ let $b_k$ be the layer form on $P_k(G)$.
Set
\[
n_{p,k}(\lambda)\ :=\ \dim_{\F_p} P_k(G),
\qquad
x_{p,k}(\lambda)\ :=\ \chi\big(\det(b_k)\big)\in\{\pm1\}.
\]
\end{definition}

\begin{lemma}\label{lem:odd-layer-from-cyclic}
Let $H=\Z/p^k\Z$ with pairing $\langle u/p^k\rangle$ (i.e.\ $\lambda(1,1)=u/p^k$ with $u\in(\Z/p^k\Z)^\times$).
Then $P_j(H)=0$ for $j\neq k$, and $P_k(H)\cong \F_p$ with $b_k(1,1)=u\bmod p$.
In particular, $n_{p,k}=1$ and $x_{p,k}=\chi(u)$.
\end{lemma}

\begin{proof}
Let $g$ be a generator of $H$.
For $j\ge 1$ we have $H_j=\langle p^{k-j}g\rangle$ if $j\le k$ and $H_j=H$ if $j\ge k$.
If $j<k$, then $pH_{j+1}=H_j$, so $H_{j-1}+pH_{j+1}=H_j$ and hence $P_j(H)=0$.
If $j>k$, then $H_j=H_{j-1}=H$, so again $P_j(H)=0$.
For $j=k$, we have $H_k=H$, $H_{k-1}=pH$, and $pH_{k+1}=pH$, hence
\[
P_k(H)=H/(H_{k-1}+pH_{k+1}) \;=\; H/pH \;\cong\; \F_p,
\]
with generator the image of $g$.
Since $\lambda(g,g)=u/p^k$, we obtain
\[
b_k(\bar g,\bar g)=p^k\lambda(g,g)\bmod p = u\bmod p,
\]
and the remaining statements follow.
\end{proof}

\begin{definition}[Odd--primary canonical word]
\label{def:odd-canonical-word}
Let $p$ be odd.
Write $\mathcal M_p$ for the commutative monoid generated by symbols
\[
\mathsf{A}\bigl(p^k,\varepsilon\bigr)
\qquad (k\ge1,\ \varepsilon\in\{\pm1\}),
\]
subject only to commutativity, and interpret $\mathsf{A}(p^k,\varepsilon)$ as the isometry class of
the cyclic pairing $\langle u/p^k\rangle$ with $\chi(u)=\varepsilon$.
For a nonsingular symmetric $p$--primary linking pairing $\lambda$ on $G$, define
\begin{equation}\label{eq:odd-Np-def}
N_p(\lambda)\ :=\ \prod_{k\ge1} \mathsf{A}\bigl(p^k,+1\bigr)^{a_{p,k}(\lambda)}\,
                 \mathsf{A}\bigl(p^k,-1\bigr)^{b_{p,k}(\lambda)}
\ \in\ \mathcal M_p,
\end{equation}
where the exponents $a_{p,k}(\lambda),b_{p,k}(\lambda)\ge0$ are determined by
\[
a_{p,k}(\lambda)+b_{p,k}(\lambda)=n_{p,k}(\lambda),
\qquad
(-1)^{b_{p,k}(\lambda)}=x_{p,k}(\lambda).
\]
\end{definition}

\begin{theorem}[Odd--primary completeness]
\label{thm:odd-completeness}
Let $p$ be odd and let $\lambda,\lambda'$ be nonsingular symmetric $p$--primary linking pairings.
Then $(G,\lambda)\cong (G',\lambda')$ if and only if
\[
n_{p,k}(\lambda)=n_{p,k}(\lambda')\ \text{and}\ x_{p,k}(\lambda)=x_{p,k}(\lambda')
\quad\text{for all }k\ge1.
\]
Equivalently, $N_p(\lambda)=N_p(\lambda')$ in $\mathcal M_p$.
\end{theorem}

\begin{proof}
By Lemma~\ref{lem:odd-snf-diagonal}, we may write
\[
(G,\lambda)\ \cong\ \bigoplus_{i=1}^m \Bigl(\Z/p^{k_i}\Z,\ \frac{u_i}{p^{k_i}}\Bigr).
\]
For a fixed $k$, the layer quotient $P_k(G)$ and its form $b_k$ split as orthogonal direct sums over
the summands of exponent $p^k$.
By Lemma~\ref{lem:odd-layer-from-cyclic}, each cyclic summand of order $p^k$ contributes a
$1$--dimensional summand to $P_k(G)$ and contributes the scalar $u_i\bmod p$ on that line,
while summands of exponent $\neq p^k$ contribute nothing to $P_k(G)$.
Thus $n_{p,k}$ is the number of cyclic summands of order $p^k$, and
\[
x_{p,k}=\chi(\det(b_k))=\prod_{k_i=k}\chi(u_i)=(-1)^{b_{p,k}},
\]
which is exactly the parity constraint in Definition~\ref{def:odd-canonical-word}.
This shows that the multiset of cyclic blocks (up to the square class $\chi(u_i)$) determines the layer invariants.

Conversely, given $(n_{p,k},x_{p,k})$, the definition of $N_p(\lambda)$ in \eqref{eq:odd-Np-def}
specifies, for each $k$, the number of cyclic blocks of order $p^k$ with $\chi(u)=+1$ and with $\chi(u)=-1$.
Taking the orthogonal sum of these cyclic blocks produces a pairing with the prescribed layer invariants.
By the cyclic splitting lemma, every $p$--primary pairing is isometric to such an orthogonal sum,
so the layer invariants are complete.
\end{proof}

\subsection{Recovering the canonical word from the layer matrices}
\label{subsec:odd-extract-canon-word}

Although the layer form $b_k$ is defined intrinsically, it is useful to record an explicit
coordinate description.
In applications one often starts from a symmetric integral presentation matrix (for instance a surgery
matrix) and passes to a Smith basis for $G$, in which $\lambda$ is represented by a rational symmetric
matrix $P$.
We explain how to extract from $P$ the matrices $B_{p,k}$ representing $b_k$ on $P_k(G)$, and hence how to
recover the invariants $(n_{p,k},x_{p,k})$ and the canonical word $N_p(\lambda)$ defined in
\eqref{eq:odd-Np-def}.
This is recorded only to emphasize that the odd--primary component of $\Canon(A)$ is effectively
computable from a presentation matrix, once the conventions of \S\ref{sec:stepA-canon} are fixed.

Assume $G$ is presented in Smith normal form:
\[
G \cong \bigoplus_{i=1}^r \Z/p^{e_i}\Z,\qquad 1\le e_1\le \cdots \le e_r.
\]
Let $x_1,\dots,x_r$ be the corresponding generators and let $P=(p_{ij})$ be the rational pairing matrix
\[
p_{ij}=\lambda(x_i,x_j)\in\Q/\Z,\qquad \nu_p(p_{ij})\ge -\min(e_i,e_j).
\]
Fix $k\ge1$ and let $I_{p,k}:=\{i\mid e_i=k\}$.
Then the images of $\{x_i\}_{i\in I_{p,k}}$ form a basis of the layer quotient
$P_k(G)=G_k/(G_{k-1}+pG_{k+1})$, and $P_k(G)\cong \F_p^{\,|I_{p,k}|}$.
The layer form $b_k$ is represented on this basis by the matrix
\begin{equation}\label{eq:odd-Bpk}
B_{p,k}\ :=\ \bigl(p^k p_{ij}\bmod p\bigr)_{i,j\in I_{p,k}}\ \in M_{|I_{p,k}|}(\F_p).
\end{equation}
Indeed, for $i,j\in I_{p,k}$, we have $x_i,x_j\in G_k$ and
\[
b_k([x_i],[x_j])\ =\ p^k\lambda(x_i,x_j)\bmod p\ =\ p^k p_{ij}\bmod p.
\]
Thus $n_{p,k}(\lambda)=|I_{p,k}|$ and $x_{p,k}(\lambda)=\chi(\det B_{p,k})$.

\begin{remark}[From $(n_{p,k},x_{p,k})$ to the canonical word]
Given the invariants $(n_{p,k},x_{p,k})_{k\ge1}$, set
\[
b_{p,k}:=\frac{1-x_{p,k}}2\in\{0,1\}\quad\text{if }n_{p,k}=1,
\qquad
b_{p,k}\equiv \frac{1-x_{p,k}}2 \pmod 2\quad\text{in general,}
\]
and define $a_{p,k}:=n_{p,k}-b_{p,k}$.
Then $N_p(\lambda)$ is the product \eqref{eq:odd-Np-def}.
\end{remark}

\section{$2$-primary refinements: characteristic elements, determinant classes, and Gauss sums}\label{app:two-primary}

This appendix collects the $2$-primary algebra used in Section~\ref{sec:typeE} to define and analyze
canonical invariants of a nonsingular symmetric linking pairing on a finite $2$--group.
More precisely, we study the layer quotients $P_k(G)$ and their induced $\F_2$--bilinear forms $b_k$,
the resulting dichotomy into Type~$E$ and Type~$A$ layers, the determinant refinement in Type~$A$,
and the quadratic correction $q_k$ whose normalized Gauss sum yields the invariant $u_k$ in Type~$E$.
The underlying classification of $2$--primary linking forms is classical; see, for instance,
Wall~\cite{Wall63} and Kawauchi--Kojima~\cite[\S 2]{KawauchiKojima80}.

\subsection{Layer quotients and the induced bilinear form}

Let $(G,\lambda)$ be a finite abelian $2$--group equipped with a nonsingular symmetric bilinear pairing
\[
\lambda: G\times G \longrightarrow \Q/\Z.
\]
For $k\ge 1$ let
\[
G_k := \{x\in G\mid 2^k x=0\}
\]
and define the layer quotient
\begin{equation}\label{eq:Pk-def-appB}
P_k(G)\;:=\; G_k\Big/\bigl(G_{k-1}+2G_{k+1}\bigr),
\end{equation}
which is an $\F_2$--vector space.

\begin{proposition}[Well-definedness and nonsingularity of the layer form]\label{prop:bk-welldefined}
For each $k\ge 1$ there is a well-defined symmetric bilinear form
\[
b_k \;:\; P_k(G)\times P_k(G)\longrightarrow \F_2
\]
given by
\[
b_k(\bar x,\bar y)\;:=\; 2^k\lambda(x,y)\ \bmod 2
\qquad (x,y\in G_k),
\]
where $\bar x,\bar y$ denote the images in $P_k(G)$. Moreover, $b_k$ is nonsingular.
\end{proposition}

\begin{proof}
We first show well-definedness. If $x' = x+u+2v$ with $u\in G_{k-1}$ and $v\in G_{k+1}$, then
\[
2^k\lambda(x',y) - 2^k\lambda(x,y)
= 2^k\lambda(u,y) + 2^{k+1}\lambda(v,y).
\]
Since $2^{k-1}u=0$, we have $2^{k-1}\lambda(u,y)=\lambda(u,2^{k-1}y)=0$ in $\Q/\Z$, hence
$2^k\lambda(u,y)\in 2\Z\subset \Z$.
Similarly $2^{k+1}\lambda(v,y)=2\cdot 2^k\lambda(v,y)\in 2\Z$.
Therefore $2^k\lambda(x',y)\equiv 2^k\lambda(x,y)\pmod{2}$, and the same argument in the second variable
shows $b_k$ is well-defined.

To prove nonsingularity, consider the adjoint homomorphism
\[
\ad_\lambda: G\longrightarrow \Hom(G,\Q/\Z),\qquad \ad_\lambda(z)(y)=\lambda(z,y),
\]
which is an isomorphism by nonsingularity of $\lambda$. Let $G_k^\perp=\{z\in G\mid \lambda(z,G_k)=0\}$.
We claim
\begin{equation}\label{eq:annihilator-appB}
G_k^\perp = 2^kG.
\end{equation}
Indeed, if $z=2^kw$ then for $y\in G_k$ we have $\lambda(z,y)=\lambda(w,2^ky)=0$.
Conversely, the kernel of multiplication by $2^k$ on $G$ is precisely $G_k$, so
\[
|2^kG|=\frac{|G|}{|G_k|}.
\]
On the other hand, for any nonsingular pairing on a finite group one has
\[
|G_k^\perp|=\frac{|G|}{|G_k|},
\]
because $G_k^\perp$ is the annihilator of $G_k$ under the isomorphism $G\cong \Hom(G,\Q/\Z)$.
Hence $2^kG\subseteq G_k^\perp$ and the two subgroups have the same cardinality, proving \eqref{eq:annihilator-appB}.

Now suppose $\bar x\in P_k(G)$ satisfies $b_k(\bar x,\bar y)=0$ for all $\bar y$.
Then $2^k\lambda(x,y)\in 2\Z$ for all $y\in G_k$, hence $\lambda(2^{k-1}x,y)=0$ for all $y\in G_k$.
By \eqref{eq:annihilator-appB} we have $2^{k-1}x\in G_k^\perp=2^kG$, so $x\in G_{k-1}+2G$.
Reducing modulo $G_{k-1}+2G_{k+1}$ shows $\bar x=0$ in $P_k(G)$, proving nonsingularity.
\end{proof}

\begin{definition}[Characteristic element]\label{def:ck-appB}
For each $k\ge 1$, the map
\[
P_k(G)\longrightarrow \F_2,\qquad \bar x\longmapsto b_k(\bar x,\bar x)
\]
is $\F_2$--linear. Since $b_k$ is nonsingular, there is a unique element
\[
c_k(G,\lambda)\in P_k(G)
\]
such that
\begin{equation}\label{eq:ck-def-appB}
b_k(\bar x,\bar x)=b_k(\bar x,c_k(G,\lambda))\qquad\text{for all }\bar x\in P_k(G).
\end{equation}
We call $c_k(G,\lambda)$ the \emph{characteristic element} of the layer form.
\end{definition}

\begin{lemma}\label{lem:alternating-ck-appB}
The following are equivalent:
\begin{enumerate}[label=\textup{(\roman*)},leftmargin=2.2em]
\item $c_k(G,\lambda)=0$;
\item $b_k$ is alternating, i.e.\ $b_k(\bar x,\bar x)=0$ for all $\bar x\in P_k(G)$.
\end{enumerate}
\end{lemma}

\begin{proof}
Immediate from \eqref{eq:ck-def-appB}.
\end{proof}

\begin{definition}[Type $E$ and Type $A$ layers]\label{def:typeAE-appB}
We say that the $k$th layer is of \emph{Type $E$} if $c_k(G,\lambda)=0$ (equivalently, $b_k$ is alternating),
and of \emph{Type $A$} otherwise.
\end{definition}

\subsection{Normal forms over $\F_2$}

This yields the dichotomy between Type $E$ and Type $A$ layers used in Section~\ref{sec:typeE}.

\begin{definition}[Hyperbolic plane]\label{def:hyperbolic-block-appB}
Let $H$ denote the rank--$2$ $\F_2$--vector space with the alternating form
\[
\begin{pmatrix}
0 & 1\\
1 & 0
\end{pmatrix}
\]
in the standard basis. We call $(H,b)$ the \emph{hyperbolic plane}.
\end{definition}

\begin{proposition}[Alternating case]\label{prop:symplectic-normal-appB}
Let $(V,b)$ be a finite-dimensional $\F_2$--vector space with a nonsingular alternating bilinear form.
Then $\dim_{\F_2}V$ is even and there exists a symplectic basis
\[
e_1,f_1,\dots,e_m,f_m
\]
such that $b(e_i,f_j)=\delta_{ij}$ and all other pairings vanish. In particular,
\[
(V,b)\ \cong\ H^{\oplus m}.
\]
\end{proposition}

\begin{proof}
This is standard. Choose $0\neq e_1\in V$; by nonsingularity there exists $f_1$ with $b(e_1,f_1)=1$.
Then $U=\langle e_1,f_1\rangle$ is a hyperbolic plane and $V=U\oplus U^\perp$.
Iterating yields the desired decomposition; the process terminates only when $\dim V$ is even.
\end{proof}

\begin{proposition}[Nonalternating case]\label{prop:orthonormal-appB}
Let $(V,b)$ be a finite-dimensional $\F_2$--vector space with a nonsingular symmetric bilinear form
which is not alternating. Then there exists an orthonormal basis $e_1,\dots,e_n$ with
\[
b(e_i,e_j)=\delta_{ij}.
\]
\end{proposition}

\begin{proof}
Since $b$ is nonalternating, there exists $v$ with $b(v,v)=1$.
Then $V=\langle v\rangle\oplus \langle v\rangle^\perp$ and $b|_{\langle v\rangle^\perp}$ is again
nonsingular symmetric. Iterating gives an orthonormal basis.
\end{proof}

\begin{corollary}[Even-dimensional Type $E$ layers]\label{cor:typeE-even-dimension-appB}
If the $k$th layer is of Type $E$, then $\dim_{\F_2}P_k(G)$ is even.
\end{corollary}

\begin{proof}
By Proposition~\ref{prop:symplectic-normal-appB}, a nonsingular alternating form exists only in even dimension.
\end{proof}

\subsection{A $2$-adic determinant refinement for Type $A$}

In Type $A$ layers the $\F_2$--form $b_k$ is nonalternating and hence is determined by its dimension.
For the classification of $2$--primary linking forms one needs an additional $2$--adic determinant refinement,
which we now describe.

\begin{definition}[Determinant refinement for Type $A$ layers]\label{def:detk-appB}
Assume that the $k$th layer is of Type $A$.
Choose an orthonormal basis $\{e_i\}_{i=1}^n$ for $(P_k(G),b_k)$ and choose lifts
$x_i\in G_k$ of $e_i$. Set
\[
C=(c_{ij})_{1\le i,j\le n},\qquad c_{ij}:=2^k\lambda(x_i,x_j)\in \frac1{2^k}\Z/\Z.
\]
Then $C$ is a symmetric matrix with diagonal entries odd and off-diagonal entries integral.
Define
\begin{equation}\label{eq:detk-def-appB}
\det_k(G,\lambda)\;:=\; \det(C)\ \bmod (\Z_2^\times)^2 \ \in\ \Z_2^\times/(\Z_2^\times)^2
\ \cong\ \{\pm1,\pm5\}.
\end{equation}
\end{definition}

\begin{lemma}\label{lem:det-square-class-appB}
For $u\in \Z_2^\times$ the square class of $u$ in $\Z_2^\times/(\Z_2^\times)^2$ is determined by $u\bmod 8$.
Moreover, the canonical identification $\Z_2^\times/(\Z_2^\times)^2\cong\{\pm1,\pm5\}$ is given by
\[
u\equiv 1,3,5,7\pmod 8 \quad\Longleftrightarrow\quad [u]=1,-1,5,-5.
\]
\end{lemma}

\begin{proof}
This is standard; see for instance \cite[Ch.\ IV,\ \S 2]{SerreLF}.
\end{proof}

\begin{lemma}\label{lem:squares-mod-2k-appB}
Let $k\ge 3$. Then $(\Z/2^k)^\times/((\Z/2^k)^\times)^2$ is canonically identified with
$\Z_2^\times/(\Z_2^\times)^2$, and the square class of a unit is determined by its reduction modulo $8$.
\end{lemma}

\begin{proof}
For $k\ge 3$ the reduction map $(\Z/2^k)^\times\to (\Z/8)^\times$ is surjective with kernel a pro-$2$ group
generated by squares. Hence the square class is determined mod $8$ and stabilizes as $k$ increases.
\end{proof}

\begin{remark}\label{rem:det-refine-welldefined-appB}
The definition \eqref{eq:detk-def-appB} is independent of the choices of orthonormal basis and lifts.
Indeed, changing lifts alters $C$ by adding an integral symmetric matrix with even diagonal, which changes
$\det(C)$ by a square in $\Z_2^\times$. Changing orthonormal basis acts by congruence with an element of
$\GL_n(\F_2)$, which again changes $\det(C)$ by a square.
\end{remark}

\subsection{The quadratic correction $q_k$}

For Type $E$ layers, the alternating form $b_k$ does not determine the $2$--primary linking pairing.
The missing information is encoded by a quadratic correction.

Let $H_k(G):=G/G_k$ (the quotient by elements of order dividing $2^k$).
Note that $H_k(G)$ is again a finite $2$--group.

\begin{lemma}\label{lem:q-linear-appB}
For $x\in G$ the quantity $2^k\lambda(x,x)\bmod 2$ depends only on the image of $x$ in $H_k(G)$ and defines
a homomorphism
\[
q: H_k(G)\longrightarrow \F_2.
\]
Moreover, $q$ vanishes if and only if the $k$th layer is of Type $E$.
\end{lemma}

\begin{proof}
If $x'=x+u$ with $u\in G_k$, then
\[
2^k\lambda(x',x')-2^k\lambda(x,x)=2^{k+1}\lambda(x,u)+2^k\lambda(u,u).
\]
The first term lies in $2\Z$. For the second term, note that $2^{k}\lambda(u,u)\bmod 2$ is $b_k(\bar u,\bar u)$,
and this vanishes for all $\bar u$ if and only if the layer is Type $E$ (Lemma~\ref{lem:alternating-ck-appB}).
Thus in Type $E$ the right-hand side lies in $2\Z$ and the reduction modulo $2$ is independent of the lift.
Linearity follows from the polarization identity
\[
2^k\lambda(x+y,x+y) \equiv 2^k\lambda(x,x)+2^k\lambda(y,y)\pmod 2
\]
when $b_k$ is alternating.
\end{proof}

\begin{proposition}[Well-definedness of $q_k$]\label{prop:qk}
Assume the $k$th layer is Type $E$. Define
\[
q_k : H_k(G)\longrightarrow \mathbb{Q}/\mathbb{Z},\qquad q_k([x]) := 2^{k-1}\lambda(x,x),
\]
where $[x]$ is the image in $H_k(G)$. Then $q_k$ is well-defined, and its polarization satisfies
\[
q_k([x+y]) - q_k([x]) - q_k([y]) = 2^k\lambda(x,y)\in \mathbb{Q}/\mathbb{Z}.
\]
\end{proposition}

\begin{proof}
If $x'=x+u$ with $u\in G_k$, then
\[
2^{k-1}\lambda(x',x')-2^{k-1}\lambda(x,x)=2^k\lambda(x,u)+2^{k-1}\lambda(u,u).
\]
Since $u$ has order dividing $2^k$, we have $2^k\lambda(x,u)=\lambda(x,2^k u)=0$ in $\Q/\Z$.
Also, $2^{k-1}\lambda(u,u)$ is integral when the layer is Type $E$: indeed
\[
2^k\lambda(u,u)\bmod 2=b_k(\bar u,\bar u)=0
\]
so $2^k\lambda(u,u)\in 2\Z$, hence $2^{k-1}\lambda(u,u)\in \Z$.
Thus $q_k$ is well-defined.

For polarization,
\[
2^{k-1}\lambda(x+y,x+y)-2^{k-1}\lambda(x,x)-2^{k-1}\lambda(y,y)=2^k\lambda(x,y),
\]
which holds in $\Q/\Z$ by bilinearity.
\end{proof}

\subsection{Gauss sums for Type $E$}

Assume the $k$th layer is Type $E$, so that $q_k$ is defined.
Define the Gauss sum
\[
\Gamma_k(G,\lambda)\;:=\;\sum_{x\in H_k(G)} \exp\bigl(2\pi i\, q_k(x)\bigr)\ \in\ \C.
\]
Since $q_k$ is quadratic with polarization induced by $\lambda$, the classical evaluation of quadratic Gauss sums
yields $|\Gamma_k(G,\lambda)|=\sqrt{|H_k(G)|}$ and $\Gamma_k(G,\lambda)/\sqrt{|H_k(G)|}$ is an $8$th root of unity.
We therefore set
\[
\widetilde\Gamma_k(G,\lambda):=\frac{\Gamma_k(G,\lambda)}{\sqrt{|H_k(G)|}}\ \in\ \mu_8,
\qquad
u_k(G,\lambda):=\widetilde\Gamma_k(G,\lambda).
\]
This is the Gauss-sum invariant appearing in the $2$--primary classification of linking forms;
see \cite[\S 2]{KawauchiKojima80} and \cite[\S 2]{Deloup99}.

\subsection{Comparison with the Kawauchi--Kojima invariants}\label{app:two-primary:comparison}

The invariants introduced above are a repackaging of the classical invariants used in the
classification of $2$--primary linking pairings.

\begin{proposition}[Agreement with the classical invariants]\label{prop:kk-comparison-appB}
Let $(G,\lambda)$ be a finite $2$--primary linking pairing.
\begin{enumerate}[label=\textup{(\roman*)},leftmargin=2.2em]
\item For each $k\ge 1$, the isometry class of the layer pairing $(P_k(G),b_k)$ agrees with the
mod-$2$ reduction of the $k$th homogeneous constituent in the sense of Kawauchi--Kojima, and hence
determines the corresponding Witt class invariant.
\item For each Type~$A$ layer, the determinant refinement $\det_k(G,\lambda)\in\{\pm1,\pm5\}$
recovers the $2$--adic unit square class (the $\Delta_k$--invariant) of the corresponding homogeneous form.
\item For each Type~$E$ layer, the normalized Gauss sum $u_k(G,\lambda)\in\mu_8$
agrees with the Gauss invariant used in the classification.
\end{enumerate}
Consequently, the collection of invariants $\{(P_k(G),b_k),\det_k,u_k\}_{k\ge 1}$ is a complete invariant
of the isometry class of $(G,\lambda)$.
\end{proposition}

\begin{proof}
Items~(i) and~(ii) follow from the comparison between the layer filtration used here and the invariants
of Kawauchi--Kojima; see \cite[\S 2]{KawauchiKojima80} and \cite[Cor.\ 2.1]{KawauchiKojima80}.
For~(iii), the formulation of the remaining $2$--primary ambiguity in terms of quadratic corrections and
normalized Gauss sums is explained in \cite[\S 2]{Deloup99}.
The final assertion follows since the cited classification results are complete.
\end{proof}

\begin{remark}[Effectivity]\label{rem:two-primary-effective-appB}
Given a presentation of $\lambda$ by a symmetric integer matrix (equivalently, a surgery matrix $A$),
one may compute the layer filtration $G_k$, choose Smith bases and explicit lifts, and evaluate
$\lambda$, $b_k$, and $q_k$ on these lifts to obtain the invariants above.
We record this only to indicate that the canonical invariants $\Canon(A)$ are effectively computable
from a presentation; no auxiliary choices beyond those made explicit in the definitions are used elsewhere.
\end{remark}

\section{Worked examples and consistency checks}\label{app:hand-golden}

\begin{table}[t]
\centering
\small
\renewcommand{\arraystretch}{1.15}
\begin{tabular}{|c|c|}
\hline
$A$ & $\Tok(A)$\\
\hline
$(1)$ &
$b_1=0$\\
\hline
$(3)$ &
$b_1=0;\ p=3,\ k=1,\ n=1,\ x=1$\\
\hline
$\begin{pmatrix}0&3\\[2pt]3&0\end{pmatrix}$ &
$b_1=0;\ p=3,\ k=1,\ n=2,\ x=-1$\\
\hline
$\begin{pmatrix}0&0\\[2pt]0&0\end{pmatrix}$ &
$b_1=2$\\
\hline
$\begin{pmatrix}2&0\\[2pt]0&2\end{pmatrix}$ &
$b_1=0;\ p=2,\ k=1,\ n=2,\ \mathrm{type}=A,\ \det=1$\\
\hline
\end{tabular}
\caption{A small suite of worked examples used as consistency checks for Step~A.}
\label{tab:golden}
\end{table}

This appendix gives several worked examples verifying the Step~A output $\Tok(A)$ by hand.
The same examples are used as regression tests for the accompanying implementation, but the computations
recorded here are independent of any software stack.
They also serve as a concrete illustration of how the isometry class of the torsion linking form
$(G(A),\lambda_A)$ is converted into the canonical token package $\Tok(A)$ using the local classification
results of Wall and Kawauchi--Kojima (and the standard refinements for $2$--groups).
See \cite{Wall63,KawauchiKojima80,Deloup05}.

\subsection{General recipe: from a matrix to tokens}

Let $A\in M_n(\Z)$ be symmetric with $\det(A)\neq 0$.
Recall from \eqref{eq:def-linking-from-A} that
\[
G(A):=\coker(A)=\Z^n/A\Z^n,
\qquad
\lambda_A(\bar x,\bar y)=x^{\mathsf T}A^{-1}y \bmod \Z,
\]
where $A^{-1}$ is computed over $\Q$, and $x,y\in\Z^n$ are representatives.
The Smith normal form of $A$ determines the abstract group structure of $G(A)$ and, in particular,
the $p$--primary decomposition.
When $G(A)$ is a $p$--group, the isometry class of $\lambda_A$ is determined by local invariants:
for odd $p$ one may reduce to a direct sum of homogeneous layers and record, for each exponent $p^k$,
the dimension $n_{p,k}$ and discriminant square class $x_{p,k}\in\{\pm 1\}$ of the induced
nondegenerate symmetric form over $\F_p$ \cite{Wall63,KawauchiKojima80}.
For $p=2$, the classification requires refinements distinguishing the two fundamental Witt types
(Type~$A$ versus Type~$E$ in our notation), and, in the even case, a Gauss sum invariant valued in
eighth roots of unity (encoded by an integer $u\in\Z/8$ in our schema) \cite{KawauchiKojima80,Deloup05}.
The canonical token list $\Tok(A)$ is precisely a deterministic packaging of these local invariants,
sorted by prime and exponent.

In what follows we compute $(G(A),\lambda_A)$ and the corresponding local invariants explicitly
for each matrix in Table~\ref{tab:golden}.

\subsection{The trivial example}

\begin{proposition}\label{prop:golden-trivial}
For $A=[1]$, the group $G(A)$ is trivial and $\Tok(A)$ contains no prime blocks.
\end{proposition}

\begin{proof}
The map $\Z\xrightarrow{\cdot 1}\Z$ is an isomorphism, hence $\coker(A)=0$.
The pairing $\lambda_A$ is therefore vacuous, and $\Tok(A)$ records only $b_1(A)=0$ and no prime data.
\end{proof}

\subsection{Odd prime examples}

\begin{proposition}\label{prop:golden-A3}
Let $A=[3]$. Then $G(A)\cong \Z/3$ with generator $g=\overline{1}$, and
\[
\lambda_A(g,g)\equiv \frac{1}{3}\pmod{\Z}.
\]
Accordingly, the only odd--primary layer is $(p,k)=(3,1)$ with invariants
$(n_{3,1},x_{3,1})=(1,+1)$, as recorded in Table~\ref{tab:golden}.
\end{proposition}

\begin{proof}
We have $\coker([3])\cong\Z/3$, generated by $g=\overline{1}$.
Since $A^{-1}=[1/3]$, \eqref{eq:def-linking-from-A} gives
\[
\lambda_A(g,g)=1\cdot \frac13\cdot 1 \equiv \frac13 \pmod{\Z}.
\]
For odd $p=3$ and exponent $3^1$, the $k=1$ layer is $V_1\cong \F_3$ (dimension $n_{3,1}=1$).
Multiplying by $p^k=3$ identifies the induced form $b_1$ on $V_1$ with the $1\times 1$ matrix $[1]$ over $\F_3$,
so $\det(b_1)=1$ and hence $x_{3,1}=\chi(1)=+1$.
\end{proof}

\begin{proposition}\label{prop:golden-Hopflike}
Let
\[
A=\begin{pmatrix}0&3\\3&0\end{pmatrix}.
\]
Then $G(A)\cong (\Z/3)^2$ and, in the basis $g_i=\overline{e_i}$, the pairing is represented by
\[
\bigl(\lambda_A(g_i,g_j)\bigr)=\frac{1}{3}\begin{pmatrix}0&1\\1&0\end{pmatrix}\in M_2(\Q/\Z).
\]
The odd--primary layer invariants at $(p,k)=(3,1)$ are $(n_{3,1},x_{3,1})=(2,-1)$.
\end{proposition}

\begin{proof}
We first compute the group.
The determinant is $\det(A)=-9$, so $|G(A)|=9$.
Since $\gcd$ of all entries of $A$ is $3$, the first Smith invariant is $d_1=3$.
For a $2\times 2$ matrix, the product of the Smith invariants equals $|\det(A)|$, hence $d_2=|\det(A)|/d_1=3$.
Thus $G(A)\cong \Z/3\oplus \Z/3$.

Next compute the pairing.
A direct multiplication shows
\[
A^{-1}=\frac{1}{3}\begin{pmatrix}0&1\\1&0\end{pmatrix},
\qquad
A\cdot A^{-1}=
\begin{pmatrix}0&3\\3&0\end{pmatrix}
\cdot
\frac{1}{3}\begin{pmatrix}0&1\\1&0\end{pmatrix}
=
\begin{pmatrix}1&0\\0&1\end{pmatrix}.
\]
Therefore, in the basis $g_i=\overline{e_i}$ we have
\[
\lambda_A(g_i,g_j)\equiv (A^{-1})_{ij}\equiv \frac{1}{3}\begin{pmatrix}0&1\\1&0\end{pmatrix}_{ij}
\pmod{\Z}.
\]

Finally compute the discriminant class.
Since the group is elementary abelian of exponent $3$, the entire $3$--primary part lies in the $k=1$ layer.
Multiplying the pairing by $3$ yields the induced $\F_3$--valued form
\[
b_1 = 3\cdot \bigl(\lambda_A(g_i,g_j)\bigr)\equiv
\begin{pmatrix}0&1\\1&0\end{pmatrix}\in M_2(\F_3).
\]
Its determinant is $\det(b_1)=-1\equiv 2\in\F_3^\times$.
Since $2$ is not a square in $\F_3^\times=\{1,2\}$, we have $x_{3,1}=\chi(2)=-1$, as claimed.
\end{proof}

\subsection{Cyclic $2$--primary examples}

\begin{proposition}\label{prop:golden-Z4plus}
Let $A=[4]$. Then $G(A)\cong \Z/4$ with generator $g=\overline{1}$ and
\[
\lambda_A(g,g)\equiv \frac{1}{4}\pmod{\Z}.
\]
This contributes a single Type~$A$ $2$--primary layer at $k=2$ with the odd numerator $a\equiv 1\pmod 4$.
\end{proposition}

\begin{proof}
We have $G(A)=\coker([4])\cong\Z/4$, generated by $g=\overline{1}$.
Since $A^{-1}=[1/4]$,
\[
\lambda_A(g,g)=1\cdot \frac14\cdot 1\equiv \frac14\pmod{\Z}.
\]
The group is cyclic of exponent $2^2$, so it contributes to the single $k=2$ layer with $n=1$.
For a cyclic $2$--group, the isometry class is determined by the odd numerator $a$ in the expression
$\langle a/2^k\rangle$ \cite{Wall63,KawauchiKojima80}; for $k=2$ this gives the two classes
$\langle 1/4\rangle$ and $\langle 3/4\rangle$.
Here $a=1$, as asserted.
\end{proof}

\begin{proposition}\label{prop:golden-Z4minus}
Let $A=[-4]$. Then $G(A)\cong \Z/4$ with generator $g=\overline{1}$ and
\[
\lambda_A(g,g)\equiv -\frac{1}{4}\equiv \frac{3}{4}\pmod{\Z}.
\]
This contributes a single Type~$A$ $2$--primary layer at $k=2$ with odd numerator $a\equiv 3\pmod 4$.
\end{proposition}

\begin{proof}
As above, $G(A)\cong\Z/4$ with generator $g=\overline{1}$.
Now $A^{-1}=[-1/4]$, hence
\[
\lambda_A(g,g)\equiv -\frac14\equiv \frac34\pmod{\Z}.
\]
This is the cyclic form $\langle 3/4\rangle$, so $a=3$.
\end{proof}

\begin{proposition}\label{prop:golden-Z8plus}
Let $A=[8]$. Then $G(A)\cong \Z/8$ with generator $g=\overline{1}$ and
\[
\lambda_A(g,g)\equiv \frac{1}{8}\pmod{\Z}.
\]
This contributes a single Type~$A$ $2$--primary layer at $k=3$ with odd numerator $a\equiv 1\pmod 8$.
\end{proposition}

\begin{proof}
We have $G(A)\cong\Z/8$ and $A^{-1}=[1/8]$, so $\lambda_A(g,g)\equiv 1/8$.
In our refinement, cyclic $2$--primary layers are recorded as Type~$A$ with the odd numerator $a\in(\Z/2^k)^\times$;
at $k=3$ the unit group is $\{1,3,5,7\}$ and each value gives a distinct isometry class since odd squares are $\equiv 1\bmod 8$
\cite{Wall63,KawauchiKojima80}.
Thus $a=1$.
\end{proof}

\subsection{The hyperbolic $2$--primary example and the Gauss sum}

\begin{proposition}\label{prop:golden-hyp2}
Let
\[
A=\begin{pmatrix}0&2\\2&0\end{pmatrix}.
\]
Then $G(A)\cong (\Z/2)^2$ and, in the basis $g_i=\overline{e_i}$, the pairing is represented by
\[
\bigl(\lambda_A(g_i,g_j)\bigr)=\frac{1}{2}\begin{pmatrix}0&1\\1&0\end{pmatrix}.
\]
This pairing is even (Type~$E$ in our schema), and the normalized Gauss sum equals $1$, hence $u\equiv 0\pmod 8$.
\end{proposition}

\begin{proof}
The determinant is $\det(A)=-4$, hence $|G(A)|=4$.
The $\gcd$ of all entries of $A$ is $2$, so the first Smith invariant is $d_1=2$ and the second is
$d_2=|\det(A)|/d_1=2$. Thus $G(A)\cong \Z/2\oplus\Z/2$.

We next compute the inverse:
\[
A^{-1}=\frac{1}{2}\begin{pmatrix}0&1\\1&0\end{pmatrix},
\qquad
A\cdot A^{-1}=
\begin{pmatrix}0&2\\2&0\end{pmatrix}
\cdot
\frac{1}{2}\begin{pmatrix}0&1\\1&0\end{pmatrix}
=
\begin{pmatrix}1&0\\0&1\end{pmatrix}.
\]
Therefore, in the generators $g_i=\overline{e_i}$ we obtain
\[
\lambda_A(g_1,g_1)=0,\quad
\lambda_A(g_2,g_2)=0,\quad
\lambda_A(g_1,g_2)=\lambda_A(g_2,g_1)=\frac12\quad(\bmod\ \Z).
\]
In particular, $\lambda_A(x,x)=0$ for all $x\in G(A)$, so the form is alternating on the exponent-$2$ layer.
This is the defining Type~$E$ behavior in the $k=1$ layer of the $2$--primary classification \cite{KawauchiKojima80,Deloup05}.

To compute the Gauss sum invariant, choose the quadratic refinement
$q\colon G(A)\to \Q/\Z$ defined on the four elements by
\[
q(0)=0,\qquad q(g_1)=0,\qquad q(g_2)=0,\qquad q(g_1+g_2)=\frac12.
\]
A direct check shows that for all $x,y\in G(A)$,
\[
q(x+y)-q(x)-q(y)\equiv \lambda_A(x,y)\pmod{\Z}.
\]
The normalized Gauss sum is
\[
\Gamma(q):=\frac{1}{|G(A)|^{1/2}}\sum_{x\in G(A)}\exp\bigl(2\pi i\, q(x)\bigr)
=\frac{1}{2}\bigl(1+1+1+(-1)\bigr)=1,
\]
so $u\equiv 0\pmod 8$.
\end{proof}

\subsection{A mixed-prime example}

\begin{proposition}\label{prop:golden-mixed}
Let $A=\mathrm{diag}(4,3)$. Then $G(A)\cong \Z/12$ and the linking form splits orthogonally
into its $2$--primary and $3$--primary parts, namely $\langle 1/4\rangle\oplus \langle 1/3\rangle$.
Accordingly, $\Tok(A)$ consists of a single Type~$A$ $2$--primary block at $k=2$ with $a=1$ and a single
odd--primary block at $(p,k)=(3,1)$ with $(n_{3,1},x_{3,1})=(1,+1)$, as in Table~\ref{tab:golden}.
\end{proposition}

\begin{proof}
Since $A$ is diagonal, we have
\[
G(A)=\coker(\mathrm{diag}(4,3))\cong \Z/4\oplus \Z/3.
\]
Because $\gcd(4,3)=1$, this group is cyclic of order $12$.

The inverse matrix is
\[
A^{-1}=\mathrm{diag}\Bigl(\frac14,\frac13\Bigr).
\]
Let $g_2=\overline{e_1}$ (order $4$) and $g_3=\overline{e_2}$ (order $3$) in $G(A)$.
Then
\[
\lambda_A(g_2,g_2)\equiv \frac14,\qquad
\lambda_A(g_3,g_3)\equiv \frac13,\qquad
\lambda_A(g_2,g_3)=0.
\]
Thus $(G(A),\lambda_A)$ is the orthogonal direct sum of the cyclic $2$--primary form $\langle 1/4\rangle$
and the cyclic $3$--primary form $\langle 1/3\rangle$.
The asserted token data follow from Proposition~\ref{prop:golden-A3} and Proposition~\ref{prop:golden-Z4plus}.
\end{proof}

\subsection{A complementary $2$--primary rank-two example}

Although it is not listed in Table~\ref{tab:golden}, we record here the diagonal companion to
Proposition~\ref{prop:golden-hyp2}, since it is the minimal example distinguishing Type~$A$ from Type~$E$
on $(\Z/2)^2$ in the $k=1$ layer.

\begin{proposition}\label{prop:diag22}
Let $A=\mathrm{diag}(2,2)$. Then $G(A)\cong (\Z/2)^2$ and, in the basis $g_i=\overline{e_i}$,
\[
\bigl(\lambda_A(g_i,g_j)\bigr)=\frac{1}{2}\begin{pmatrix}1&0\\0&1\end{pmatrix}.
\]
This pairing is Type~$A$ (odd) in the $k=1$ layer.
\end{proposition}

\begin{proof}
We have $\det(A)=4$ and clearly $\coker(A)\cong \Z/2\oplus \Z/2$.
Moreover $A^{-1}=\mathrm{diag}(1/2,1/2)$, hence
\[
\lambda_A(g_1,g_1)\equiv \frac12,\qquad \lambda_A(g_2,g_2)\equiv \frac12,\qquad \lambda_A(g_1,g_2)=0.
\]
In particular, $\lambda_A(g_1,g_1)\neq 0$, so the form is not alternating on the exponent-$2$ layer.
This is exactly the Type~$A$ behavior in the $2$--primary classification at $k=1$
\cite{KawauchiKojima80,Deloup05}.
\end{proof}

\section{Dictionary from tokens to matrices}
\label{app:dictionary}

This appendix records concrete symmetric integral matrices used in Step~B
to pass from a (local) token description of a torsion linking pairing to an
explicit matrix representative.
The only input from topology is the standard fact that if $B$ is a symmetric
integral matrix with $\det(B)\neq 0$, then the torsion group
$G=\coker(B)$ carries a nondegenerate symmetric linking pairing
\begin{equation}
\label{eq:linking-from-B-appD}
\lambda_B([x],[y]) \;=\; x^{\mathsf T} B^{-1} y \in \Q/\Z ,
\end{equation}
where $x,y\in \Z^n$ lift the classes in $\coker(B)$ (cf.\ \eqref{eq:def-linking-from-A}).

\subsection{Cyclic blocks via Hirzebruch--Jung continued fractions}

Let $m>q>0$ be coprime integers.
Write the (negative) Hirzebruch--Jung continued fraction expansion
\[
\frac{m}{q} \;=\; [a_1,\dots,a_r]
\;:=\;
a_1-\cfrac{1}{a_2-\cfrac{1}{\ddots-\cfrac{1}{a_r}}},
\qquad a_i\ge 2.
\]
(For instance, one may compute $a_1=\lceil m/q\rceil$ and iterate the Euclidean
algorithm.)

Define the associated \emph{linear plumbing matrix}
\[
C(m,q)\;:=\;
\begin{pmatrix}
a_1 & 1   &        &        & 0\\
1   & a_2 & 1      &        &  \\
    & 1   & \ddots & \ddots &  \\
    &     & \ddots & a_{r-1}& 1\\
0   &     &        & 1      & a_r
\end{pmatrix}.
\]

\begin{lemma}[Cyclic presentation matrix]
\label{lem:cf-cyclic}
With notation as above:
\begin{enumerate}
\item $\det C(m,q)=m$, hence $\coker(C(m,q))\cong \Z/m$.
\item If $e_1\in \Z^r$ denotes the first standard basis vector, then
\[
\lambda_{C(m,q)}([e_1],[e_1]) \;=\; (C(m,q)^{-1})_{11} \;=\; \frac{q}{m}\in \Q/\Z .
\]
In particular, $C(m,q)$ presents the cyclic pairing $\langle q/m\rangle$
on $\Z/m$.
\end{enumerate}
\end{lemma}

\begin{proof}
The determinant identity $\det C(m,q)=m$ is standard for linear plumbings and
may be proved by induction on $r$ using the recurrence for tridiagonal
determinants; the continued fraction identity ensures the determinant equals
$m$.

For the second claim, solve $C(m,q)\,v=e_1$ over $\Q$.  By Cramer's rule,
$v_1=(\det C_{11})/(\det C)$, where $C_{11}$ is the minor obtained by deleting
row~1 and column~1.  The same determinant recursion shows $\det C_{11}=q$,
so $(C^{-1})_{11}=v_1=q/m$.  Equation~\eqref{eq:linking-from-B-appD} then gives
$\lambda_{C(m,q)}([e_1],[e_1])=q/m$.
\end{proof}

\subsection{Odd primes: Wall generators}

Fix an odd prime $p$ and a choice of nonsquare class $\epsilon_p\in (\F_p^\times)/(\F_p^\times)^2$,
represented by an integer $\epsilon_p\in\{2,\dots,p-1\}$.

For $k\ge 1$ we define matrices representing the two cyclic linking forms of
order $p^k$:
\[
B\bigl(\mathsf A(p^k,+)\bigr)\;:=\;(p^k),
\qquad
B\bigl(\mathsf A(p^k,-)\bigr)\;:=\;C(p^k,\epsilon_p).
\]
By Lemma~\ref{lem:cf-cyclic}, $\coker(B(\mathsf A(p^k,+)))\cong \Z/p^k$ with
$\lambda(1,1)=1/p^k$, while $\coker(B(\mathsf A(p^k,-)))\cong \Z/p^k$ with
$\lambda(1,1)=\epsilon_p/p^k$, which realizes the nonsquare class.
These are the two Wall generators for odd primes~\cite{Wall63}.

\subsection{The prime $2$: Miranda generators}

For $k\ge 1$ and an odd integer $q$ (viewed modulo $2^k$), let
\[
B\bigl(\mathsf A(2^k,q)\bigr)\;:=\;C(2^k,q),
\]
so that $\coker(B(\mathsf A(2^k,q)))\cong \Z/2^k$ and
$\lambda(1,1)=q/2^k$ by Lemma~\ref{lem:cf-cyclic}.
For $k\ge 3$ the four square classes correspond to $q\equiv 1,3,5,7\pmod 8$
(cf.\ \cite{Miranda84}).

The even (type~E) $2$-primary generators are represented by the following
$2\times 2$ matrices (cf.\ \cite[\S1]{Miranda84}):
\[
B\bigl(\mathsf E(2^k)\bigr)
:=
\begin{pmatrix}
0 & 2^k\\
2^k & 0
\end{pmatrix},
\qquad
B\bigl(\mathsf F(2^k)\bigr)
:=
\begin{pmatrix}
0 & 2^k\\
2^k & 2^{k+1}
\end{pmatrix}.
\]
Both have determinant $-2^{2k}$, hence present $2$-primary groups.
Their inverses are
\[
B(\mathsf E(2^k))^{-1}=\frac{1}{2^k}
\begin{pmatrix}0&1\\[2pt]1&0\end{pmatrix},
\qquad
B(\mathsf F(2^k))^{-1}=\frac{1}{2^k}
\begin{pmatrix}-2&1\\[2pt]1&0\end{pmatrix},
\]
so the induced exponent-one layer forms are alternating (type~E).
Moreover, Miranda shows that the corresponding Gauss invariants satisfy
$u_{k-1}(\mathsf E(2^k))=0$ and $u_{k-1}(\mathsf F(2^k))=4$ for $k\ge 2$
\cite[Lemma~2.3]{Miranda84}.

\subsection{Assembly and shifts}

Given a token package $T$, we form $B(T)$ by block-summing the matrices attached
to its local tokens (odd primes via $\mathsf A(p^k,\pm)$, the $2$-primary part via
$\mathsf A(2^k,q)$, $\mathsf E(2^k)$ and $\mathsf F(2^k)$), and finally adjoining
a zero block of size $b_1(T)$ to encode free rank:
\[
B(T)\;:=\;\bigoplus_{\tau\in T} B(\tau)\ \ \oplus\ \ (0)^{\oplus b_1(T)}.
\]
By construction $\coker(B(T))_{\mathrm{tors}}$ with its pairing \eqref{eq:linking-from-B-appD}
is isometric to the discriminant pairing encoded by $T$.

For normalization we use the grading shift
\[
\mathrm{sh}(B)\;:=\;3\,b_+(B)-2\,b_-(B),
\]
where $b_\pm(B)$ are the numbers of positive/negative eigenvalues of $B\otimes \R$.

\begin{lemma}[Stabilization invariance of the normalized module]
\label{lem:shift-equivalences}
Let $B$ be a symmetric integral matrix.
Then there are equivalences
\[
\mathcal{L}(B\oplus (1))\ \simeq\ \mathcal{L}(B),
\qquad
\mathcal{L}(B\oplus (-1))\ \simeq\ \mathcal{L}(B)
\]
in $\Ho(\Mod_{\TMF})$.
\end{lemma}

\begin{proof}
By definition $\mathcal{L}(B)=L_B[\mathrm{sh}(B)]$.
GKMP show that the unnormalized modules satisfy canonical equivalences
\(L_{B\oplus(1)}\simeq L_B[-3]\) and \(L_{B\oplus(-1)}\simeq L_B[2]\)
(see \cite[Lemma~5.7]{GKMP25}).
On the other hand, stabilization by $(1)$ increases $b_+$ by $1$ and leaves $b_-$ unchanged, so
$\mathrm{sh}(B\oplus(1))=\mathrm{sh}(B)+3$.
Similarly $\mathrm{sh}(B\oplus(-1))=\mathrm{sh}(B)-2$.
Combining these,
\[
\mathcal{L}(B\oplus(1))=L_{B\oplus(1)}[\mathrm{sh}(B)+3]\ \simeq\ L_B[-3][\mathrm{sh}(B)+3]=\mathcal{L}(B)
\]
and the same computation gives $\mathcal{L}(B\oplus(-1))\simeq\mathcal{L}(B)$.
\end{proof}

\begin{proposition}[Correctness of the implementation]
\label{prop:appendD-implementation}
Let $A$ be a symmetric integral matrix.  Then $A$ is stably congruent to
$B(\Tok(A))$ (with a zero block of size $b_1(A)$ appended if $A$ is singular).
Consequently,
\[
\mathcal{R}(\Tok(A))\ \simeq\ \mathcal{L}(A).
\]
\end{proposition}

\begin{proof}
By construction, $\Tok(A)$ records $b_1(A)$ and the isometry class of the torsion
linking pairing $\Disc(A)$.
The assembled matrix $B(\Tok(A))$ has the same free rank and presents a torsion
linking pairing isometric to $\Disc(A)$.
The stable classification theorem (Proposition~\ref{prop:stable-classification})
therefore gives stable congruence between $A$ and $B(\Tok(A))$.
Applying Lemma~\ref{lem:shift-equivalences} yields the claimed equivalence of
the associated stabilized \TMF-modules.
\end{proof}

\end{document}